\documentclass[a4paper,11pt]{article}
\usepackage{stmaryrd}
\usepackage{graphicx}
\usepackage{amsfonts}
\usepackage{amsmath,amsthm}
\usepackage{amssymb}
\usepackage{mathrsfs}
\usepackage{indentfirst}
\usepackage{titlesec}
\usepackage{caption2}
\usepackage{color}
\usepackage[normalem]{ulem}

\usepackage[english]{babel}
\usepackage{cite}
\usepackage{bbm}

\newtheorem{theorem}{Theorem}[section]
\newtheorem{definition}{Definition}[section]

\newtheorem{remark}[theorem]{Remark}
\newtheorem{lemma}[theorem]{Lemma}

\providecommand{\1}{\mathbf 1}

\begin{document}

\title{Large deviations and averaging for systems of slow-fast stochastic reaction-diffusion equations}
\author{Wenqing Hu \thanks{Department of Mathematics and Statistics, Missouri University of Science and Technology
(formerly University of Missouri, Rolla). Email: \texttt{huwen@mst.edu}} \ , \
Michael Salins \thanks{Department of Mathematics and Statistics, Boston University. Email: \texttt{msalins@bu.edu}} \ , \
Konstantinos Spiliopoulos \thanks{Department of Mathematics and Statistics, Boston University. Email: \texttt{kspiliop@math.bu.edu}. K.S. was partially supported by NSF DMS 1550918} \
}

\date{\today}

\maketitle

\begin{abstract}
We study a large deviation principle for a system of stochastic reaction--diffusion equations (SRDEs)
with a separation of fast and slow components and small noise in the slow component. The derivation of the large deviation principle is based on the weak convergence method in infinite dimensions, which results in studying averaging for controlled SRDEs. By appropriate choice of the parameters, the fast process and the associated control that arises from the weak convergence method decouple from each other. We show that in this decoupling case one can use the weak convergence method to characterize the limiting process via a ``viable pair" that captures the limiting controlled dynamics and the effective invariant measure simultaneously. The characterization of the limit of the controlled slow-fast processes in terms of viable pair enables us to obtain a variational representation of the large deviation action functional. Due to the infinite-dimensional nature of our set--up, the proof of tightness as well as the analysis of the limit process and in particular the proof of the large deviations lower bound is considerably more delicate here than in the finite-dimensional situation. Smoothness properties of optimal controls in infinite dimensions (a necessary step for the large deviations lower bound) need to be established.  We emphasize that many issues that are present in the infinite dimensional case, are completely absent in finite dimensions.
\end{abstract}



\

\section{Introduction}

Let $\varepsilon>0$ and $\delta=\delta(\varepsilon)>0$. Let $D$ be a smooth bounded domain of $\mathbb{R}^d$, with $d\geq 1$.
In this paper, we study a system of stochastic reaction--diffusion equations with slow--fast dynamics
on the domain $D\subset \mathbb{R}^d$, $d\geq 1$, as follows:
\begin{equation}\label{Eq:FastSlowStochasticRDE}
\left\{\begin{array}{l}
\displaystyle{\dfrac{\partial X^{\varepsilon,\delta}}{\partial t}(t,x)=\mathcal{A}_1 X^{\varepsilon,\delta}(t,x)+b_1(x, X^{\varepsilon,\delta}(t,x), Y^{\varepsilon,\delta}(t,x))}\\
\displaystyle{\hspace{2cm}+\sqrt{\varepsilon}\sigma_1(x, X^{\varepsilon,\delta}(t,x), Y^{\varepsilon,\delta}(t,x))\dfrac{\partial W^{Q_1}}{\partial t}(t,x)} \ ,
\\
\displaystyle{\dfrac{\partial Y^{\varepsilon,\delta}}{\partial t}(t,x)=\dfrac{1}{\delta^{2}}\left[\mathcal{A}_2 Y^{\varepsilon,\delta}(t,x)+
b_2(x, X^{\varepsilon,\delta}(t,x), Y^{\varepsilon,\delta}(t,x))\right]}\ \\
\displaystyle{\hspace{2cm}+\dfrac{1}{\delta}\sigma_2(x, X^{\varepsilon,\delta}(t,x), Y^{\varepsilon,\delta}(t,x))\dfrac{\partial W^{Q_2}}{\partial t}(t,x)} \ ,
\\
X^{\varepsilon,\delta}(0,x)=X_0(x) \ , \ Y^{\varepsilon,\delta}(0,x)=Y_0(x) \ , \ x\in D \ ,
\\
\mathcal{N}_1 X^{\varepsilon,\delta}(t,x)=\mathcal{N}_2 Y^{\varepsilon,\delta}(t,x)=0 \ , \ t\geq 0 \ , \ x\in \partial D \ .
\end{array}\right.
\end{equation}

Here $\varepsilon>0$ is a small parameter and $\delta=\delta(\varepsilon)>0$ is such that $\delta \rightarrow 0$ as $\varepsilon\downarrow 0$.
The operators $\mathcal{A}_1$ and $\mathcal{A}_2$ are two strictly elliptic operators and $b_i(x,X,Y)$, $i=1,2$
are the nonlinear terms.
The noise processes $W^{Q_1}$ and $W^{Q_2}$ are two cylindrical Wiener processes with covariance matrices $Q_1^2$
and $Q_2^2$, and $\sigma_i(x,X,Y)$, $i=1,2$ are functions multiplied by the noises. The initial values $X_0$
and $Y_0$ are assumed to be in $L^2(D)$.
The boundary conditions are given by operators $\mathcal{N}_i$, $i=1,2$ which may correspond to either Dirichlet or Neumann conditions.

Since $\delta>0$ is  small as $\varepsilon>0$ gets small, one can think of the solution $X^{\varepsilon,\delta}(t,x)$
in (\ref{Eq:FastSlowStochasticRDE})
as the ``slow" process (or slow motion) and the solution $Y^{\varepsilon,\delta}(t,x)$ in (\ref{Eq:FastSlowStochasticRDE})
as the ``fast" process (or fast motion). Notice that the noise term $\sqrt{\varepsilon}\sigma_1(x, X^{\varepsilon,\delta}(t,x), Y^{\varepsilon,\delta}(t,x))\dfrac{\partial W^{Q_1}}{\partial t}(t,x)$
in the equation for the slow process $X^{\varepsilon,\delta}(t,x)$ has a small parameter $\sqrt{\varepsilon}$,
and both the deterministic as well as the noise term in the equation for the fast process $Y^{\varepsilon,\delta}(t,x)$
contain large parameters $\dfrac{1}{\delta^2}$ and $\dfrac{1}{\delta}$.
So that in the limit, we expect an interplay between an averaging effect in the fast process $Y^{\varepsilon,\delta}(t,x)$
and the effect of the diminishing noise in the slow process $X^{\varepsilon,\delta}(t,x)$. In \cite{CerraiRDEAveraging1}, Cerrai demonstrated that as $\varepsilon \to 0$, the slow motion $X^{\varepsilon,\delta}$ converges to a limit derived from averaging the fast motion over its invariant measure. In this work, we will study the large deviations principle (LDP) of $X^{\varepsilon,\delta}$. The characterization of such an interplay between large deviations
and averaging principle at the level of mathematical rigor requires delicate analysis of the asymptotic regimes
while taking the limits $\varepsilon\rightarrow 0$ and $\delta \rightarrow 0$.

Our goal is to derive a large deviation principle for the process $X^{\varepsilon,\delta}(t,x)$ as $\varepsilon\rightarrow 0$ and thus $\delta=\delta(\varepsilon)\rightarrow 0$.
We will be studying a particular regime that roughly speaking, says that
$\delta$ goes to $0$ much faster than $\sqrt{\varepsilon}$ (i.e., $\dfrac{\delta}{\sqrt{\varepsilon}}\rightarrow 0$ as $\varepsilon \downarrow 0$, for details, see Section \ref{S:MainResults}).
The analysis of the problem in other asymptotic regimes is left for future work and briefly discussed in Section \ref{S:Generalizations}.

One of the most effective methods in analyzing large deviation effects is the weak convergence method \cite{LDPInfiniteSDE,VariationalInfniteBM},
which is the method we are using in this paper. Roughly speaking,
by a variational representation (see \cite{VariationalInfniteBM}) of exponential
functionals of Wiener processes, one can represent the exponential functional of the slow process $X^{\varepsilon,\delta}(t,x)$ that
appears in the Laplace principle (which is equivalent to
large deviations principle) as a variational infimum over a family of \textit{controlled} slow processes $X^{\varepsilon,\delta,u}(t,x)$.
In particular, we have for any bounded continuous function $h: C([0,T]; L^2(D))\rightarrow \mathbb{R}$ that
\begin{equation}\label{Introduction:LDPRepresentationFastSlowSRDE}
-\varepsilon\ln \mathbf{E}\left[\exp\left(-\dfrac{1}{\varepsilon}h(X^{\varepsilon,\delta})\right)\right]
=
\displaystyle{\inf\limits_{u \in L^{2}([0,T]; U)}\mathbf{E}\left[\dfrac{1}{2}\int_0^T |u(s)|_U^2ds
+h(X^{\varepsilon,\delta,u})\right]} \ .
\end{equation}

Here the Hilbert space $U$ is called the control space, and the infimum is over all controls $u\in L^2([0,T];U)$ with finite $L^2([0,T]; U)$--norm.
The controlled slow motion $X^{\varepsilon,\delta,u}$
that appears on the right hand side of (\ref{Introduction:LDPRepresentationFastSlowSRDE}) comes from a controlled slow--fast
system $(X^{\varepsilon,\delta,u}, Y^{\varepsilon,\delta,u})$ of reaction--diffusion equations corresponding to (\ref{Eq:FastSlowStochasticRDE}):
\begin{equation}\label{Introduction:FastSlowStochasticRDEWithControl}
\left\{\begin{array}{l}
\displaystyle{\dfrac{\partial X^{\varepsilon,\delta,u}}{\partial t}(t,x)}=\displaystyle{\mathcal{A}_1 X^{\varepsilon,\delta,u}(t,x)+b_1(x, X^{\varepsilon,\delta,u}(t,x), Y^{\varepsilon,\delta,u}(t,x))}
\\
 \qquad \qquad \qquad \qquad \displaystyle{+\sigma_1(x, X^{\varepsilon,\delta,u}(t,x), Y^{\varepsilon,\delta,u}(t,x))(Q_1u(t))(x)}
\\
 \qquad \qquad \qquad \qquad \displaystyle{+\sqrt{\varepsilon}\sigma_1(x, X^{\varepsilon,\delta,u}(t,x), Y^{\varepsilon,\delta,u}(t,x))\dfrac{\partial W^{Q_1}}{\partial t}(t,x)} \ ,
\\
\dfrac{\partial Y^{\varepsilon,\delta,u}}{\partial t}(t,x) =\displaystyle{\dfrac{1}{\delta^{2}}\left[\mathcal{A}_2 Y^{\varepsilon,\delta,u}(t,x)+
b_2(x, X^{\varepsilon,\delta,u}(t,x), Y^{\varepsilon,\delta,u}(t,x))\right]}
\\
 \qquad \qquad \qquad \qquad \displaystyle{+\dfrac{1}{\delta\sqrt{\varepsilon}} \sigma_2(x, X^{\varepsilon,\delta,u}(t,x), Y^{\varepsilon,\delta,u}(t,x))(Q_2u(t))(x)}
\\
\qquad \qquad \qquad \qquad \displaystyle{+\dfrac{1}{\delta}\sigma_2(x, X^{\varepsilon,\delta,u}(t,x), Y^{\varepsilon,\delta,u}(t,x))\dfrac{\partial W^{Q_2}}{\partial t}(t,x)} \ ,
\\
X^{\varepsilon,\delta,u}(0,x)=X_0(x) \   , \ Y^{\varepsilon,\delta,u}(0,x)=Y_0(x) \ , \ x\in D \ ,
\\
\mathcal{N}_1 X^{\varepsilon,\delta,u}(t,x)=\mathcal{N}_2 Y^{\varepsilon,\delta,u}(t,x)=0 \  , \ t\geq 0 \ , \ x\in \partial D \ .
\end{array}\right.
\end{equation}

In view of (\ref{Introduction:LDPRepresentationFastSlowSRDE}) and (\ref{Introduction:FastSlowStochasticRDEWithControl}), we see that
in order to obtain a limit as $\varepsilon\downarrow 0$ of the Laplace functional $-\varepsilon\ln \mathbf{E}\left[\exp\left(-\dfrac{1}{\varepsilon}h(X^{\varepsilon,\delta})\right)\right]$
(i.e., to prove a large deviations principle),
one needs to analyze the limit as $\varepsilon\downarrow 0$ (and thus $\delta \rightarrow 0$) of the \textit{controlled} slow--fast system $(X^{\varepsilon,\delta,u}, Y^{\varepsilon,\delta,u})$
in (\ref{Introduction:FastSlowStochasticRDEWithControl}). This is the first technical part of the current work (Section \ref{S:LLN_ControlledSRDE}).
In fact, to analyze system (\ref{Introduction:FastSlowStochasticRDEWithControl}), the first difficulty
is in that the control term $u$ is only known to be square integrable. 
This makes the proof of tightness much more involved and due to the infinite dimensional aspect of the problem, deriving the necessary bounds is
considerably more involved when compared to the finite dimensional case. Moreover, to characterize the limit
as $\varepsilon\downarrow 0$ of the pair $(X^{\varepsilon,\delta,u}, Y^{\varepsilon,\delta,u})$, we will introduce the so called ``viable pair" construction (compare with
\cite{LDPWeakConvergence}, \cite{LDPFastSlow}, \cite{LDPQuenchedMultiscale} in finite dimensions).

The viable
pair is a pair of a trajectory and an occupation measuree $(\psi, \mathrm{P})$ that captures both the limit averaging dynamics of the controlled slow motion $X^{\varepsilon,\delta,u}$
and the invariant measure of the controlled fast process $Y^{\varepsilon,\delta,u}$. In addition, the measure $\mathrm{P}$ is obtained as the limit of a family of appropriate occupation measures $\mathrm{P}^{\varepsilon,\Delta}$ that live on the product space of fast motion and control, with $\Delta(\varepsilon) \to 0$ to be specified later on. Showing tightness of the family $\{\mathrm{P}^{\varepsilon,\Delta},\varepsilon,\Delta>0\}$  is considerably more delicate in infinite dimensions. Tightness and weak convergence of measure are topological properties and one must be careful about the topologies that are being considered.

In the pair $(\psi, \mathrm{P})$, the function
$\psi$ is the solution of the limiting averaging equation for the process $X^{\varepsilon,\delta,u}$, and the measure $\mathrm{P}$ is a probability measure on
the product space of the function space for the fast motion $Y^{\varepsilon,\delta,u}$ and the control space $U$.  The limiting
measure $\mathrm{P}$ characterizes simultaneously the structure of the invariant measure of $Y^{\varepsilon,\delta,u}$ and the control function $u$.
Note that in general, these two objects are intertwined and coupled together into the measure $\mathrm{P}$, so that the averaging with respect to the
measure $\mathrm{P}$ cannot be done as in the classical averaging principle
(see \cite[Chapters 7,8]{FWbook}). Rather, one has to fulfill the definition of the viable pair
as in Definition \ref{Def:ViablePair} below. The regime $\delta/\sqrt{\varepsilon} \rightarrow 0$ that we study leads to a decoupling of the limiting occupation measure $\mathrm{P}(dudYdt)=\eta_t(du|Y)\mu^{\psi_t}(dY)dt$, where $\eta_t(du|Y)$ is a stochastic kernel characterizing the control and $\mu^{\psi_t}(dY)$ is the invariant measure for the \textit{uncontrolled} fast process $Y^{\varepsilon,\delta}$ in \ref{Eq:FastSlowStochasticRDE}
with $X^{\varepsilon,\delta}$ replaced by $\psi_t$. The result on the weak convergence of the pair $(X^{\varepsilon,\delta,u},\mathrm{P}^{\varepsilon,\Delta})$ is the content of Theorem \ref{T:MainTheorem1}.

With the analysis of the limit of the controlled slow--fast process $(X^{\varepsilon,\delta,u}, Y^{\varepsilon,\delta,u})$ and the construction of viable pair, we then
prove  the Laplace principle (equivalently large deviation principle) for the slow process $X^{\varepsilon,\delta}$ in (\ref{Eq:FastSlowStochasticRDE}), which is the second result of the paper, Theorem \ref{Theorem:LargeDeviationPrinciple}. Proving the Laplace principle amounts to finding an appropriate functional $S(\cdot)$ such that for any bounded and continuous function $h: C( [0,T];L^{2}(D))\rightarrow \mathbb{R}$ we have
$$\lim\limits_{\varepsilon\downarrow 0}\varepsilon\ln \mathbf{E} \left[\exp \left(-\dfrac{1}{\varepsilon}h(X^{\varepsilon,\delta})\right)\right]
=-\inf\limits_{\phi\in C([0,T]; L^{2}(D))}[S(\phi)+h(\phi)] \ .$$

It turns out the Laplace principle upper bound can be proven using the weak convergence of the pair $(X^{\varepsilon,\delta,u},\mathrm{P}^{\varepsilon,\Delta})$ per Theorem \ref{T:MainTheorem1}. The situation is considerably more complicated for the Laplace principle lower bound, which is the second technical part of the paper. In order to prove the Laplace principle lower bound, we need to construct nearly optimal controls that achieve the bound. Due to the dependence on the fast motion $Y$, the nearly optimal controls will in principle be in feedback form and functions of both time $t$ and the fast motion $Y$. Hence, in order for averaging to work one needs to have some regularity of such controls, where in principle we only have that they are square integrable. In addition, given that the ergodic theorem for the cost is used with respect to $\delta\downarrow 0$ for fixed $t$, one also needs to have extra control on the growth of the control in order to be able to conclude that the time integral converges for each fixed length of time. For these reasons, we have been able to rigorously prove the lower bound in two special, but still general, cases, (a): one dimensional case with multiplicative white-noise, $d=1$ and $Q_1=I$, and (b): potential multidimensional case, $d\geq 1$ but with $\sigma_{1}(x,X,Y)=\sigma_{1}(x,X)$ being independent of $Y$. We elaborate in detail the reasons for doing this in Sections \ref{S:LDPproof} and \ref{S:Generalizations}.

Large deviations of stochastic partial differential equations of reaction--diffusion type
has been considered in previous works such as \cite{LDPInfiniteSDE,Cerrai-RocknerLDP,Chenal-MilletLDP,ChowLDP,Kallianpur-Xiong1996,Peszat1994,Sowers1992,Zabczyk}, but without the effect of multiple scales.
Results in the case of slow--fast systems of stochastic reaction--diffusion equations
has been considered in \cite{LDPFastSlowDuanEtAl} in dimension one, with additive noise in the fast motion and no noise component in the slow motion. In finite dimensions the large deviations problem for multiscale diffusions has been well studied, see \cite{Baldi,LDPWeakConvergence,FS, LDPFastSlow, LDPQuenchedMultiscale, Veretennikov, VeretennikovSPA2000}. To the best of our knowledge, the problem of large deviations for multiscale stochastic reaction diffusion equations  in multiple dimensions, with multiplicative noise is being considered for the first time in the present paper.


{At this point it is instructive to compare the derivations of the large deviations between the finite and the infinite dimensional settings. Following the weak convergence approach, the general strategy for the infinite dimensional case that appears in this paper is similar to the general strategy in the corresponding finite dimensional case, see \cite{LDPWeakConvergence}. However, the infinite dimensionality aspect of the problem means that most of the required a-priori estimates that are needed for tightness and then for convergence of the underlying control problem are considerably more delicate here than in the finite dimensional case. {In \cite{LDPWeakConvergence} the fast motion evolves in a finite-dimensional periodic domain while in the current paper the fast motion evolves in an unbounded infinite dimensional space. Both the unboundedness and the infinite dimensionality make ergodic properties such 爱as the existence of a unique invariant measure more difficult to derive.}

In addition, in the proof of the lower bound for the Laplace principle one needs to identify a nearly optimal control that nearly achieves the lower bound of the action functional. However, this is not enough. Due to the presence of the fast component $Y$ such a control will have to depend on $Y$ and for the subsequent averaging procedure to proceed such a dependency needs to be sufficiently smooth. While, this was clear in the finite dimensional case, see \cite{LDPWeakConvergence}, the situation here is considerably more complicated. The work in this paper rigorously resolves these issues under the appropriate conditions. }

The paper is organized as follows: in Section \ref{S:Notations} we give background definitions, set--up as well as our assumptions.
In Section \ref{S:MainResults} we review basic facts about weak--convergence method in infinite dimensions,
we define the viable pair as well as state our main results on averaging for controlled SRDE and
the large deviation theorem.  Section \ref{S:LLN_ControlledSRDE} is devoted to the  analysis of the limit of the controlled slow--fast processes
$(X^{\varepsilon,\delta,u}, Y^{\varepsilon,\delta,u})$ and the corresponding averaging result. In Section \ref{S:LDPproof} we prove the large deviations theorem.
Section \ref{S:Generalizations} is dedicated to discussions, remarks and generalizations for future work. The Appendix \ref{S:ErgodicProperties}
 collects some classical ergodic results
for the uncontrolled stochastic reaction--diffusion equation, which corresponds to the fast motion of our problem with frozen slow component.

\section{Set up: notations, function spaces and assumptions}\label{S:Notations}
In this section we set up the notation that will be used throughout the paper and state our assumptions.

We denote by $H$
the Hilbert space $L^2(D)$, endowed with the usual scalar product $\langle\bullet, \bullet\rangle_H$ and with the corresponding
norm $|\bullet|_H$. Let the norm in $L^\infty(D)$ be denoted by $|\bullet|_0$. We shall denote by $B_b(H)$ the Banach space of bounded Borel functions $\varphi: H\rightarrow \mathbb{R}$, endowed with the sup--norm
$$|\varphi|_{B_b(H)}:=\sup\limits_{X\in H}|\varphi(X)| \ .$$

The space $C_b(H)$ is the sub--space of uniformly continuous mappings and $C_b^k(H)$ is the subspace of all $k$--times (Fr\'echet) differentiable mappings,
having bounded and uniformly continuous derivatives, up to the $k$--th order ($k\in \mathbb{N}$). The space $C_b^k(H)$
is a Banach space endowed with the norm
$$|\varphi|_{C^k_b(H)}:=|\varphi|_{B_b(H)}+\sum\limits_{i=1}^k \sup\limits_{X\in H}|D^i\varphi(X)|_{\mathcal{L}^i(H)},$$
where $\mathcal{L}^1(H):=H$ and by recurrence $\mathcal{L}^i(H):=\mathcal{L}(H, \mathcal{L}^{i-1}(H))$ for any $i>1$. We denote by $\text{Lip}(H)$
the set of functions $\varphi: H\rightarrow \mathbb{R}$ such that
$$[\varphi]_{\text{Lip}(H)}:=\sup\limits_{X,Y\in H , X\neq Y}\dfrac{|\varphi(X)-\varphi(Y)|}{|X-Y|_H}<\infty \ .$$

We shall denote by $\mathcal{L}(H)$ the space of bounded linear operators in $H$ and
we shall denote by $\mathcal{L}_2(H)$ the subspace of Hilbert--Schmidt
operators, endowed with the norm $$\|Q\|_2=\sqrt{\text{Tr}[Q^*Q]} \ .$$

The stochastic perturbations in the slow and in the fast motion in system (\ref{Eq:FastSlowStochasticRDE})
are given, respectively, by the Gaussian noise $\dfrac{\partial W^{Q_1}}{\partial t}(t,x)$ and $\dfrac{\partial W^{Q_2}}{\partial t}(t,x)$
for $t\geq 0$ and $x\in D$, which are assumed to be white in time and colored in space, in the case of space dimension
$d>1$. The driving noises may or may not be independent. In order to deal with both cases at once, we define a cylindrical Wiener process on a Hilbert space $\mathbb{R}^\infty$, the space of infinite sequences of real numbers. Formally,
$$W(t) = \bigotimes_{k=1}^\infty \beta_k(t) \ ,$$
where $\{\beta_{k}\}$ is a sequence of independent one--dimensional Brownian motions. The linear operators $Q_i : \mathbb{R}^\infty \to H$, $i=1,2$ add color to the noise and also decide if the noises are independent.
The cylindrical Wiener processes $W^{Q_i}(t,x)$ are defined as
 \[W^{Q_i}(t,x) = Q_i W(t)(x).\]

As an example, in the case of spatial dimension $d=1$, the systems can be perturbed by space--time white noise. Let $f_i$ denote the element of $\mathbb{R}^\infty$ for which the $i$th component is one and all of the other components are zero. Let $e_i$ be a complete orthonormal basis of $H$. If the linear operators satisfy $Q_1 f_{2i} = e_i$, $Q_1 f_{2i - i} =0$, $Q_2 f_{2i} = 0$, and $Q_2 f_{2i-1} = e_i$, then
$\dfrac{\partial W^{Q_1}}{\partial t}(t,x)$ and $\dfrac{\partial W^{Q_2}}{\partial t}(t,x)$ are independent space--time white noises.
On the other hand, if we choose $Q_1 f_i = Q_2 f_i = e_i$, then $W^{Q_1}$ and $W^{Q_2}$ are the same space--time white noise.

We identify a Hilbert space subset $U$ of $\mathbb{R}^\infty$.
If $x \in \mathbb{R}^\infty$, let $x_i$ denote the $i$th component of the sequence. We define the Hilbert space $U \subset \mathbb{R}^\infty$
endowed with inner product $\left<x,y\right>_U = \sum\limits_{i=1}^\infty x_iy_i$. Thus the Hilbert space
$U=\{x\in \mathbb{R}^\infty, \sum\limits_{i=1}^\infty x_i^2<\infty\}$, and the norm for $x\in U$
is given by $|x|^2_U=\sum\limits_{i=1}^\infty x_i^2$.

The operators $\mathcal{A}_1$ and $\mathcal{A}_2$ appearing, respectively, in the slow and in the fast motion equation, are second order
uniformly elliptic differential operators, having continuous coefficients on $D$, and the boundary operators $\mathcal{N}_1$ and $\mathcal{N}_2$
can be either the identity operator (Dirichlet boundary condition) or a first--order operator of the following type
$$\sum\limits_{j=1}^d \beta_j(x)\dfrac{\partial}{\partial x_j}+\gamma(x)I \ , \ x \in \partial D \ , $$
for some $\beta_j, \gamma\in C^1(\bar{D})$ such that $$\inf\limits_{x\in \partial D}|\langle \beta(x), \nu(x)\rangle|>0 \ ,$$
where $\nu(x)$ is the unit normal at $x \in \partial D$ (uniformly non--tangential condition).

The realizations $A_1$ and $A_2$ in $H$ of the differential operators $\mathcal{A}_1$ and $\mathcal{A}_2$, endowed, respectively,
with the boundary conditions $\mathcal{N}_1$ and $\mathcal{N}_2$, generate two analytic semigroups $S_1(t)$ and $S_2(t)$, $t \geq 0$.
In addition, for $\theta\in\mathbb{R}$ and $i=1,2$ we define the Sobolev space $H^{\theta}_i$ with norm
\[
|x|_{\theta,i}=\left|\left(-A_{i}\right)^{\theta/2}x\right|_{_{H}},
\]
where $A_1, A_2$ denote the realizations of $\mathcal{A}_1$ and $\mathcal{A}_2$ in $H$, endowed with their respective boundary conditions. Clearly, for $\theta=0$ we have $H^{0}_{i}=H=L^{2}(D)$ for $i=1,2$.

In what follows, we shall assume that $A_1$, $A_2$ and $Q_1$, $Q_2$ satisfy the following conditions.

\textit{Hypothesis 1.} For $i=1,2$, there exist  complete orthonormal systems $\{e_{i,k}\}_{k\in \mathbb{N}}$
in $H$, and sequences of non--negative real
numbers $\{\alpha_{i,k}\}_{k\in \mathbb{N}}$ ,
such that
$$A_i e_{i,k}=-\alpha_{i,k}e_{i,k} \ , \ k \geq 1 \ .$$
The covariance operators $Q_i: U \to H$, $i=1,2$ are diagonalized by the same orthonormal basis $\{e_{i,k}\}_{k \in \mathbb{N}}$ in the following sense. For $i=1,2$, there exists an orthonormal set $\{f_{i,k}\}_{k\in \mathbb{N}} \subset U$. The set of $\{f_{i,k}\}_{k \in \mathbb{N}}$ is not necessarily complete. There exist sequences of non--negative real numbers $\{\lambda_{i,k}\}_{\substack{i=1,2\\k\in\mathbb{N}}}$ satisfying
\[Q_i f_{i,k} = \lambda_{i,k} e_{i,k}.\]
Notice that if $\text{span}\{f_{1,k}\}_{k \in \mathbb{N}} \perp \text{span}\{f_{2,k}\}_{k \in \mathbb{N}}$, then the driving noises of the fast and the slow motion are independent.

If $d=1$, then we have, recalling that $|\bullet|_0$ is the $L^\infty(D)$ norm,
\begin{equation*}
\kappa_i:=\sup\limits_{k\in \mathbb{N}}\lambda_{i,k}|e_{i,k}|_0<\infty \ , \ \zeta_i:=\sum\limits_{k=1}^\infty \alpha_{i,k}^{-\beta_i}|e_{i,k}|_0^2<\infty\nonumber
\end{equation*}
for some constant $\beta_i\in (0,1)$, and if $d\geq 2$, we have
\begin{equation}\label{Eq:Hypothesis1Equation2}
\kappa_i:=\sum\limits_{k=1}^\infty \lambda_{i,k}^{\rho_i}|e_{i,k}|_0^2<\infty \ , \ \zeta_i:=\sum\limits_{k=1}^\infty \alpha_{i,k}^{-\beta_i}|e_{i,k}|_0^2<\infty
\end{equation}
for some constants $\beta_i\in (0,+\infty)$ and $\rho_i\in (2,+\infty)$ such that
\begin{equation}\label{Eq:Hypothesis1Equation3}
\dfrac{\beta_i(\rho_i-2)}{\rho_i}<1 \ .
\end{equation}
Moreover
\begin{equation}\label{Eq:Hypothesis1Equation4}
\inf\limits_{\substack{k\in \mathbb{N}\\i=1,2}}\alpha_{i,k}=:\lambda>0 \ .
\end{equation}

We impose the following conditions on the terms $b_1, b_2$ and $\sigma_1, \sigma_2$.
For $i=1,2$, let us define the Lipschitz constants
\begin{equation*}
\sup\limits_{x\in D, Y\in \mathbb{R}} \ \sup\limits_{X_1, X_2\in \mathbb{R}, X_1\neq X_2}\dfrac{|b_i(x,X_1,Y)-b_i(x,X_2,Y)|}{|X_1-X_2|}=:L_{b_i}^X \ ,
\end{equation*}
\begin{equation*}
\sup\limits_{x\in D, X\in \mathbb{R}} \ \sup\limits_{Y_1, Y_2\in \mathbb{R}, Y_1\neq Y_2}\dfrac{|b_i(x,X,Y_1)-b_i(x,X,Y_2)|}{|Y_1-Y_2|}=:L_{b_i}^Y \ ,
\end{equation*}
\begin{equation*}
\sup\limits_{x\in D, Y\in \mathbb{R}} \ \sup\limits_{X_1, X_2\in \mathbb{R}, X_1\neq X_2}\dfrac{|\sigma_i(x,X_1,Y)-\sigma_i(x,X_2,Y)|}{|X_1-X_2|}=:L_{\sigma_i}^X \ ,
\end{equation*}
\begin{equation*}
\sup\limits_{x\in D, X\in \mathbb{R}} \ \sup\limits_{Y_1, Y_2\in \mathbb{R}, Y_1\neq Y_2}\dfrac{|\sigma_i(x,X,Y_1)-\sigma_i(x,X,Y_2)|}{|Y_1-Y_2|}=:L_{\sigma_i}^Y \ .
\end{equation*}

\textit{Hypothesis 2.} 1. The mappings $b_i: D\times \mathbb{R}^2 \rightarrow \mathbb{R}$ and $\sigma_i: D\times \mathbb{R}^2 \rightarrow \mathbb{R}$
are measurable, both for $i=1$ and for $i=2$, and $\sum\limits_{i=1,2}(L_{b_i}^X+L_{\sigma_i}^X+L_{b_i}^Y+L_{\sigma_i}^Y)\leq M$ for some $M>0$.
Moreover,
$$\sup\limits_{x \in D}|b_2(x , 0,0)|<\infty \ , \ \sup\limits_{x\in D}|\sigma_2(x,0,0)|<\infty \ .$$

2. Recalling $\lambda$, the constant introduced in (\ref{Eq:Hypothesis1Equation4}), we have that
\begin{equation}\label{Eq:Hypothesis2Equation1}
L_{b_2}^Y<\lambda. \
\end{equation}

3. $\sigma_2$ grows linearly in $X$, but is bounded in $Y$. There exists $c>0$,
\begin{equation}\label{Eq:Hypothesis2Equation2}
\sup\limits_{x \in D}\sup_{Y \in \mathbb{R}}|\sigma_2(x, X, Y)|\leq c(1+|X|).
\end{equation}

4. The Lipschitz constants $L_{b_2}^Y$ and $L_{\sigma_2}^Y$ are chosen so that
\begin{align}\label{Eq:Hypothesis2Equation3}
\dfrac{L_{b_2}^Y}{\lambda}+\sqrt{K_2(L_{\sigma_2}^Y)^2\int_0^{\infty}s^{-\beta_2\frac{\rho_2-2}{\rho_2}}
e^{-\lambda\frac{\rho_2+2}{\rho_2}s}ds}=:\mathfrak{L}_{b_2,\sigma_2}^Y&<1
\end{align}
where
\begin{align}
K_2&=\left(\dfrac{\beta_2}{e}\right)^{\beta_2\frac{\rho_2-2}{\rho_2}}\zeta_2^{\frac{\rho_2-2}{\rho_2}}\kappa_2^{\frac{2}{\rho_2}}, \nonumber \end{align}
and $\lambda, \beta_2, \rho_2, \zeta_2, \kappa_2$ are all from Hypothesis 1.
\begin{remark}
 Condition (\ref{Eq:Hypothesis2Equation3}) is a technical condition used in proving that in a certain ergodic sense the presence of the control in the fast dynamics $Y$ does not influence the corresponding invariant measure (a consequence of the regime under consideration $\delta/\sqrt{\varepsilon}\downarrow 0$) and in proving that the invariant measure is weakly Lipschitz continuous with respect to the slow component (see Lemma \ref{Lm:mu-continuous}).
\end{remark}


\textit{Hypothesis 3.}
 $b_1$ and $\sigma_1$ grow at most linearly in $X$ and sublinearly in $Y$. To be precise, there exists $0\leq \zeta<1-\frac{\beta_{1}(\rho_{1}-2)}{\rho_{1}}$ and a constant $C>0$ such that
\begin{equation} \label{Eq:Sm1-growth}
  \sup_{x\in D} \left(|b_1(x,X,Y)|+|\sigma_1(x,X,Y)|\right)\leq C(1 + |X| + |Y|^{\zeta})\ .
\end{equation}

\begin{remark}
The proofs of Section \ref{S:LLN_ControlledSRDE} show that if $\sigma_{1}$ is bounded with respect to $X$, then one can relax the growth restrictions with respect to $Y$ for the drift term $b_{1}$ and assume arbitrary sublinear growth of $b_{1}$ with respect to $Y$.
\end{remark}
\begin{remark}
Notice that if $Q_{1}$ is trace class, then $\rho_{1}=2$, in which case one can allow arbitrary sublinear growth of $\sigma_{1}$ with respect to $Y$.
Also, we remark here that since we have (\ref{Eq:Hypothesis1Equation3}), (\ref{Eq:Hypothesis1Equation4})
as well as the fact that $0<\beta_2\frac{\rho_2-2}{\rho_2}<1$, the integral term
$\int_0^{\infty}s^{-\beta_2\frac{\rho_2-2}{\rho_2}}e^{-\lambda\frac{\rho_2+2}{\rho_2}s}ds<\infty$. In fact we can write it in terms of the gamma function
\[
\int_0^{\infty}s^{-\beta_2\frac{\rho_2-2}{\rho_2}}e^{-\lambda\frac{\rho_2+2}{\rho_2}s}ds= \left(\lambda\frac{\rho_2+2}{\rho_2}\right)^{\beta_2\frac{\rho_2-2}{\rho_2}-1}\Gamma\left(1-\beta_2\frac{\rho_2-2}{\rho_2}\right)
\]
 Notice also that if $\sigma_{2}$ does not depend on $Y$, then the requirement (\ref{Eq:Hypothesis2Equation3}) follows directly from the requirement (\ref{Eq:Hypothesis2Equation1}).
\end{remark}
However, for the proof of the upper bound of the Laplace principle and for reasons that will become clearer later on, we need to strengthen these requirements to the following, which is strictly stronger than Hypothesis 3.

\

\textit{Hypothesis 4.}
 $b_1$ is as in Hypothesis 3. In regards to $\sigma_{1}$, either $d=1$ and there are positive constants $0<c_{0}\leq c_{1}<\infty$  such that $0<c_{0}\leq \sigma^{2}_{1}(x,X,Y)\leq c_{1}$, or  $d\geq 1$ and $\sigma_{1}(x,X,Y)=\sigma_{1}(x,X)$ is independent of $Y$ and can grow at most linearly in $X$ uniformly in $x\in D$.

Moreover, for $i=1,2$, we shall set
$$B_i(X,Y)(x):=b_i(x, X(x), Y(x))$$
and
$$[\Sigma_i(X,Y)Z](x):=\sigma_i(x, X(x), Y(x))Z(x)$$
for any $x\in D$, $X,Y,Z\in H$ and $i=1,2$. From Hypothesis 2 we know that the mappings
$$(X,Y)\in H\times H \mapsto B_i(X,Y)\in H \ ,$$
are Lipschitz continuous, as well as the mappings
$$(X,Y)\in H\times H \mapsto \Sigma_i(X,Y)\in \mathcal{L}(H; L^1(D))$$
and
$$(X,Y)\in H\times H \mapsto \Sigma_i(X,Y)\in \mathcal{L}(L^\infty(D); H) \ .$$

For any metric space $E$, we define $\mathscr{P}(E)$ to be the collection of probability measures on $E$.

As known from the existing literature such as \cite{DaPrato-Zabczyk} (also see \cite{CerraiRDEAveraging1}), according to Hypotheses 1 and 2
for any $\varepsilon>0$, $\delta>0$ and $X_0, Y_0\in H$ and for any $p\geq 1$ and $T>0$ there exists a unique
mild solution $(X^{\varepsilon, \delta}, Y^{\varepsilon, \delta})\in L^p(\Omega;C([0,T];H) \times C([0,T];H))$ to system (\ref{Eq:FastSlowStochasticRDE}).

Finally, concerning the small parameters $\varepsilon>0$ and $\delta>0$, we assume that we have the following.

\

\textit{Hypothesis 5.} We assume that $\varepsilon\downarrow 0$, $\delta=\delta(\varepsilon)\downarrow 0$ and $\Delta=\Delta(\delta,\varepsilon)\downarrow 0$, such that
\begin{equation}\label{Eq:AssumptionRegime1Restricted}
\lim\limits_{\varepsilon\downarrow 0}\dfrac{\delta}{\sqrt{\varepsilon}}=0,\text{ and }\lim\limits_{\varepsilon\downarrow 0}\dfrac{\delta}{\Delta\sqrt{\varepsilon}}=0 \ .
\end{equation}

It is clear that when $\varepsilon\downarrow 0$, both $\delta\downarrow 0$ and $\Delta\downarrow 0$. Hence, for notational convenience we will many times simply write $\varepsilon\downarrow 0$, which implicitly implies that $\delta,\Delta\downarrow 0$ as well.
In addition, we note that (\ref{Eq:AssumptionRegime1Restricted}) implies that
$\dfrac{\Delta}{\delta^2}\rightarrow \infty$ as $\varepsilon\downarrow 0$.   
Parameter
$\Delta$ can be viewed as a time-scale separation parameter. In particular,  as we shall see in Section \ref{S:LLNInvariantMeasure}, Hypothesis 5 enables us to decouple
the invariant measure with respect to which the averaging is being done from the control process.

\section{Weak convergence and large deviations}\label{S:MainResults}

In this section we review the weak convergence approach to large deviations, \cite{DupuisEllis}, and then we state our main results of the paper on the averaging principle for controlled stochastic reaction-diffusion equations and on the large deviations principle for $\{X^{\varepsilon,\delta},\varepsilon>0\}$. As we also mentioned in the introduction, large deviations for SRDEs in the small noise regime (but in the absence of multiple scales), have been derived in \cite[Theorem 9]{LDPInfiniteSDE}. In particular, the authors in \cite{LDPInfiniteSDE} use the weak convergence formulation as well and establish large deviations for infinite dimensional SRDEs in the absence of multiple scales.
Before stating the main result of this paper, we review next the mathematical framework appropriately formulated in our setting of interest.
\begin{theorem}[see \cite{VariationalInfniteBM,LDPInfiniteSDE}]\label{Thoerem:WeakConvergenceInfiniteDimensions}
Let $f$ be a bounded, Borel measurable function mapping $C([0,T]; \mathbb{R}^\infty)$
into $\mathbb{R}$. Then
\begin{equation*}
-\ln\mathbf{E}(\exp\{-f(W)\})
=
\inf\limits_{u\in \mathcal{P}_2(U)}\mathbf{E}\left(\dfrac{1}{2}\int_0^T |u(s)|_U^2ds+f\left(W+\int_0^\bullet u(s)ds\right)\right) \ .
\end{equation*}
\end{theorem}

Here the set $\mathcal{P}_2(U)$ consists of all $U$--valued predictable processes $\phi(s)$
for which $\displaystyle{\int_0^T |\phi(s)|_U^2ds}<\infty$
almost surely. 

Let $\mathcal{E}$ and $\mathcal{E}_0$ be Polish spaces. For each $\varepsilon>0$, let $\mathcal{G}^\varepsilon: \mathcal{E}_0\times C([0,T]; \mathbb{R}^\infty)\rightarrow \mathcal{E}$ be a measurable map.
Consider the family of random elements $X^{\varepsilon,x}\equiv \mathcal{G}^\varepsilon(x, \sqrt{\varepsilon}W)$. From Theorem \ref{Thoerem:WeakConvergenceInfiniteDimensions}, we immediately derive that
for any  bounded and continuous function $h: \mathcal{E} \rightarrow \mathbb{R}$,
\begin{align}\label{Eq:EpsilonVariationalRepresentationInfinite}
&-\varepsilon\ln \mathbf{E}\left[\exp\left(-\dfrac{1}{\varepsilon}h(X^{\varepsilon,x^\varepsilon})\right)\right]
=
\inf\limits_{u\in \mathcal{P}_2(U)}\mathbf{E}\left[\dfrac{1}{2}\int_0^T |u(s)|_U^2ds\right.\nonumber\\
&\hspace{5cm}\left.+h\circ\mathcal{G}^\varepsilon\left(x^\varepsilon, \sqrt{\varepsilon}W+\int_0^\bullet u(s)ds\right)\right] \ .
\end{align}

Let us recall that
$H=L^2(D)$ and $A_1, A_2$ denote the realizations of $\mathcal{A}_1$ and $\mathcal{A}_2$ in $H$, endowed with their respective boundary conditions. Also, $A_1$ and $A_2$ generate $C_0$--semigroups $S_1(t)$ and $S_2(t)$. Notice that if $A_2$ is the infinitesimal generator of $S_2(t)$, then $\dfrac{1}{\delta^{2}}A_2$ is the infinitesimal generator of $S_2\left(\dfrac{t}{\delta^{2}}\right)$. We now recall the definition of a mild solution of (\ref{Eq:FastSlowStochasticRDE}). The mild solution to (\ref{Eq:FastSlowStochasticRDE}) solves
\begin{equation}\label{Eq:MildFastSlowStochasticRDE}
\left\{\begin{array}{l}
\displaystyle{X^{\varepsilon,\delta}(t)=S_1(t)X_0+ \int_0^t S_1(t-s)B_1(X^{\varepsilon,\delta}(s), Y^{\varepsilon,\delta}(s))ds
}\\
\displaystyle{\qquad \qquad \qquad  +\sqrt{\varepsilon}\int_0^t S_1(t-s)\Sigma_1(X^{\varepsilon,\delta}(s), Y^{\varepsilon,\delta}(s))dW^{Q_1}(s) \ ,}
\\
\displaystyle{Y^{\varepsilon,\delta}(t)=S_2\left(\dfrac{t}{\delta^{2}}\right)Y_0 + \frac{1}{\delta^{2}}\int_0^t S_2\left(\dfrac{t-s}{\delta^{2}}\right) B_2( X^{\varepsilon,\delta}(s), Y^{\varepsilon,\delta}( s))ds}\\
\displaystyle{\qquad \qquad \qquad  +\dfrac{1}{\delta}\int_0^t S_2\left(\dfrac{t-s}{\delta^{2}}\right)\Sigma_2( X^{\varepsilon,\delta}(s), Y^{\varepsilon,\delta}(s))dW^{Q_2}(s) \ .}
\end{array}\right.
\end{equation}

The solution map (interpreted as in (\ref{Eq:MildFastSlowStochasticRDE})) of (\ref{Eq:FastSlowStochasticRDE}) can be viewed as a Borel measurable map
$$\mathcal{G}^{\varepsilon,\delta}((X_0, Y_0), \sqrt{\varepsilon}W)=X^{\varepsilon,\delta} \ .$$
By (\ref{Eq:EpsilonVariationalRepresentationInfinite}), for any bounded and continuous function $h: C(H)\rightarrow \mathbb{R}$
we have
\begin{equation}\label{Eq:LDPRepresentationFastSlowSRDE}
-\varepsilon\ln \mathbf{E}\left[\exp\left(-\dfrac{1}{\varepsilon}h(X^{\varepsilon,\delta})\right)\right]
=
\displaystyle{\inf\limits_{u\in \mathcal{P}_2(U)}\mathbf{E}\left[\dfrac{1}{2}\int_0^T |u(s)|_{U}^2ds
+h(X^{\varepsilon,\delta,u})\right]} \ .
\end{equation}

Here the process $(X^{\varepsilon,\delta,u}, Y^{\varepsilon,\delta,u})$ is a controlled version of (\ref{Eq:FastSlowStochasticRDE}) where the control $u\in \mathcal{P}_2(U)$. The corresponding mild solutions satisfy the following controlled system of stochastic reaction--diffusion equations,
\begin{equation}\label{Eq:FastSlowStochasticRDEWithControl}
\left\{\begin{array}{l}
\displaystyle{\dfrac{\partial X^{\varepsilon,\delta,u}}{\partial t}(t,x)}=\displaystyle{\mathcal{A}_1 X^{\varepsilon,\delta,u}(t,x)+b_1(x, X^{\varepsilon,\delta,u}(t,x), Y^{\varepsilon,\delta,u}(t,x))}
\\
 \qquad \qquad   \displaystyle{+\sigma_1(x, X^{\varepsilon,\delta,u}(t,x), Y^{\varepsilon,\delta,u}(t,x))(Q_1u(t))(x)}
\\
 \qquad \qquad   \displaystyle{+\sqrt{\varepsilon}\sigma_1(x, X^{\varepsilon,\delta,u}(t,x), Y^{\varepsilon,\delta,u}(t,x))\dfrac{\partial W^{Q_1}}{\partial t}(t,x)} \ ,
\\
\dfrac{\partial Y^{\varepsilon,\delta,u}}{\partial t}(t,x) =\displaystyle{\dfrac{1}{\delta^{2}}\left[\mathcal{A}_2 Y^{\varepsilon,\delta,u}(t,x)+
b_2(x, X^{\varepsilon,\delta,u}(t,x), Y^{\varepsilon,\delta,u}(t,x))\right]}
\\
 \qquad \qquad   \displaystyle{+\dfrac{1}{\delta\sqrt{\varepsilon}} \sigma_2(x, X^{\varepsilon,\delta,u}(t,x), Y^{\varepsilon,\delta,u}(t,x))(Q_2u(t))(x)}
\\
\qquad \qquad   \displaystyle{+\dfrac{1}{\delta}\sigma_2(x, X^{\varepsilon,\delta,u}(t,x), Y^{\varepsilon,\delta,u}(t,x))\dfrac{\partial W^{Q_2}}{\partial t}(t,x)} \ ,
\\
X^{\varepsilon,\delta,u}(0,x)=X_0(x) \   , \ Y^{\varepsilon,\delta,u}(0,x)=Y_0(x) \ , \ x\in D \ ,
\\
\mathcal{N}_1 X^{\varepsilon,\delta,u}(t,x)=\mathcal{N}_2 Y^{\varepsilon,\delta,u}(t,x)=0 \  , \ t\geq 0 \ , \ x\in \partial D \ .
\end{array}\right.
\end{equation}

In particular, the mild formulation of the solution $(X^{\varepsilon,\delta,u}, Y^{\varepsilon,\delta,u})$ is the controlled process that solves
\begin{align}
X^{\varepsilon,\delta,u}(t)&=S_1(t)X_0+ \int_0^t S_1(t-s)B_1(X^{\varepsilon,\delta,u}(s), Y^{\varepsilon,\delta,u}(s))ds
\nonumber\\
&\qquad   +\int_0^t S_1(t-s)\Sigma_1(X^{\varepsilon,\delta,u}(s), Y^{\varepsilon,\delta,u}(s))Q_1 u(s) ds\ ,\nonumber\\
&\qquad  +\sqrt{\varepsilon}\int_0^t S_1(t-s)\Sigma_1(X^{\varepsilon,\delta,u}(s), Y^{\varepsilon,\delta,u}(s))dW^{Q_1}(s) \ ,\nonumber\\
Y^{\varepsilon,\delta,u}(t)&=S_2\left(\dfrac{t}{\delta^{2}}\right)Y_0 +
\frac{1}{\delta^{2}}\int_0^t S_2\left(\dfrac{t-s}{\delta^{2}}\right) B_2( X^{\varepsilon,\delta,u}(s), Y^{\varepsilon,\delta,u}( s))ds \label{Eq:MildFastSlowStochasticRDEWithControl}\\
&\qquad   +\dfrac{1}{\delta\sqrt{\varepsilon}}\int_0^t
S_2\left(\dfrac{t-s}{\delta^{2}}\right)\Sigma_2(X^{\varepsilon,\delta,u}(s), Y^{\varepsilon,\delta,u}(s))Q_2 u(s)ds \ ,\nonumber\\
&\qquad  +\dfrac{1}{\delta}
\int_0^t S_2\left(\dfrac{t-s}{\delta^{2}}\right)\Sigma_2(X^{\varepsilon,\delta,u}(s), Y^{\varepsilon,\delta,u}(s))dW^{Q_2}(s) \ .\nonumber
\end{align}

For each $N\in \mathbb{N}$, we define the set
\begin{equation*}
\mathcal{P}_2^N(U)=\left\{u\in \mathcal{P}_2(U): \displaystyle{\int_0^T |u(s)|_U^2ds\leq N}\right\} \ .
\end{equation*}

As in Theorem 10 of \cite{LDPInfiniteSDE} and for each $u\in \mathcal{P}_2^N(U)$ uniformly in $\varepsilon$, there is a unique pair $(X^{\varepsilon,\delta,u}, Y^{\varepsilon,\delta,u})\in L^p(\Omega;C([0,T];H)\times C([0,T];H))$ that satisfies (\ref{Eq:MildFastSlowStochasticRDEWithControl}).

Now, by Section 1.2 of \cite{DupuisEllis}, it is known that the Laplace principle, which amounts to finding the limit of the left hand side of (\ref{Eq:LDPRepresentationFastSlowSRDE}) as $\varepsilon\downarrow 0$, is equivalent to finding the large deviations principle for $\{X^{\varepsilon,\delta},\varepsilon>0\}$. This is the path that we follow in this paper for finding the large deviations principle for the family $\{X^{\varepsilon,\delta},\varepsilon>0\}$ in $C([0,T],H)$. Also, as it is shown in \cite{VariationalInfniteBM}, the representation implies that we can actually consider $u=u^{\varepsilon}\in \mathcal{P}_2^N(U)$  for a sufficiently large but fixed $N\in\mathbb{N}$.

Let us denote the slow motion space to be $\mathcal{X}=H=L^2(D)$, the fast motion space to be $\mathcal{Y}=H=L^2(D)$, and the control space to be
$U$. In addition, let us define $\xi=\xi(X,Y, u): \mathcal{X}\times \mathcal{Y} \times U \rightarrow H$ by
\begin{equation}\label{Eq:xiDef}
\xi(X, Y, u)=\Sigma_1(X,Y)Q_1 u+B_1(X, Y) \ .
\end{equation}

Moreover, for any fixed $X\in \mathcal{X}$,  consider the fast process $Y^{X}$ defined by the equation
\begin{align}
\dfrac{\partial Y^{X}}{\partial t}(t,x)&=[\mathcal{A}_2 Y^{X}(t,x)+
b_2(x, X(x), Y^{X}(t,x))]\nonumber\\
&\qquad+\sigma_2(x, X(x), Y^{X}(t,x))\dfrac{\partial W^{Q_2}}{\partial t}(t,x), t\geq 0 \text{ and } x\in D \ .
\nonumber\\
Y^{X}(0,x)&=Y_0(x) \ , \ x\in D \ , \ \mathcal{N}_2 Y^{X}(t,x)=0 \ , \ t\geq 0 \ , \ x\in \partial D \ .\label{Eq:FastProcessSRDEReg1}
\end{align}

Let $\mathcal{L}=\mathcal{L}^{X}$ be the generator of the process $Y^{X}$. As noted in \cite{Cerrai-FreidlinRDEAveraging1}, this generator has the form
\begin{align}\label{Eq:FastGeneratorReg1}
&\mathcal{L}^{X}\varphi(Y) = \langle A_2Y + B_2(X,Y), D_Y\varphi(Y) \rangle_H + \frac{1}{2}\text{Tr}[\Sigma_2(X,Y)Q_2Q_2^\star\Sigma_2^\star(X,Y)D_Y^2\varphi(Y)],
\end{align}
where $D_Y$ and $D^2_Y$ are the first and second Fr\'echet derivatives in $H$. The domain of definition of the operator $\mathcal{L}=\mathcal{L}^{X}$ is a set $\mathcal{D}(\mathcal{L})\subset H$ such that
any $\varphi\in \mathcal{D}(\mathcal{L})\rightarrow \mathbb{R}$ is twice continuously differentiable with $D^2_Y \varphi(Y)\in H$ for any $Y\in H$,
and the mapping $Y\mapsto \text{Tr}[D^2_Y \varphi(Y)]$ is continuous on $H$ with values in $\mathbb{R}$.

In addition, as we also review in Appendix \ref{S:ErgodicProperties}, Hypothesis 1 and 2, guarantee that the process $Y^{X}$ is, for each $X\in H$, ergodic and strongly mixing with a unique invariant measure, which we denote by $\mu^{X}(dy)$. For any bounded and continuous $f \in C_b(H) $
\begin{equation} \label{Eq:mu-X}
  \int_H f(Y) \mu^X(dY) = \lim_{T \to +\infty} \frac{1}{T} \int_0^T f(Y^{X}(t))dt.
\end{equation}

We now state without proof an important result on the continuity of $X \mapsto \mu^X$, which is a consequence of Lemma 3.1 of \cite{CerraiRDEAveraging1}.
\begin{lemma} \label{Lm:mu-continuous}
  Suppose that $f: H \to \mathbb{R}$ is Lipschitz continuous. For any $X \in H$ define
  $F(X)=\int_{\mathcal{Y}} f(Y)\mu^{X}(dY)$. Then for any $X_1,X_2 \in H$,
\begin{align}
\left|F(X_{1})-F(X_{2})\right|&\leq \|f\|_{\textnormal{Lip}} \left|X_{1}-X_{2}\right|_{H}.\label{Eq:WeaklyContinuousInvariantMeasure}
\end{align}
\end{lemma}

We need to understand not just the limit of the slow dynamics $X^{\varepsilon,\delta,u}$ but also the measure with respect to which the averaging is being done.  This is complicated in our case, due to  the dependence of the dynamics on the unknown control process $u=u^{\varepsilon}$. Following the recipe of \cite{LDPWeakConvergence}, for the periodic finite dimensional case, we introduce the family of random occupation measures
\begin{equation}\label{Eq:OccupationMeasureBeforeAveraging}
\mathrm{P}^{\varepsilon,\Delta}(dudYdt)=\dfrac{1}{\Delta}\int_t^{t+\Delta}\1_{du}(u(s))\1_{dY}\left(Y^{\varepsilon,\delta,u}(s)\right)dsdt
\end{equation}
on $U \times \mathcal{Y} \times[0,T]$, where  $\Delta=\Delta(\varepsilon)\to 0$ is as in Hypothesis 5. These occupation measures encode the behavior of the control and the fast process. It is the correct way to study the problem because the fast motion's behavior will not converge pathwise to anything, but its occupation measure will converge to a limiting measure. We adopt the convention that the control $u(t)=u^{\varepsilon}(t)=0$ for $t>T$. Then, we consider the joint limit in distribution of pair $(X^{\varepsilon,\delta,u}, \mathrm{P}^{\varepsilon,\Delta})$ as $\varepsilon,\delta,\Delta\downarrow 0$.

In order to state our main results, we introduce the following definition of a viable pair corresponding
to \cite{LDPWeakConvergence}, but appropriately extended to infinite dimensions.
\begin{definition}\label{Def:ViablePair}
A pair $(\psi, \mathrm{P})\in C([0,T]; L^2(D))\times \mathscr{P}(U\times \mathcal{Y}\times [0,T])$
will be called viable with respect to $(\xi, \mathcal{L})$,
or simply viable if there is no confusion, if the following are satisfied. The trajectory
$\psi\in C([0,T]; H)$, $\mathrm{P}$ is square integrable in the sense that for some $\theta>0$,
$$\displaystyle{\int_{U\times\mathcal{Y}\times [0,T]}\left(|u|_{U}^2+|Y|^{2}_{\theta,2}\right)\mathrm{P}(dudYds)<\infty}$$
and the following hold for all $t\in [0,T]$:
\begin{equation}\label{Eq:ViablePairAveragingDynamics}
\psi(t)=S_1(t)X_0+\int_{U \times\mathcal{Y}\times[0,t]}S_1(t-s)\xi(\psi(s),Y,u)\mathrm{P}(dudYds) \ ,
\end{equation}
the measure $\mathrm{P}$ is such that
\begin{align}\label{Eq:ViablePairInvariantMeasure}
\mathrm{P}\in\mathbb{P}=\left\{
   \begin{matrix}
      P\in\mathscr{P}(U\times\mathcal{Y}\times[0,T]):
       \mathrm{P}(dudYdt)=\eta(du|Y,t)\mu(dY|t)dt, \\
  \mu(dY|t)= \mu^{\psi(t)}(dY) \text{ for } t \in [0,T]
  \end{matrix}\right\}
\end{align}
where $\mu^X$ is from (\ref{Eq:mu-X}), {and $\eta(du|Y,t)$ is a stochastic kernel on $U$ given $\mathcal{Y}\times[0,T]$, (see Appendix A.5 of \cite{DupuisEllis} for stochastic kernels)}, and
\begin{equation}\label{Eq:ViablePairNormalization}
\mathrm{P}(U\times \mathcal{Y}\times [0,t])=t \ .
\end{equation}
\end{definition}

We denote a viable pair by $(\psi, \mathrm{P})\in \mathcal{V}_{(\xi, \mathcal{L})}$. Notice that condition (\ref{Eq:ViablePairInvariantMeasure}) in Definition \ref{Def:ViablePair} essentially means that the second marginal of the limiting occupation measure, $\mu(dY|t)$, coincides with the invariant measure associated with  the $Y^{X, Y_0}$ with $X=\psi(t)$ from (\ref{Eq:FastProcessSRDEReg1}).   Heuristically speaking, the viable
pair $(\psi, \mathrm{P})\in \mathcal{V}_{(\xi, \mathcal{L})}$ captures both the limit averaging dynamics of the controlled slow motion $X^{\varepsilon,\delta,u}$
in terms of (\ref{Eq:ViablePairAveragingDynamics}) and the invariant measure of the controlled fast process $Y^{\varepsilon,\delta,u}$
in terms of (\ref{Eq:ViablePairInvariantMeasure}).  Using the viable pair definition, we can then state the main results of our paper.
\begin{theorem} \label{T:MainTheorem1}(Averaging for controlled system)
For $u\in \mathcal{P}_2^N(U)$ let
$(X^{\varepsilon,\delta,u}, Y^{\varepsilon,\delta,u})$ be the mild solution to (\ref{Eq:MildFastSlowStochasticRDEWithControl}) and $T<\infty$. Let also $\mathrm{P}^{\varepsilon,\Delta}(dudYdt)$ be given by (\ref{Eq:OccupationMeasureBeforeAveraging}). Assume Hypotheses 1, 2, 3 and 5, $X_0\in H$ and $Y_0\in H$. Then, the family of processes $\{X^{\varepsilon,\delta,u}: \varepsilon \in (0,1), u \in \mathcal{P}_2^N(u)\}$ is tight in $C([0,T]; H)$ and
the family of measures $\{\mathrm{P}^{\varepsilon,\Delta}: \varepsilon \in (0,1), u \in \mathcal{P}_2^N\}$ is tight in $\mathscr{P}(U\times \mathcal{Y}\times [0,T])$, where $U\times \mathcal{Y} \times [0,T]$ is endowed with the weak topology on $U$, the norm topology on $\mathcal{Y}$ and the standard topology on $[0,T]$. Hence, given any subsequence of $\{(X^{\varepsilon,\delta,u},\mathrm{P}^{\varepsilon,\Delta}),\varepsilon,\delta,\Delta>0\}$, there exists a subsubsequence
that converges in distribution with limit $(\bar{X},\mathrm{P})$. With
probability $1$, the accumulation point $(\bar{X},\mathrm{P})$ is a
viable pair with respect to $(\xi, \mathcal{L})$ according to
Definition \ref{Def:ViablePair}.
\end{theorem}

\begin{theorem}\label{Theorem:LargeDeviationPrinciple}
(Large Deviation Principle) Let $(X^{\varepsilon,\delta}, Y^{\varepsilon,\delta})$ be the mild solution to (\ref{Eq:FastSlowStochasticRDE}) and let $T<\infty$. Assume Hypothesis 1, 2, 4 and 5 and let $Y_{0}\in H$ and $X_{0}\in H$. Define
$$S(\phi)=S_{X_{0}}(\phi)=\inf\limits_{(\phi,\textrm{P})\in \mathcal{V}_{(\xi, \mathcal{L})}}
\left[\dfrac{1}{2}\int_{U\times \mathcal{Y}\times [0,T]}|u|_{U}^2\mathrm{P}(dudYdt)\right] \ ,$$
with the convention that the infimum over the empty set is $\infty$. Then for every bounded and continuous function $h: C([0,T]; H)\rightarrow \mathbb{R}$ we have
$$\lim\limits_{\varepsilon\downarrow 0}-\varepsilon\ln \mathbf{E}_{X_{0},Y_{0}} \left[\exp \left(-\dfrac{1}{\varepsilon}h(X^{\varepsilon,\delta})\right)\right]
=\inf\limits_{\phi\in C([0,T]; H)}[S(\phi)+h(\phi)] \ .$$
In particular, $\{X^{\varepsilon,\delta}\}$ satisfies the large deviations principle in $C([0,T];H)$ with action functional
$S(\cdot)$. 
\end{theorem}

{
   We show in \eqref{Eq:ActionFunctOrdinaryForm}, Section \ref{SS:LDPUpperBound}, that an equivalent representation for the large deviations rate functional $S(\psi)$ can be given as the following minimizing control problem over an admissible class of measurable functions $v:[0,T]\times\mathcal{Y} \to U$.

   Let
   \begin{equation*}
      \begin{array}{ll}
      \mathcal{A}_{\psi,T}^{o} & \displaystyle{=\left\{v: [0,T] \times \mathcal{Y} \to U: \int_0^T\int_{\mathcal{Y}}\left(|v(t,Y)|_U^2+|Y|^{2}_{\theta,2}\right)\mu^{\psi(t)}(dY)dt<\infty  \ , \right.}
      \\
      &  \displaystyle{ \ \left.\psi(t)=S_1(t)X_0+\int_{\mathcal{Y}\times[0,t]}S_1(t-s)\xi(\psi(s),Y,v(s,Y))\mu^{\psi(s)}(dY)ds, \ t \in [0,T]\right\}} \ .
      \end{array}
   \end{equation*}
   where $\xi$ is given by \eqref{Eq:xiDef}.

   An equivalent representation of the rate function is
   \[S(\psi) =\inf_{v \in \mathcal{A}_{\psi,T}^o} \frac{1}{2} \int_0^T  \int_\mathcal{Y} |v(t,Y)|_U^2 \mu^{\psi(t)}(dY)dt.\]

   In the special case where $\Sigma_1(X,Y)=\Sigma(X)$ is independent of $Y$, $\mathcal{A}_{\psi,T}^o$ has the simpler form
   \begin{align*}
     \mathcal{A}_{\psi,T}^o = \Big\{ u \in L^2([0,T];U) :
     \psi(t) = &S_1(t)X_0 + \int_0^t \int_\mathcal{Y} S_1(t-s)B_1(\psi(s),Y)\mu^{\psi(s)}(dY)ds\\
     &
      + \int_0^t S_1(t-s)\Sigma_1(\psi(s))Q_1u(s)ds, \ \  t \in [0,T].\Big\}
   \end{align*}
   and the rate function has the representation
   \[S(\psi) = \inf_{u \in \mathcal{A}_{\psi,T}^o} \frac{1}{2}\int_0^T |u(s)|_U^2ds.\]
}


In Section \ref{S:LLN_ControlledSRDE} we consider the limit of the controlled SRDE (\ref{Eq:MildFastSlowStochasticRDEWithControl}) as $\varepsilon,\delta\downarrow 0$ and prove Theorem \ref{T:MainTheorem1}. Then, in Section \ref{S:LDPproof} we give the proof of Theorem \ref{Theorem:LargeDeviationPrinciple}.

\section{Analysis of the limit of the controlled SRDEs-Proof of Theorem \ref{T:MainTheorem1}}\label{S:LLN_ControlledSRDE}
In this section, we analyze the limit of the system of controlled SRDEs (\ref{Eq:MildFastSlowStochasticRDEWithControl}) as $\varepsilon,\delta\downarrow 0$. As we mentioned in Section \ref{S:MainResults}, we need to introduce the family of occupation measures $\mathrm{P}^{\varepsilon,\Delta}(dudYdt)$ as defined by (\ref{Eq:OccupationMeasureBeforeAveraging}).

We emphasize that these occupation measures are measure-valued random variables. We are interested in proving that the laws of the occupations measures are tight in order to prove that a subsequence converges weakly. The study of tightness for these occupation measures is considerably more delicate over the infinite dimensional spaces $U$ and $\mathcal{Y}$ than in the finite dimensional space studied in \cite{LDPWeakConvergence}. Tightness of measures as well as weak convergence of measures are inherently topological properties and, therefore, we must be  careful about the topologies that we are discussing.

The appropriate topology to impose on $U \times \mathcal{Y} \times [0,T]$ is the weak topology on $U$ times the norm topology on $\mathcal{Y}$ times $[0,T]$. If we restrict ourselves to bounded subsets of $U$, then this topology is metrizable because the weak topology on bounded subsets of $U$, which is a separable Hilbert space, is metrizable. We recall the famous Prokhorov Theorem and specifically draw the reader's attention to the sensitivity of these results on the chosen topology.

\begin{definition}
  Let $E$ be a metric space. A family of probability measures $\{P_\alpha\} \subset \mathscr{P}(E)$is called tight if for all $\eta>0$ there exists a compact set $K_\eta \subset E$ such that
  \[\inf_{\alpha}P_\alpha(K_\eta) > 1-\eta.\]
\end{definition}
\begin{definition}
  Let $E$ be a metric space. A family of probability measures $\{P_\alpha\} \subset \mathscr{P}(E)$ is called relatively compact if for any subsequence in $\{P_n\} \subset \bigcup_\alpha\{{P}_\alpha\}$, there exists a subsequence (relabeled $P_n$) that converges weakly to some limit $\bar{P}$. That is, for any continuous function $f: E \to \mathbb{R}$,
  \[\int\limits_E f(x) P_n(dx) \to \int\limits_E f(x) \bar{P}(dx)\]
\end{definition}

Notice that both of these definitions are topological. Tightness refers to compact sets and relative compactness refers to continuous functions.
\begin{theorem}[Prokhorov's Theorem]
  Let $E$ be a metric space. If a family of probability measures on $E$ is tight, then it is relatively compact.
\end{theorem}

In order to prove that the laws of $\mathrm{P}^{\varepsilon,\Delta}$ are relatively compact, we need to apply the Prokhorov Theorem twice (because they are probability measures on the space of measures on the space $U\times \mathcal{Y}\times [0,T]$). For this, it is convenient to recall the use of tightness functions (see Appendix A.3 of \cite{DupuisEllis}).

Showing that the laws of $X^{\varepsilon,\delta,u}$ are tight in $C([0,T];H)$ is standard. We need to demonstrate that the paths have enough spatial and temporal regularity so that they belong to compact subsets.

We show that any limit of the pair $(X^{\varepsilon,\delta,u}, \mathrm{P}^{\varepsilon,\Delta})\rightarrow (\bar{X}, \mathrm{P})$
in the space $C([0,T]; H)\times \mathscr{P}(U\times \mathcal{Y}\times[0,T])$ is  a viable pair $(\bar{X}, \mathrm{P})$ according to Definition \ref{Def:ViablePair}. As in \cite{VariationalInfniteBM}, the representation (\ref{Eq:LDPRepresentationFastSlowSRDE}) guarantees that it is enough to consider controls $u\in\mathcal{P}_2^{N}(U)$ for an appropriate large enough $N\in\mathbb{N}$ that is independent of $\varepsilon$. In particular, we shall consider controls $u=u^{\varepsilon}$ that may depend on $\varepsilon$, but such that there exists $N\geq 0$ such that for all $\varepsilon \in (0,1)$, we have $u^\varepsilon \in \mathcal{P}_2^N$.

 In Section \ref{S:TightnessPair}, we show that the pair $(X^{\varepsilon,\delta,u^\varepsilon}, \mathrm{P}^{\varepsilon,\Delta})$ is appropriately tight. Then in Sections \ref{S:LLNSlowProcess} and \ref{S:LLNInvariantMeasure} we show that any accumulation point as $\varepsilon\downarrow0$ will be a viable pair per Definition \ref{Def:ViablePair}.

\subsection{Step 1: Tightness of the pair $\{(X^{\varepsilon,\delta,u}, \mathrm{P}^{\varepsilon,\Delta}) \}$}\label{S:TightnessPair}

We show the tightness of the pair
$\{(X^{\varepsilon,\delta,u}, \mathrm{P}^{\varepsilon,\Delta}), 0<\varepsilon<1, u \in \mathcal{P}_2^N\}$ in the space\\ $C([0,T]; H)\times \mathscr{P}(U\times \mathcal{Y}\times[0,T])$.
Tightness guarantees that for any subsequence of $\varepsilon \to 0$
there exists a sub--subsequence that converges, in distribution, to some limit $(\bar{X},\mathrm{P})$, i.e.,
\begin{equation}\label{Eq:ConvergencePairSlowProcessAndOccupationMeasure}
(X^{\varepsilon,\delta,u}, \mathrm{P}^{\varepsilon,\Delta})\rightarrow (\bar{X},\mathrm{P}) \ .
\end{equation}

The tightness proof is obtained by a--priori bounds for the slow process $X^{\varepsilon,\delta,u}(t)$ in (\ref{Eq:FastSlowStochasticRDEWithControl})
in a suitable H\"{o}lder norm with respect to time and in a suitable Sobolev norm with respect to space, as well as second moment bounds
for the fast process $Y^{\varepsilon,\delta,u}(t)$.

\subsubsection{A--priori bounds of the slow process $X^{\varepsilon,\delta,u}$}

In this section we denote positive constants as $c$'s, sometimes with subscripts indicating
dependence on other parameters, such as $c_p$ or $c_{p,\theta}$, etc. .

The following lemmas will be used in our later analysis.
\begin{lemma}
  Assume Hypothesis 1. For $i=1$ or $i=2$ and $\theta\in \mathbb{R}$ there exists a constant such that for any $t>0$, and $X , Y \in H$,
  \begin{equation} \label{Eq:Sum-of-SGQfj}
    \sum_{j=1}^\infty \left|(-A_i)^{\theta/2}S_i(t)\Sigma_i(X,Y)Q_i f_{i,j} \right|_H^2 \leq c_\theta t^{-\theta-\frac{\beta_i(\rho_i -2)}{\rho_i}} e^{-\lambda t}\|\Sigma_i(X,Y)\|_{\mathcal{L}(L^\infty(D);H)}^2.
  \end{equation}
\end{lemma}
\begin{proof}
  By assumption $Q_i f_{i,j} = \lambda_{i,j} e_{i,j}$ and $S_i(t) e_{i,k} = e^{-\alpha_{i,k}t}e_{i,k}$. Then expanding the $H$ norm with respect to the orthonormal basis $\{e_{i,k}\}$,
  \begin{align*}
    &\sum_{j=1}^\infty \left|(-A_i)^{\theta/2}S_i(t)\Sigma_i(X,Y)Q_i f_{i,j} \right|_H^2 = \sum_{j=1}^\infty \sum_{k=1}^\infty \left<(-A_i)^{\theta/2}S_i(t)\Sigma_i(X,Y)Q_i f_{i,j}, e_{i,k}\right>_H^2\\
    &=\sum_{j=1}^\infty \sum_{k=1}^\infty \left<\Sigma_i(X,Y)Q_i f_{i,j}, S_i^\star(t)(-A_i)^{\theta/2}e_{i,k}\right>_H^2\\
    &=\sum_{j=1}^\infty \sum_{i=1}^\infty \lambda_{i,j}^2 \alpha_{i,k}^{\theta}e^{-2\alpha_{i,k} t}\left<\Sigma_i(X,Y) e_{i,j},e_{i,k}\right>_H^2.
  \end{align*}
  By the H\"older inequality with exponents $\rho_i/2$ and $\rho_i/(\rho_i -2)$, the above expression is bounded by
  \begin{align} \label{Eq:doublesumsbound}
    \leq &\left(\sum_{j=1}^\infty \sum_{k=1}^\infty \lambda_{i,j}^{\rho_i} \left<\Sigma_i(X,Y) e_{i,j}, e_{i,k} \right>_H^2\right)^{2/\rho_i}\\
    &\times \left(\sum_{j=1}^\infty \sum_{k=1}^\infty \alpha_{i,k}^{\theta\rho_i/(\rho_i-2)}e^{-\frac{2\rho_i\alpha_{i,k}}{\rho_i-2}t} \left< \Sigma_i(X,Y) e_{i,j}, e_{i,k} \right>_H^2 \right)^{(\rho_i-2)/\rho_i}\nonumber\\
    \leq & \left(\sum_{j=1}^\infty \lambda_{i,j}^{\rho_i} |\Sigma_i(X,Y) e_{i,j}|_H^2 \right)^{\rho_i/2}
    \left(\sum_{k=1}^\infty \alpha_{i,k}^{\theta\rho_i/(\rho_i-2)} e^{-\frac{2\rho_i\alpha_{i,k}}{\rho_i-2}t} |\Sigma_i^\star(X,Y) e_{i,k}|_H^2 \right)^{(\rho_i -2)/\rho_i}\nonumber\\
    \leq & \left(\sum_{j=1}^\infty \lambda_{i,j}^{\rho_i} |e_{i,k}|_0^2 \right)^{\rho_i/2} \left(\sum_{k=1}^\infty \alpha_{i,k}^{\theta\rho_i/(\rho_1-2)} e^{-\frac{2 \rho_i \alpha_{i,k}}{\rho_i -2}t} |e_i,k|_0^2 \right)^{(\rho_i -2)/\rho_i} \|\Sigma_1(X,Y)\|_{\mathcal{L}(L^\infty(D),H)}^2.\nonumber
  \end{align}
  We used the fact that $\Sigma_i^\star(X,Y) = \Sigma_i(X,Y)$, which holds because for any $g, h \in L^\infty(D)$,
  \[\left<\Sigma_i(X,Y)g,h\right>_H = \int\limits_D {\sigma_i(X(x),Y(x))g(x)h(x) dx} = \left<g, \Sigma_i(X,Y)h\right>_H.\]

Because $x \mapsto x^{\beta_i}e^{-x}$ is bounded for $x>0$, there exists a constant such that $ x^{\beta_i} e^{-2x} \leq C_{\beta_i} e^{-x}$. Therefore,
\begin{align*}
  &\alpha_{i,k}^{\frac{\theta \rho_i}{\rho_i-2}} e^{-\frac{2\rho_i \alpha_{i,k}}{\rho_i-2}t}\\
   &\leq   \left(\frac{\rho_i}{\rho_i-2} \right)^{-\frac{\theta\rho_i}{\rho_i-2}} t^{-\frac{\theta \rho_i}{\rho_i-2} - \beta_i} \alpha_{i,k}^{-\beta_i} \left(\frac{\rho_i \alpha_{i,k} t}{\rho_i-2} \right)^{\frac{\theta\rho_i}{\rho_i-2} + \beta_i} e^{-\frac{2\rho_i \alpha_{i,k} t}{\rho_i-2}}\\
   &\leq C t^{-\frac{\theta\rho_i}{\rho_i-2} - \beta_i}\alpha_{i,k}^{-\beta_i}e^{-\frac{\rho_i \alpha_{i,k} t}{\rho_i-2}}.
\end{align*}
  Consequently, by (\ref{Eq:Hypothesis1Equation2}) and (\ref{Eq:doublesumsbound}),
  \begin{align} \label{Eq:SmQf-bound-with-sum}
    &\sum_{j=1}^\infty \left|(-A_i)^{\theta/2}S_i(t)\Sigma_i(X,Y)Q_i f_{i,j} \right|_H^2\nonumber\\
     &\leq c_\theta\left( \sum_{k=1}^\infty \alpha_{i,k}^{-\beta_i}|e_{i,k}|_0^2\right)^{(\rho_i -2)/\rho_i}t^{-\theta-\frac{\beta_i(\rho_i -2)}{\rho_i}} e^{-\lambda t}\|\Sigma_i(X,Y)\|_{\mathcal{L}(L^\infty(D),H)}^2.
  \end{align}
\end{proof}
\begin{remark}
  For $i=1$ or $i=2$, let $\Pi_{i,N}: H \to H$ be the projection operator in $H$ onto the span of $\{e_{i,1},... , e_{i,N}\}$. By the same arguments that we used to arrive at (\ref{Eq:SmQf-bound-with-sum}), we see that there exists a constant such that for any $N\geq1$, $t>0$, and $X,Y \in H$,
  \begin{align}\label{Eq:SmQf-tail}
    &\sum_{j=1}^\infty |(I - \Pi_{i,N})(-A_i)^{\theta/2}S_i(t)\Sigma_i(X,Y)Q_i f_{i,j}|_H^2\nonumber\\
     &\leq C  \left(\sum_{k=N+1}^\infty \alpha_{i,k}^{-\beta_i}|e_{i,k}|_0^2 \right)^{(\rho_i - 2)/\rho_i} t^{-\theta-\frac{\beta_i(\rho_i-2)}{\rho_i}} e^{-\lambda t} \|\Sigma_i(X,Y)\|_{\mathcal{L}(L^\infty(D);H)}^2.
  \end{align}
\end{remark}
\begin{lemma}
  Assume Hypothesis 1. There exists $C>0$ such that for $i=1,2$, $t>0$, $X, Y \in H$, and $u \in U$,
  \begin{equation} \label{Eq:Smu-bound}
    |(-A_i)^{\theta/2}S_i(t) \Sigma_i(X,Y)Q_iu|_H \leq C t^{-\frac{\theta}{2}-\frac{\beta_i(\rho_i -2)}{2\rho_i}} e^{-\frac{\lambda t}{2}} \|\Sigma_i(X,Y)\|_{\mathcal{L}(L^\infty(D),H)}|u|_U.
  \end{equation}
\end{lemma}
\begin{proof}
  If we expand $u$ by its Fourier series,
  \[|(-A_i)^{\theta/2}S_i(t)\Sigma_i(X,Y)Q_iu|_H= \left|\sum_{j=1}^\infty (-A_i)^{\theta/2}S_i(t) \Sigma_i(X,Y)Q_i f_{i,j} \left< u, f_{i,j}\right>_U \right|_H. \]
  By the H\"older inequality, the above expression is bounded by
  \[\left(\sum_{j=1}^\infty |(-A_i)^{\theta/2}S_i(t)\Sigma_i(X,Y)Q_if_{i,j}|_H^2 \right)^{1/2} \left(\sum_{j=1}^\infty \left<u,f_{i,j}\right>_U^2 \right)^{1/2}.\]
  It follows from \ref{Eq:Sum-of-SGQfj} and the fact that $\{f_{i,j}\}_{j \in \mathcal{N}}$ is an orthonormal subset of $U$ that
  \[|(-A_i)^{\theta/2}S_i(t)\Sigma_i(X,Y)Q_iu|_H \leq C t^{-\frac{\theta}{2}-\frac{\beta_i(\rho_i -2)}{2\rho_i}} e^{-\frac{\lambda t}{2}} \|\Sigma_i(X,Y)\|_{\mathcal{L}(L^\infty(D),H)}|u|_U.\]
\end{proof}

\begin{lemma} \label{Lm:Smu-compact}
  Assume Hypothesis 1. For any $t>0$, $i=1,2$, and $X,Y \in H$, the linear mapping $u\mapsto S_1(t)\Sigma_1(X,Y)Q_1u$ is compact.
\end{lemma}
\begin{proof}
  Let $\Pi_{i,N}$ be the projection operator in $H$ onto the span of $\{e_{i,1},...,e_{i,N}\}$. Let $\{u_n\}\subset U$ be a bounded sequence. For any fixed $N\geq 1$, $\Pi_{i,N}S_i(t)\Sigma_i(X,Y):U \to H$ is finite dimensional and bounded. By (\ref{Eq:SmQf-tail}) and (\ref{Eq:Smu-bound}),
  \begin{align*}
    &\|(I - \Pi_{i,N})S_i(t)\Sigma_i(X,Y)Q_i\|_{\mathcal{L}(U,H)} \\
    &\leq C \left( \sum_{k=N+1}^\infty \alpha_k^{-\beta_i}|e_{i,k}|_0^2\right)^{(\rho_i -2)/\rho_i}t^{-\frac{\beta_i(\rho_i-2)}{2\rho_i}} e^{-\frac{\lambda}{2} t}  \|\Sigma_i(X,Y)\|^2_{\mathcal{L}(L^\infty(D);H)}.
  \end{align*}
  Any linear operator that is a uniform limit of finite dimensional operators is compact (see for example \cite[Theorem II.4.4]{Conway}).
\end{proof}

Let us now define
$$\begin{array}{ll}
\Gamma^{\varepsilon,\delta,u}_{1}(t)&:=\displaystyle{\int_0^t S_1(t-s)\Sigma_1(X^{\varepsilon,\delta,u}(s), Y^{\varepsilon,\delta,u}(s))dW^{Q_1}(s)} \ ,
\\
\Gamma^{\varepsilon,\delta,u}_{2}(t)&:=\displaystyle{\dfrac{1}{\delta}\int_0^t S_2\left(\dfrac{t-s}{\delta^2}\right)\Sigma_2(X^{\varepsilon,\delta,u}(s), Y^{\varepsilon,\delta,u}(s))dW^{Q_2}(s)} \ .
\end{array}$$
\begin{lemma}\label{Lm:EstimateOfStochasticConvolutionSlowProcess}
Under Hypotheses 1, 2 and 3, there exists $\overline{\theta}>0$ and $\overline{p}=\frac{2}{1-\beta_{1}(\rho_{1}-2)/\rho_{1}}>2$  such that for any $\varepsilon>0$,
$T>0$, $p>\bar{p}$ and $\theta\in [0,\overline{\theta})$, we have
\begin{equation}\label{Eq:EstimateOfStochasticConvolutionSlowProcess}
\mathbf{E}\sup\limits_{t<T}|\Gamma_1^{\varepsilon,\delta,u}(t)|_{\theta,1}^p\leq c_{p,T,\theta}\int_0^T \left(1+\mathbf{E}|X^{\varepsilon,\delta,u}(s)|_H^p+\mathbf{E}|Y^{\varepsilon,\delta,u}(s)|_H^{\zeta p}\right)ds
\end{equation}
for some positive constant $c_{T,\theta}$ which is independent of $\varepsilon>0$, and
\begin{equation}\label{Eq:EstimateOfStochasticConvolutionFastProcess}
\int_0^T \mathbf{E}|\Gamma_2^{\varepsilon,\delta,u}(s)|_{\theta,2}^2ds\leq c_{T,\theta}  \mathbf{E} \int_0^T \left(1+|X^{\varepsilon,\delta,u}(s)|_H^2\right)ds \ .
\end{equation}
\end{lemma}

\begin{proof}
First we prove (\ref{Eq:EstimateOfStochasticConvolutionFastProcess}). By the It\^{o} isometry,
\begin{align*}
  &\mathbf{E} |\Gamma_2^{\varepsilon,\delta,u}(t)|_{\theta,2}^2  = \frac{1}{\delta^2}\int_{0}^{t}\sum_{j=1}^\infty |(-A_2)^{\theta/2} S_2((t-s)/\delta^2) \Sigma_2(X^{\varepsilon,\delta,u}(s),Y^{\varepsilon,\delta,u}(s))Q_2 f_j|_H^2 ds.
\end{align*}
Then by (\ref{Eq:Sum-of-SGQfj}),
\[\mathbf{E} |\Gamma_2^{\varepsilon,\delta,u}(t)|_{\theta,2}^2 \leq \frac{C}{\delta^2}\int_0^t \left(\frac{t-s}{\delta^2}\right)^{-\theta - \frac{\beta_2(\rho_2-2)}{\rho_2}} e^{-\frac{\lambda(t-s)}{\delta^2}} \left\|\Sigma_2(X^{\varepsilon,\delta,u}(s), Y^{\varepsilon,\delta,u}(s)) \right\|_{\mathcal{L}(L^\infty(D);H)}^2 ds.  \]
By Young's inequality for convolutions,
\begin{align}
\mathbf{E} \int_0^T |\Gamma_2^{\varepsilon,\delta,u}(t)|_{\theta,2}^2 dt &\leq \frac{1}{\delta^2} \left( \int_0^T \left(\frac{s}{\delta^2}\right)^{-\theta - \frac{\beta_2(\rho_2-2)}{\rho_2}} e^{-\frac{\lambda s}{\delta^2}}ds \right)\nonumber\\
&\quad\times\left(\int_0^T\left\|\Sigma_2(X^{\varepsilon,\delta,u}(s), Y^{\varepsilon,\delta,u}(s)) \right\|_{\mathcal{L}(L^\infty(D),H)}^2 ds \right). \nonumber
\end{align}
Time changing the first integral,
\[\leq \left(\int_0^\infty s^{-\theta-\frac{\beta_2(\rho_2-2)}{\rho_2}}e^{-\lambda s} ds \right) \left( \int_0^T \left\|\Sigma_2(X^{\varepsilon,\delta,u}(s),Y^{\varepsilon,\delta,u}(s))\right\|_{\mathcal{L}(L^\infty(D),H)}^2ds\right).\]
If we choose $\theta$ small enough so that $-\theta-\frac{\beta_2(\rho_2-2)}{\rho_2}>-1$ (which is possible by (\ref{Eq:Hypothesis1Equation3})), then the first integral is finite and
\[\mathbf{E} \int_0^T |\Gamma_2^{\varepsilon,\delta,u}(t)|_{\theta,2}^2 dt \leq C\int_0^T\left\|\Sigma_2(X^{\varepsilon,\delta,u}(s), Y^{\varepsilon,\delta,u}(s)) \right\|_{\mathcal{L}(L^\infty(D);H)}^2 ds. \]
The result follows by (\ref{Eq:Hypothesis2Equation2}).

Equation (\ref{Eq:EstimateOfStochasticConvolutionSlowProcess}) is similar to (4.2)
in \cite{CerraiRDEAveraging1} and is also a consequence of the stochastic factorization formula of \cite{DaPrato-Zabczyk}.
\end{proof}

The next lemma estimates the control terms
\[Z_1^{\varepsilon,\delta,u}(t) =\int_0^t S_1(t-s)\Sigma_1(X^{\varepsilon,\delta,u}(s),Y^{\varepsilon,\delta,u}(s))Q_1u(s)ds  \]
and
\[Z_2^{\varepsilon,\delta,u}(t) = \frac{1}{\delta\sqrt{\varepsilon}} \int_0^t S_2\left( \frac{t-s}{\delta^2}\right)\Sigma_2(X^{\varepsilon,\delta,u}(s),Y^{\varepsilon,\delta,u}(s))Q_2u(s)ds.\]
\begin{lemma} \label{Lm:Control-terms-estimates} 
  Under Hypotheses 1, 2 and 3, there exists $\overline{\theta}>0$ and $\overline{p}=\frac{2}{1-\beta_{1}(\rho_{1}-2)/\rho_{1}}>2$  such that for any $\varepsilon>0$,
$T>0$, $p>\overline{p}$ and $\theta\in [0,\overline{\theta}]$, we have for any $u \in \mathcal{P}^N_2$,
\begin{equation}\label{Eq:Z1-estimates}
  \mathbf{E} \sup_{t<T} |Z_1^{\varepsilon,\delta,u}(t)|_{\theta,1}^p \leq c_{p,T,\theta,N} \mathbf{E}\int_0^T \left(1 + |X^{\varepsilon,\delta,u}(s)|_H^p + |Y^{\varepsilon,\delta,u}(s)|_H^{\zeta p} \right)ds
\end{equation}
and
\begin{equation}\label{Eq:Z2-estimates}
  \mathbf{E}\int_0^T|Z_2^{\varepsilon,\delta,u}(t)|_{\theta,2}^2 ds \leq c_{p,T,\theta,N} \frac{\delta^2}{\varepsilon}\mathbf{E} \left(1 + \sup_{s \leq T}|X^{\varepsilon,\delta,u}(s)|_H^2 \right).
\end{equation}
\end{lemma}
\begin{proof}
For any $t \in [0,T]$, by (\ref{Eq:Smu-bound})
\[|Z_1^{\varepsilon,\delta,u}(t)|_{\theta,1} \leq C\int_0^t (t-s)^{-\frac{\theta}{2} - \frac{\beta_1(\rho_1-2)}{2\rho_1}} e^{-\frac{\lambda(t-s)}{2}}\|\Sigma_1(X^{\varepsilon,\delta,u}(s),Y^{\varepsilon,\delta,u}(s))\|_{\mathcal{L}(L^\infty(D),H)}|u(s)|_Uds.\]
By a H\"older inequality, along with (\ref{Eq:Sm1-growth}),
\begin{align*}
  &|Z_1^{\varepsilon,\delta,u}(t)|_{\theta,1} \\
  &\leq C|u|_{L^2([0,t];U)} \left(\int_0^t (t-s)^{-{\theta} - \frac{\beta_1(\rho_1-2)}{\rho_1}} e^{-{\lambda(t-s)}}(1 + |X^{\varepsilon,\delta,u}(s)|_H^2 + |Y^{\varepsilon,\delta,u}(s)|_H^{2\zeta})ds \right)^{\frac{1}{2}}.
\end{align*}
Applying another H\"older inequality with $p/2> \bar{p}/2$ and recalling that by assumption $|u|_{L^2([0,T];U)}\leq N^{1/2}$,
\begin{align} \label{Eq:Z_1-bound}
  |Z_1^{\varepsilon,\delta,u}(t)|_{\theta,1} \leq CN^{1/2} &\left(\int_0^t s^{-\frac{p\theta}{(p-2)} - \frac{\beta_1p(\rho_1-2)}{\rho_1(p-2)}} e^{-\frac{\lambda p s}{p-2}} ds  \right)^{\frac{p-2}{2p}}\nonumber\\
   &\times\left(\int_0^t (1 + |X^{\varepsilon,\delta,u}(s)|_H^p + |Y^{\varepsilon,\delta,u}(s)|_H^{p\zeta})ds \right)^{\frac{1}{p}}
\end{align}
The first integral is finite as long as $\frac{p\theta}{(p-2)} + \frac{\beta_1p(\rho_1-2)}{\rho_1(p-2)}<1$. By the definition of $\overline{p}$, $\frac{\beta_1p(\rho_1-2)}{\rho_1(p-2)}<1$. We then can choose $\bar{\theta}$ small enough so that the condition is satisfied.

The analysis for $Z_2^{\varepsilon,\delta,u}$ is a little bit different. By (\ref{Eq:Smu-bound}) and (\ref{Eq:Hypothesis2Equation2})
\begin{align*}
  |Z_2^{\varepsilon,\delta,u}(t)|_{\theta,2} \leq \frac{C}{\delta\sqrt{\varepsilon}} \int_0^t \left(\frac{t-s}{\delta^2} \right)^{-\frac{\theta}{2}- \frac{\beta_2(\rho_2-2)}{2\rho_2}} e^{-\frac{\lambda(t-s)}{2\delta^2}} \left( 1 + |X^{\varepsilon,\delta,u}(s)|_H\right)|u(s)|_Uds.
\end{align*}
By Young's inequality for convolutions,
\begin{align*}
  &\int_0^T|Z_2^{\varepsilon,\delta,u}(t)|_{\theta,2}^2 dt \leq \frac{C}{\delta^2\varepsilon} \left(\int_0^T \left(\frac{s}{\delta^2}\right)^{-\frac{\theta}{2} - \frac{\beta_2(\rho_2-2)}{2\rho_2}} e^{-\frac{\lambda(t-s)}{2\delta^2}}ds \right)^2 \times\nonumber\\
  &\times\left(\int_0^T \left( 1 + |X^{\varepsilon,\delta,u}(s)|_H^2 \right)|u(s)|_U^2 ds\right) \leq \frac{CN\delta^2}{\varepsilon} \left(1 + \sup_{s \leq T} |X^{\varepsilon,\delta,u}(s)|_H^2 \right).
\end{align*}
\end{proof}

\begin{lemma}\label{Lm:LpEstimateSlowAndFastProcess}
Under Hypotheses 1, 2 and 3, for any $T>0$, $p=\frac{2}{\zeta}$ and any $u\in \mathcal{P}_2^N$ for some $N\in\mathbb{N}$,
there exists a positive constant $c_{p,T,N}$ and a positive $\varepsilon_0>0$
such that for any $X_0 , Y_0 \in H$ and $0<\varepsilon<\varepsilon_0$, we have
\begin{equation}\label{Eq:LpEstimateSlowProcess}
\mathbf{E}\sup\limits_{t\in [0,T]}|X^{\varepsilon,\delta,u}(t)|_H^p\leq c_{p,T,N}(1+|X_0|_H^p+|Y_0|_H^2) \ ,
\end{equation}
\begin{equation}\label{Eq:LpEstimateFastProcess}
\int_0^T\mathbf{E}|Y^{\varepsilon,\delta,u}(t)|_H^2dt\leq c_{p,T,N}(1+|X_0|_H^2+|Y_0|_H^2) \ .
\end{equation}
\end{lemma}
\begin{proof} Let us write
$$\begin{array}{l}
\displaystyle{X^{\varepsilon,\delta,u}(t)=S_1(t)X_0+ \int_0^t S_1(t-s)B_1(X^{\varepsilon,\delta,u}(s), Y^{\varepsilon,\delta,u}(s))ds
}\\
\displaystyle{\qquad \qquad \qquad \qquad \qquad +Z_1^{\varepsilon,\delta,u}(t)+\sqrt{\varepsilon}\Gamma_1^{\varepsilon,\delta,u}(t) \ .}
\end{array}$$

By the growth conditions on $B_1$ (\ref{Eq:Sm1-growth}), and the boundedness of the semigroup,
$$\left|\int_0^t S_1(t-s)B_1(X^{\varepsilon,\delta,u}(s), Y^{\varepsilon,\delta,u}(s))ds\right|_H^p
\leq c_{p,T} \int_0^t (1+|X^{\varepsilon,\delta,u}(s)|_H^p+|Y^{\varepsilon,\delta,u}(s)|_H^{\zeta p} ) ds \ .$$

Thus by using (\ref{Eq:EstimateOfStochasticConvolutionSlowProcess}) and (\ref{Eq:Z1-estimates}) with $\theta=0$, we can conclude with
\begin{align*}
\mathbf{E}\sup\limits_{s\leq t}|X^{\varepsilon,\delta,u}(s)|_H^p&\leq c_{p,T,N}(1+|X_0|_H^p)
+c_{p,T,N}\int_0^t \mathbf{E}|Y^{\varepsilon,\delta,u}(s)|_H^{\zeta p}ds\nonumber\\
&\qquad+c_{p,T,N}\int_0^t \left(1+\mathbf{E}\sup\limits_{r\leq s}|X^{\varepsilon,\delta,u}(r)|_H^p\right)ds \ ,
\end{align*}
so that by Gr\"onwall's inequality we have
\begin{equation}\label{Eq:SlowProcessControlledByFast}
\begin{array}{ll}
\mathbf{E}\sup\limits_{s\leq t}|X^{\varepsilon,\delta,u}(s)|_H^p & \leq \displaystyle{c_{p,T,N}\left(1+|X_0|_H^p+\int_0^t \mathbf{E}|Y^{\varepsilon,\delta,u}(s)|_H^{\zeta p}ds\right)}
\\
& \displaystyle{=c_{p,T,N}\left(1+|X_0|_H^p+\int_0^t \mathbf{E}|Y^{\varepsilon,\delta,u}(s)|_H^2 ds\right) \ ,}
\end{array}
\end{equation}
where we chose $p =2/\zeta$. Next we want to estimate
$$\int_0^t \mathbf{E}|Y^{\varepsilon,\delta,u}(s)|_H^2ds \  .$$
We will be using the assumptions from Hypothesis 2, in particular that $\Sigma_2(X,Y)$ does not grow with respect to $Y$.

Set $\Lambda_2^{\varepsilon,\delta,u}(t):=Y^{\varepsilon,\delta,u}(t)-Z_2^{\varepsilon,\delta,u}(t)-\Gamma_2^{\varepsilon,\delta,u}(t)$, we have $\Lambda_2^{\varepsilon,\delta,u}(0)=Y_0$ and
 $\Lambda_2^{\varepsilon,\delta,u}$ is weakly differentiable in time and
\[\frac{d}{dt}\Lambda_2^{\varepsilon,\delta,u}(t) = \frac{1}{\delta^2}A_2 \Lambda_2^{\varepsilon,\delta,u}(t) + \frac{1}{\delta^2}B_2(X^{\varepsilon,\delta,u}(t),Y^{\varepsilon,\delta,u}(t)).\]
Therefore,
\begin{align*}
& \dfrac{1}{2}\dfrac{d}{dt}|\Lambda_2^{\varepsilon,\delta,u}(t)|_H^2
= \left<\dfrac{d}{dt}\Lambda_2^{\varepsilon,\delta,u}(t), \Lambda_2^{\varepsilon,\delta,u}(t)\right>_H
\\
&\leq  \dfrac{1}{\delta^2}\langle A_2\Lambda_2^{\varepsilon,\delta,u}(t), \Lambda_2^{\varepsilon,\delta,u}(t)\rangle_H + \dfrac{1}{\delta^2}\langle B_2(X^{\varepsilon,\delta,u}(t), Z_2^{\varepsilon,\delta,u}(t)+ \Gamma_2^{\varepsilon,\delta,u}(t)), \Lambda_2^{\varepsilon,\delta,u}(t)\rangle_H
\\
&  \quad + \dfrac{1}{\delta^2}\left< B_2(X^{\varepsilon,\delta,u}(t), \Lambda_2^{\varepsilon,\delta,u}(t)+Z_2^{\varepsilon,\delta,u}(t)+\Gamma_2^{\varepsilon,\delta,u}(t))
-\right.\\
&\hspace{3cm}\left. -B_2(X^{\varepsilon,\delta,u}(t), Z_2^{\varepsilon,\delta,u}(t) + \Gamma_2^{\varepsilon,\delta,u}(t)), \Lambda_2^{\varepsilon,\delta,u}(t)\right>_H
\\
&\leq -\dfrac{1}{\delta^2} \left(\dfrac{\lambda - L_{b_2}^Y}{2} \right)|\Lambda_2^{\varepsilon,\delta,u}(t)|_H^2
+ \dfrac{c}{\delta^2} \left(1 + |X^{\varepsilon,\delta,u}(t)|_H^2 + |Z_2^{\varepsilon,\delta,u}(t) + \Gamma_2^{\varepsilon,\delta,u}(t)|_H^2\right).
\end{align*}

Here the last inequality is due to Young's inequality. By a comparison principle, letting $\rho=\dfrac{\lambda-L_{b2}^X}{2}$, we have
\begin{align}
&|\Lambda_2^{\varepsilon,\delta,u}(t)|_H^2 \nonumber\\
\leq & e^{-\rho t/\delta^2} |Y_0|_H^2 + \displaystyle{\frac{c}{\delta^2} \int_0^t e^{-\rho(t-s)/\delta^2} \left(1 + |X^{\varepsilon,\delta,u}(s)|_H^2 + |Z^{\varepsilon,\delta,u}_2(s)|_H^2 + |\Gamma_2^{\varepsilon,\delta,u}(s)|_H^2 \right) ds}\nonumber
\end{align}

By Young's inequality for convolutions,
\begin{align*}
  \int_0^T |\Lambda_2^{\varepsilon,\delta,u}(t)|_H^2 dt \leq &c\delta^2 |Y_0|_H^2 + c\int_0^T \left(1 + |X^{\varepsilon,\delta,u}(t)|_H^2 + |Z^{\varepsilon,\delta,u}(t)|_H^2 + |\Gamma_2^{\varepsilon,\delta,u}(t)|_H^2 \right)dt \ .
\end{align*}

Thus using (\ref{Eq:EstimateOfStochasticConvolutionFastProcess}) and (\ref{Eq:Z2-estimates})
we see that
\begin{equation}\label{Eq:FastProcessControlledBySlow}
\int_0^T \mathbf{E} |Y^{\varepsilon,\delta,u}(t)|_H^2dt\leq c |Y_0|_H^2 + c+c\int_0^T \mathbf{E}\sup\limits_{r\leq t}|X^{\varepsilon,\delta,u}(r)|_H^2 dt+cN\frac{\delta}{\sqrt{\varepsilon}}\mathbf{E}  \sup_{t\leq T}|X^{\varepsilon,\delta,u}(t)|_H^2.
\end{equation}

Combining (\ref{Eq:FastProcessControlledBySlow}) and (\ref{Eq:SlowProcessControlledByFast}) we see that
$$\mathbf{E}\sup\limits_{s\leq t}|X^{\varepsilon,\delta,u}(s)|_H^p\leq c_{T,N}(1+|X_0|_H^p+|Y_0|_H^2)
+c_{T}\left(1 + N\frac{\delta}{\sqrt{\varepsilon}}  \right) \mathbf{E} \sup\limits_{r\leq T}|X^{\varepsilon,\delta,u}(r)|_H^2 ds \ .$$
By Young's inequality,
$$c_{T}\left(1 + N \frac{\delta}{\sqrt{\varepsilon}} \right)|X^{\varepsilon,\delta,u}(t)|_H^2 \leq \frac{1}{2}|X^{\varepsilon,\delta,u}(t)|_H^p + c_{T}\left(1 + N\frac{\delta}{\sqrt{\varepsilon}} \right)^{(p-2)/p}$$
where the constant on the right hand side is different from the constant on the left.
We have assumed in (\ref{Eq:AssumptionRegime1Restricted}) that $\dfrac{\delta}{\sqrt{\varepsilon}} \to 0$. Consequently, (\ref{Eq:LpEstimateSlowProcess}) follows.
By (\ref{Eq:LpEstimateSlowProcess}) and
(\ref{Eq:FastProcessControlledBySlow})
we obtain (\ref{Eq:LpEstimateFastProcess}). \end{proof}

\begin{lemma}\label{Lm:SobolevEstimateSlowProcess}
Under Hypotheses 1, 2 and 3,
there exists $0<\overline{\theta}< \frac{1-\zeta}{2}$ and $p=\frac{2}{\zeta}$, such that
 for any $u\in \mathcal{P}_2^N$,  $T>0$,  $X_0\in H$
 and $Y_0\in H$ we have
\begin{equation}\label{Eq:SobolevEstimateSlowProcess}
\sup\limits_{\varepsilon\in (0,1)}\mathbf{E}\sup\limits_{t\leq T}|X^{\varepsilon,\delta,u}(t) - S_1(t)X_0|_{\theta,1}^p\leq c_{p,\theta,T,N}(1+|X_0|_H^p+|Y_0|_H^2)
\end{equation}
for some positive constant $c_{p,\theta,T,N}$.
\end{lemma}
\begin{proof}  Assume that $X_0\in H$.
We have
$$\begin{array}{l}
\displaystyle{X^{\varepsilon,\delta,u}(t)-S_1(t)X_0= \int_0^t S_1(t-s)B_1(X^{\varepsilon,\delta,u}(s), Y^{\varepsilon,\delta,u}(s))ds
}\\
\displaystyle{\qquad \qquad \qquad \qquad \qquad +Z^{\varepsilon,\delta,u}_1(t)
+\sqrt{\varepsilon}\Gamma_1^{\varepsilon,\delta,u}(t) \ .}
\end{array}$$

We showed that $\Gamma_1^{\varepsilon,\delta,u}$ and $Z_1^{\varepsilon,t,u}$ have the required regularity in Lemmas \ref{Lm:EstimateOfStochasticConvolutionSlowProcess} and \ref{Lm:Control-terms-estimates} and that
\[\mathbf{E}\sup_{t \in [0,T]}\left(|\Gamma_1^{\varepsilon,\delta,u}(t)|_{\theta,1}^p + |Z_1^{\varepsilon,\delta,u}(t)|_{\theta,1}^p\right) \leq C( 1  + \mathbf{E}\sup_{t \in [0,T]}|X^{\varepsilon,\delta,u}(t)|_H^p).\]
By the Lipschitz continuity of $B_1$ and the regularizing properties of the semigroup,
\[\sup_{t \in [0,T]}\left|\int_0^t S_1(t-s) B(X^{\varepsilon,\delta,u}(s),Y^{\varepsilon,\delta,u}(s))ds \right|_{\theta,1}^p \leq C(1 + \mathbf{E}\sup_{t \in [0,T]}|X^{\varepsilon,\delta,u}(t)|_H^p).\]
The result then follows from (\ref{Eq:LpEstimateSlowProcess}).
\end{proof}

\begin{lemma} \label{Lm:HighRegFastProcess}
  There exists $\theta>0$ such that for any $T>0$, there exists a constant $C_{T,N,\theta}>0$ such that for any $u \in \mathcal{P}_2^N$, $Y_0, X_0 \in H$
  \begin{equation} \label{Eq:HighRegFastProcess}
    \int_0^T |Y^{\varepsilon,\delta,u}(s)|_{\theta,2}^2 ds \leq C_{T,N,\theta}(1 + |Y_0|_H^2 + |X_0|_H^2).
  \end{equation}
  Notice that these bounds are independent of $\varepsilon$ and $\delta$.
\end{lemma}
\begin{proof}
  This proof is a consequence of the analytic properties of the semigroup $|S_i(t)X_0|_{\theta,i} \leq Ct^{-\theta}|X_0|_H$. The mild formulation for $Y^{\varepsilon,\delta,u}$ is
  \begin{align*}
    Y^{\varepsilon,\delta,u}(t) = &S_2 \left( \frac{t}{\delta^2} \right)Y_{0} + \frac{1}{\delta^2}\int_0^t S_2 \left(\frac{t-s}{\delta^2} \right) B(X^{\varepsilon,\delta,u}(s), Y^{\varepsilon,\delta,u}(s))ds\\
    &+Z_2^{\varepsilon,\delta,u}(t) + \Gamma_2^{\varepsilon,\delta,u}(t).
  \end{align*}
  We bound each term of the mild solution separately.
  The semigroup term satisfies
  \begin{equation} \label{Eq:SemigroupL2}
    \int_0^T \left| S_2\left(\frac{t}{\delta^2}\right)Y_0\right|_{\theta,2}^2dt\leq C \int_0^T \left(\frac{t}{\delta^2}\right)^{-\theta}e^{-\lambda t/\delta^2}|Y_0|_H^2 dt \leq \delta^2 C_\theta |Y_0|_H^2.
  \end{equation}
  Denote the drift term
  \[\Lambda_2(t) = \frac{1}{\delta^2}\int_0^t S_2 \left(\frac{t-s}{\delta^2} \right) B(X^{\varepsilon,\delta,u}(s), Y^{\varepsilon,\delta,u}(s))ds.\]
  Then
  \[|\Lambda_2(t)|_{\theta,2} \leq C \frac{1}{\delta^2} \int_0^t \left( \frac{t-s}{\delta^2} \right)^{-\theta/2} e^{-\frac{\lambda(t-s)}{2\delta^2}} |B(X^{\varepsilon,\delta,u}(s),Y^{\varepsilon,\delta,u}(s))|_H ds \]
  and by Young's inequality for convolutions and the linear growth of $B_2$,
  \begin{align} \label{Eq:DriftL2}
    \int_0^T |\Lambda_2(t)|_{\theta,2}^2 dt \leq &C \left(\int_0^\infty s^{-\theta/2}e^{-\lambda s/2}ds \right) \int_0^T \left( 1 + |X^{\varepsilon,\delta,u}(s)|_H^2 + |Y^{\varepsilon,\delta,u}(s)|_H^2  \right)ds.
  \end{align}

  We combine estimates (\ref{Eq:SemigroupL2}), (\ref{Eq:DriftL2}),  along with (\ref{Eq:EstimateOfStochasticConvolutionFastProcess}) and (\ref{Eq:Z2-estimates})  for estimating $\Gamma_2^{\varepsilon,\delta,u}$ and $Z^{\varepsilon,\delta,u}_2$  to see that
  \begin{align*}
    \mathbf{E}\int_0^T |Y^{\varepsilon,\delta,u}(t)|_{\theta,2}^2 dt \leq &C_{T,N,\theta}\mathbf{E} \left( |Y_0|_H^2 + \sup_{s \leq T}|X^{\varepsilon,\delta,u}(s)|_H^2 + \int_0^T |Y^{\varepsilon,\delta,u}(s)|_H^2 ds \right).\\
  \end{align*}
  It follows from (\ref{Eq:LpEstimateSlowProcess}) and (\ref{Eq:LpEstimateFastProcess}) that
  \[\mathbf{E}\int_0^T |Y^{\varepsilon,\delta,u}(t)|_{\theta,2}^2 dt \leq C_{T,N,\theta} \left(1 + |X_0|_H^2 + |Y_0|_H^2 \right).\]
\end{proof}

\begin{lemma}\label{Lm:OscillationEstimateSlowProcess}
Under Hypotheses 1, 2 and 3, there exists $0<\bar{\theta}<\frac{1-\zeta}{2}$ and $p=\frac{2}{\zeta}$,
such that for any $u\in \mathcal{P}_2^N$, $T>0$, $X_0\in H$
and $Y_0\in H$ it holds
\begin{align*}
&\sup\limits_{\varepsilon\in (0,1]}\mathbf{E}|X^{\varepsilon,\delta,u}(t)-X^{\varepsilon,\delta,u}(s)|_H^p\\
&\leq c_{\theta,p,T,N}\left(|t-s|^{\beta(\theta)p}(|X_0|_H^p+|Y_0|_H^2+1) + |(S_1(t-s) - I)X_0|_H^p\right).
\end{align*}
for $s,t\in [0,T]$ and some positive constant $c_{\theta,p,T,N}$ and $\beta(\theta)>0$.
\end{lemma}
\begin{proof} We can proceed in a similar way as in the proof of Proposition 4.4 of \cite{CerraiRDEAveraging1},
but we have to take into account the control $u$. For any $t,h\geq 0$ and $t, t+h\in [0,T]$, we have
\begin{align*}
X^{\varepsilon,\delta,u}(t+h)&-X^{\varepsilon,\delta,u}(t)=(S_1(h)-I)(X^{\varepsilon,\delta,u}(t) - S_1(t)X_0) + (S_1(h) - I)S_1(t)X_0\\
&+\int_t^{t+h} S_1(t+h-s)B_1(X^{\varepsilon,\delta,u}(s), Y^{\varepsilon,\delta,u}(s))ds
\\
&+\int_t^{t+h} S_1(t+h-s)\Sigma_1(X^{\varepsilon,\delta,u}(s), Y^{\varepsilon,\delta,u}(s))Q_1u(s)ds
\\
&+\sqrt{\varepsilon} \int_t^{t+h} S_1(t+h-s)\Sigma_1(X^{\varepsilon,\delta,u}(s), Y^{\varepsilon,\delta,u}(s))dW^{Q_1}(s) \ .
\end{align*}
We can then argue more or less in the same way as the proof of Proposition 4.4 in \cite{CerraiRDEAveraging1}. The equicontinuity of the integral terms is due to the regularizing properties of $S_1(t)$ along with the a-priori estimates of Lemma \ref{Lm:LpEstimateSlowAndFastProcess}.
For example, H\"older estimates such as (\ref{Eq:Z_1-bound}) with $\theta=0$ show that we have uniform continuity as $h$ goes to zero. The stochastic integral term requires a stochastic factorization argument. The H\"older continuity of the $(S_1(h) -I)(X^{\varepsilon,\delta,u}(t)-S_1(t)X_0)$ as $h \to 0$ is due to the fact that (\ref{Eq:SobolevEstimateSlowProcess}) holds and $H^\theta_1$ is compactly embedded in $H$.
 \end{proof}

\subsubsection{Tightness of the pair $\{(X^{\varepsilon,\delta,u}, \mathrm{P}^{\varepsilon,\Delta}), \varepsilon>0, 0\leq t \leq T\}$}\label{SSS:TightnessViablePair}
\begin{lemma}\label{Lm:TightnessSlowProcess}
Under Hypotheses 1, 2 and 3, for any $T>0$ and $X_0\in H$ and any $Y_0\in H$,
the family of processes $\{X^{\varepsilon,\delta,u}: \varepsilon \in (0,1), u \in \mathcal{P}_2^N\}$ is tight in $C([0,T]; H)$.
\end{lemma}
\begin{proof}
  We apply an Arzela-Ascoli argument to show that $$\{X^{\varepsilon,\delta,u}(\cdot) - S_1(\cdot)X_0: \varepsilon \in (0,1), u \in \mathcal{P}_2^N\}$$ is tight by using Lemmas \ref{Lm:SobolevEstimateSlowProcess} and \ref{Lm:OscillationEstimateSlowProcess}. Therefore, $\{X^{\varepsilon,\delta,u}: \varepsilon \in (0,1), u \in \mathcal{P}_2^N\}$ is also tight because the set differs by a fixed non-random trajectory.
\end{proof}

\begin{lemma}\label{Lm:TightnessOccupationMeasure}
Under Hypotheses 1, 2 and 3, for any $T>0$, $X_0\in H$ and any $Y_0\in H$
the family of measures ${\{\mathrm{P}^{\varepsilon,\Delta}: \varepsilon \in (0,1), u \in \mathcal{P}_2^N\}}$ is tight in $\mathscr{P}(U\times \mathcal{Y}\times [0,T])$, where $U\times \mathcal{Y} \times [0,T]$ is endowed with the weak topology on $U$, the norm topology on $\mathcal{Y}$ and the standard topology on $[0,T]$.
\end{lemma}
\begin{proof}
  We use tightness functions (see \cite[Appendix A.3]{DupuisEllis}). For $\theta>0$ satisfying Lemma \ref{Lm:HighRegFastProcess}, let $g: U \times \mathcal{Y} \times [0,T] \to \mathbb{R}$ be defined by
  \[g(u,Y,t) = |u|_U^2 + |Y|_{\theta,2}^2.\]
  If $M>0$, then the set $\{u \in U: |u|_U^2 \leq M\}$ is compact in the weak topology on $U$ by Alaoglu's Theorem. The set $\{Y \in \mathcal{Y}: |Y|_{\theta,2} \leq M\}$ is compact in $\mathcal{Y}$ because the operator $A_2$ is unbounded.
  Let $E = U \times \mathcal{Y} \times [0,T]$ be the metric space endowed with the weak toplogy on $U$ times the norm topology on $\mathcal{Y}$ times the topology on $[0,T]$.
  The function $g$ is a tightness function in $E$ because the set
  \[\{(u,Y,t): g(u,Y,t)\leq M\} \subset \{u:|u|_U^2 \leq M\} \times \{Y: |Y|_{\theta,2}^2\leq M\} \times [0,T]\]
  is precompact.

  By applying Theorem A.3.17 of \cite{DupuisEllis}, we see that the function $G: \mathscr{P}(E) \to \mathbb{R}$ given by
  \[G(\nu) = \int\limits_E g(x) d\nu(x)\]
  is a tightness function on $\mathscr{P}(E)$.
  Applying Theorem A.3.17 of \cite{DupuisEllis} again, we see that the function $\mathscr{G}: \mathscr{P}(\mathscr{P}(E)) \to \mathbb{R}$ given by
  \[\mathscr{G}(\mu) = \int\limits_{\mathscr{P}(E)} G(\nu) d\mu(\nu)\]
  is a tightness function.   If we choose a sequence of controls $u^\varepsilon \in \mathcal{P}_2^N$,
  then by (\ref{Eq:HighRegFastProcess}) and (\ref{Eq:OccupationMeasureBeforeAveraging}), letting $\nu^{\varepsilon,\Delta}$ denote the law of $\mathrm{P}^{\varepsilon,\Delta}$,
  {\begin{align*}
    \sup_{0<\varepsilon<1} \mathscr{G}(\nu^{\varepsilon,\Delta}) &= \sup_{0<\varepsilon<1}\mathbf{E} \left[G\left(\mathrm{P}^{\varepsilon,\Delta}\right)\right]
    = \sup_{0<\varepsilon<1}\mathbf{E} \left[\int\limits_{U\times\mathcal{Y}\times[0,T]} (|u|_U^2 + |Y|_{\theta,2}^2) \mathrm{P}^{\varepsilon,\Delta}(dudYdt) \right] \\
    &= \sup_{0<\varepsilon<1}\mathbf{E} \int_0^T \frac{1}{\Delta} \int_t^{t+\Delta} \left(|u^\varepsilon(s)|_U^2 + |Y^{\varepsilon,\delta,u}(s)|_{\theta,2}^2 \right)dsdt\\
    &\leq \sup_{0<\varepsilon<1}\mathbf{E} \int_0^{T+\Delta} \left( |u^\varepsilon(s)|_U^2 + |Y^{\varepsilon,\delta,u}(s)|_{\theta,2}^2\right)ds <+\infty.
  \end{align*}}
  Since $\mathscr{G}$ is a tightness function, the laws of $\mathrm{P}^{\varepsilon,\Delta}$ are tight.
\end{proof}

\begin{lemma}\label{Lm:UniformIntegrabilityOccupationMeasure}
Under Hypotheses 1 and 2, the family $\{\mathrm{P}^{\varepsilon,\Delta}(dudYdt)=\mathrm{P}^{\varepsilon,\Delta}_t(dudY)dt \ , \ \varepsilon>0\}$ is uniformly integrable
in the sense that
$$\lim\limits_{M\rightarrow\infty}\sup\limits_{\varepsilon>0}\mathbf{E}_{X_0}\left[\int\limits_{\{(u,Y,t): |u|_U> M , |Y|_{\theta,2}>M\}}(|u|_U + |Y|_{\theta,2})
\mathrm{P}^{\varepsilon,\Delta}(dudYdt)\right] = 0 \ .$$
\end{lemma}

The proof of Lemmas \ref{Lm:TightnessOccupationMeasure} and \ref{Lm:UniformIntegrabilityOccupationMeasure} follows the strategy of the proof of \cite[Proposition 3.1]{LDPWeakConvergence}, where here we also need to use the additional bound (\ref{Eq:LpEstimateFastProcess}) since in this paper the fast process does not take values in a bounded space.

With Lemmas \ref{Lm:TightnessSlowProcess} and \ref{Lm:TightnessOccupationMeasure} and the Prokhorov Theorem, we infer that for any sequence $\varepsilon\rightarrow 0$,
there exist a subsequence along which $(X^{\varepsilon,\delta,u}, \mathrm{P}^{\varepsilon,\Delta})\rightarrow (\bar{X}, \mathrm{P})$ in the space $C([0,T]; H)\times \mathscr{P}(U\times \mathcal{Y}\times[0,T])$.
The next two sections show that any such accumulation point $(\bar{X}, \mathrm{P})$ is a viable pair
in the sense of Definition \ref{Def:ViablePair}.

\subsection{Step 2: Proof of (\ref{Eq:ViablePairAveragingDynamics})}\label{S:LLNSlowProcess}
Let us recall the mild solution $(X^{\varepsilon,\delta,u}, Y^{\varepsilon,\delta,u})$ to the controlled problem (\ref{Eq:MildFastSlowStochasticRDEWithControl}). In particular, let us write for the slow component
\begin{align}\label{Eq:MildFastSlowStochasticRDEWithControl2}
X^{\varepsilon,\delta,u}(t)&=S_1(t)X_0+ \int_0^t S_1(t-s)B_1(X^{\varepsilon,\delta,u}(s), Y^{\varepsilon,\delta,u}(s))ds\nonumber\\
&\quad+\int_0^t S_1(t-s)\Sigma_1(X^{\varepsilon,\delta,u}(s), Y^{\varepsilon,\delta,u}(s))Q_1 u(s) ds\nonumber\\
&\quad +\sqrt{\varepsilon}\int_0^t S_1(t-s)\Sigma_1(X^{\varepsilon,\delta,u}(s), Y^{\varepsilon,\delta,u}(s))dW^{Q_1}(s) =\sum_{i=1}^{4}J^{\varepsilon,\delta,u}_{i}(t),
\end{align}
where $J^{\varepsilon,\delta,u}_{i}(t)$ represents the $i^{\text{th}}$ term on the right hand side of (\ref{Eq:MildFastSlowStochasticRDEWithControl2}).

Our goal is to show that each one of the terms $J^{\varepsilon,\delta,u}_{i}(t)$ is tight in $C([0,T]; H)$ and to identify its limit.  To prove the tightness of paths, we apply an infinite dimensional version of the Arzela-Ascoli Theorem. The Arzela-Ascoli Theorem guarantees that the sets
{\[K_{\theta,\theta_{1},M} = \left\{\varphi \in C([0,T];H): \sup_{t \in [0,T]}|\varphi(t)|_{\theta,1} \leq M, \ \ \sup_{\substack{s,t \in [0,T]\\t\not = s}} \frac{|\varphi(t)-\varphi(s)|_H}{|t-s|^{\theta_{1}}} \leq M \right\}\]}
are compact subsets of $C([0,T];H)$. Such a set consists of equicontinuous paths which live in a compact subset of $H$.

We show that the paths of $J^{\varepsilon,\delta,u}_i$ are tight by proving that they live in sets like $K_{\theta,\theta_{1},M}$ with high probability uniformly with respect to $\varepsilon,\delta$ and $u$.

The term $J^{\varepsilon,\delta,u}_{1}(t)=S_1(t)X_{0}$ is non-random and doesn't depend on $\varepsilon$. Using the same arguments as in the proof of Lemma \ref{Lm:EstimateOfStochasticConvolutionSlowProcess} and Lemma \ref{Lm:Control-terms-estimates}, we can show that for $i=2,3$
\begin{align*}
\sup_{\varepsilon\in(0,1)}\mathbf{E}\sup_{t\in[0,T]}|J^{\varepsilon,\delta,u}_{i}(t)|_{\theta,1}^{p}&<\infty, \text{ with }  p=2/\zeta.
\end{align*}

At the same time, Doob's inequality and Lemmas \ref{Lm:EstimateOfStochasticConvolutionSlowProcess}, \ref{Lm:SobolevEstimateSlowProcess} and \ref{Lm:HighRegFastProcess} give
\begin{align*}
\sup_{\varepsilon\in(0,1)}\mathbf{E}\sup_{t\in[0,T]}|J^{\varepsilon,\delta,u}_{4}(t)|_{\theta,1}^{2}&<\infty.
\end{align*}

Hence, we obtain that for $0<\bar{\theta}<\frac{1-\zeta}{2}$ and for any $\theta\in(0,\bar{\theta}]$, we have for $i=2,3,4$
\begin{align*}
\lim_{M\rightarrow\infty}\sup_{\varepsilon\in(0,1)}\mathbf{P}\left(\sup_{t\in[0,T]}|J^{\varepsilon,\delta,u}_{i}(t)|_{\theta,1}^{p}>M\right)&=0.
\end{align*}

The equicontinuity of the $J^{\varepsilon,\delta,u}_i$ paths is a consequence of Lemma \ref{Lm:OscillationEstimateSlowProcess} and the Kolmogorov continuity criterion.

The latter implies that the terms $J^{\varepsilon,\delta,u}_{i}(t)$ are indeed tight in $\mathcal{C}([0,T]; H)$. By Lemmas \ref{Lm:TightnessSlowProcess} and \ref{Lm:TightnessOccupationMeasure}, we also know that the family  $\{(X^{\varepsilon,\delta,u}(t), \mathrm{P}^{\varepsilon,\Delta}), \varepsilon>0, 0\leq t \leq T\}$ is also tight. Therefore, we can extract a subsequence along which  $J^{\varepsilon,\delta,u}_{i}(\cdot)$  and $(X^{\varepsilon,\delta,u}(\cdot), \mathrm{P}^{\varepsilon,\Delta})$ converge in distribution. Let us denote by $\bar{J}_{i}(\cdot)$  and $(\bar{X}(\cdot), \mathrm{P})$ the corresponding limits. Our next goal is to identify them.

We know that $\bar{J}_{1}(t)=S_1(t)X_0$. Also, the bounds of Lemma \ref{Lm:LpEstimateSlowAndFastProcess} guarantee that $\bar{J}_{4}(t)=0$ for all $t\in[0,T]$. It remains to identify $\bar{J}_{i}(t)$ for $i=2,3$. At this point, we will use Skorokhod representation theorem (Theorem 1.8 in \cite{EithierKurtz1986}), which, for the purposes of identifying the limit, allows us to assume
that the aforementioned convergence holds with probability one. The Skorokhod representation theorem involves the introduction of another probability space, but this distinction is ignored in the
notation.

Let us present the argument only for $\bar{J}_{3}(t)$ as the argument for $\bar{J}_{2}(t)$ is the same but simpler. Because we have proved tightness, we know that $J_i^{\varepsilon,\delta,u}$ all converge in $C([0,T];H)$ to a limit. In order to identify the limit, it is sufficient to identify the pointwise limits of $J^{\varepsilon,\delta,u}(t)$ for any $t \in [0,T]$. We have that
\begin{align}
{J}^{\varepsilon,\delta,u}_{3}(t)&=\int_0^t S_1(t-s)\Sigma_1(X^{\varepsilon,\delta,u}(s), Y^{\varepsilon,\delta,u}(s))Q_1 u(s) ds\nonumber\\
&=\int_{U\times\mathcal{Y}\times[0,t]} S_1(t-s)\Sigma_1(X^{\varepsilon,\delta,u}(s), Y)Q_1 u \mathrm{P}^{\varepsilon,\Delta}(dudYds)\nonumber\\
&\quad+\left( \int_0^t S_1(t-s)\Sigma_1(X^{\varepsilon,\delta,u}(s), Y^{\varepsilon,\delta,u}(s))Q_1 u(s) ds-\right.\nonumber\\
&\hspace{1cm}\left.-\int_{U\times\mathcal{Y}\times[0,t]} S_1(t-s)\Sigma_1(X^{\varepsilon,\delta,u}(s), Y)Q_1 u \mathrm{P}^{\varepsilon,\Delta}(dudYds)\right)\label{Eq:J3term}
\end{align}

By (\ref{Eq:ConvergenceOccupationMeasureSmu}) below, the first term on the right hand side of (\ref{Eq:J3term}) satisfies
\begin{align*}
&\int_{U\times\mathcal{Y}\times[0,t]} S_1(t-s)\Sigma_1(X^{\varepsilon,\delta,u}(s), Y)Q_1 u \mathrm{P}^{\varepsilon,\Delta}(dudYds)\nonumber\\
&\qquad\qquad\rightarrow \int_{U\times\mathcal{Y}\times[0,t]} S_1(t-s)\Sigma_1(\bar{X}(s), Y)Q_1 u \mathrm{P}(dudYds).
\end{align*}

By (\ref{Eq:ConvergenceFastMotion1}) and (\ref{Eq:ConvergenceFastMotion2}) we have the second term of (\ref{Eq:J3term}) converges to zero in probability. Therefore, we have shown that any limit $(\bar{X}, \mathrm{P})$ of $(X^{\varepsilon,\delta,u}, \mathrm{P}^{\varepsilon,\Delta})$ solves (\ref{Eq:ViablePairAveragingDynamics}).
\begin{lemma}
\label{L:LemmaForODElimit1} Let $t\in [0,T]$ be given.
Assume that $(X^{\varepsilon,\delta,u}, \mathrm{P}^{\varepsilon,\Delta})\rightarrow (\bar{X}, \mathrm{P})$ in distribution in $C([0,T];H)\times\mathscr{P}(E)$ for some subsequence of $\varepsilon\downarrow 0$,
and Hypotheses 1, 2 and 3 hold. Then the following limits are valid in distribution along this subsequence:
\begin{align}\label{Eq:ConvergenceOccupationMeasureB}
&\int_{U\times \mathcal{Y}\times [0,t]}S_{1}(t-s)B_1(X^{\varepsilon,\delta,u}(s), Y)\mathrm{P}^{\varepsilon,\Delta}(dudYds)\nonumber\\
&\qquad\qquad\rightarrow
\int_{U\times \mathcal{Y}\times [0,t]}S_{1}(t-s)B_1(\bar{X}(s), Y)\mathrm{P}(dudYds),
\end{align}
and
\begin{align}\label{Eq:ConvergenceOccupationMeasureSmu}
&\int_{U\times \mathcal{Y}\times [0,t]}S_{1}(t-s)\Sigma_1(X^{\varepsilon,\delta,u}(s), Y)Q_1 u\mathrm{P}^{\varepsilon,\Delta}(dudYds)\nonumber\\
&\qquad\qquad\rightarrow
\int_{U\times \mathcal{Y}\times [0,t]}S_{1}(t-s)\Sigma_1(\bar{X}(s), Y)Q_1u\mathrm{P}(dudYds).
\end{align}
\end{lemma}
\begin{lemma} \label{L:LemmaForODElimit2}
  Let $t\in[0,T]$ be given.
Assume that $(X^{\varepsilon,\delta,u}, \mathrm{P}^{\varepsilon,\Delta})\rightarrow (\bar{X}, \mathrm{P})$ in distribution in $C([0,T];H)$ for some subsequence of $\varepsilon\downarrow 0$,
and Hypotheses 1, 2 and 3 hold. Then the following limits are valid in distribution along this subsequence:
\begin{align}\label{Eq:ConvergenceFastMotion1}
&\int_0^{t}S_{1}(t-s)B_1(X^{\varepsilon,\delta,u}(s), Y^{\varepsilon,\delta,u}(s))ds\nonumber\\
&\qquad\qquad-\int_{U\times\mathcal{Y}\times [0, t]}
S_{1}(t-s)B_1(X^{\varepsilon,\delta,u}(s), Y)\mathrm{P}^{\varepsilon,\Delta}(dudYds) \rightarrow 0 \ ,
\end{align}
and
\begin{align}\label{Eq:ConvergenceFastMotion2}
&\int_0^{t}S_{1}(t-s)\Sigma_1(X^{\varepsilon,\delta,u}(s), Y^{\varepsilon,\delta,u}(s))Q_1 u(s)ds-\nonumber\\
&\qquad\qquad-\int_{U\times\mathcal{Y}\times [0, t]}
S_{1}(t-s)\Sigma_1(X^{\varepsilon,\delta,u}(s), Y)Q_1 u\mathrm{P}^{\varepsilon,\Delta}(dudYds) \rightarrow 0 \ .
\end{align}
\end{lemma}
\begin{proof}[Proof of Lemma \ref{L:LemmaForODElimit1}]
We begin by proving (\ref{Eq:ConvergenceOccupationMeasureB}). This is a consequence of the weak convergence of the occupation measures, but the situation is somewhat delicate because we are integrating the measures against $H$-valued functions. By Skorohod's Theorem, there exists a probability space and a subsequence along which on which $X^{\varepsilon,\delta,u}$ converges almost surely to $\bar{X}$ in $C([0,T];H)$.
By the Lipschitz continuity of $B_1$ and $\Sigma_1$,
\begin{align} \label{Eq:XCont-B}
  &\left|\int_{U\times \mathcal{Y}\times [0,t]}S_{1}(t-s)B_1(X^{\varepsilon,\delta,u}(s), Y)\mathrm{P}^{\varepsilon,\Delta}(dudYds)\right.\nonumber\\
  &\hspace{5cm}\left.- \int_{U\times \mathcal{Y}\times [0,t]}S_{1}(t-s)B_1(\bar{X}(s), Y)\mathrm{P}^{\varepsilon,\Delta}(dudYds)\right|_H\nonumber\\
  &\leq C\int_{U\times \mathcal{Y}\times [0,t]}\left|X^{\varepsilon,\delta,u}(s) - \bar{X}(s) \right|_H \mathrm{P}^{\varepsilon,\Delta}(dudYds)\nonumber\\
  &\leq C \left|X^{\varepsilon,\delta,u} - \bar{X}\right|_{C([0,T];H)},
\end{align}
which converges almost surely to zero as $\varepsilon \to 0$. This estimate is uniform with respect to the occupation measures.
Similar arguments show that
\begin{align} \label{Eq:XCont-Sig}
  &\left|\int_{U\times \mathcal{Y}\times [0,t]}S_{1}(t-s)\Sigma_1(X^{\varepsilon,\delta,u}(s), Y)Q_1u\mathrm{P}^{\varepsilon,\Delta}(dudYds)\right.\nonumber\\
  &\hspace{3cm}\left.- \int_{U\times \mathcal{Y}\times [0,t]}S_{1}(t-s)\Sigma_1(\bar{X}(s), Y)Q_1u\mathrm{P}^{\varepsilon,\Delta}(dudYds)\right|_H\nonumber\\
  &\leq C\int_{U\times \mathcal{Y}\times [0,T]}(t-s)^{-\frac{\beta_1(\rho_1-2)}{2\rho_1}}e^{-\frac{\lambda}{2}(t-s)}\left|X^{\varepsilon,\delta,u}(s) - \bar{X}(s) \right|_H|u|_U \mathrm{P}^{\varepsilon,\Delta}(dudYds)\nonumber\\
  &\leq C \left(\int_{U\times \mathcal{Y}\times [0,t]} (t-s)^{-\frac{\beta_1(\rho_1-2)}{2\rho_1}}e^{-\frac{\lambda}{2}(t-s)} |u|_U \mathrm{P}^{\varepsilon,\Delta}(dudYds) \right) \left|X^{\varepsilon,\delta,u} - \bar{X}\right|_{C([0,T];H)}\nonumber\\
  &\leq C \left(\int_{U \times \mathcal{Y} \times [0,t]} |u|_U^2 \mathrm{P}^{\varepsilon,\Delta}(dudYds) \right)\left|X^{\varepsilon,\delta,u}-\bar{X}\right|_{C([0,T];H)}.
\end{align}
The final inequality is due to H\"older's inequality and the fact that $\frac{\beta_1(\rho_1-2)}{\rho_1}<1$.
The integral of $|u|_U^2$ is bounded by assumption.
Based on (\ref{Eq:XCont-B}) and (\ref{Eq:XCont-Sig}), it is sufficient to prove (\ref{Eq:ConvergenceOccupationMeasureB}) and (\ref{Eq:ConvergenceOccupationMeasureSmu}) with $X^{\varepsilon,\delta,u}$ replaced by $\bar{X}$.

By Skorohod's Theorem we can find a probability space on which $\mathrm{P}^{\varepsilon,\Delta} \Rightarrow \mathrm{P}$ almost surely in the topology of weak convergence of measures. Recall that the space $U \times \mathcal{Y} \times [0,T]$ is endowed with the weak topology on $U$ times the norm topology on $\mathcal{Y}$ times the usual topology on $[0,T]$. Weak convergence of measures means that for any bounded continuous function $g: U\times \mathcal{Y} \times [0,T] \to \mathbb{R}$,
\[\int\limits_{U \times \mathcal{Y}\times[0,T]} g(u,Y,s) \mathrm{P}^{\varepsilon,\Delta}(dudYds) \to
   \int\limits_{U \times \mathcal{Y} \times [0,T]} g(u,Y,s) \mathrm{P}(dudYds).\]
It is not automatically true then that similar statements hold for unbounded $H$-valued continuous functions like $S(t-s)B_1(\bar{X}(s),Y)$ and $S(t-s)\Sigma_1(\bar{X}(s),Y)Q_1u$.

We now show that
\begin{align*}
&\int\limits_{U \times \mathcal{Y}\times[0,t]} S_1(t-s)\Sigma_1(\bar{X}(s),Y)Q_1u \mathrm{P}^{\varepsilon,\Delta}(dudYds)\\& \to
  \int\limits_{U \times \mathcal{Y}\times[0,t]} S_1(t-s)\Sigma_1(\bar{X}(s),Y)Q_1u \mathrm{P}(dudYds).
\end{align*}
The argument with $B_1$ will be similar but simpler.

First, we argue that the convergence is valid for any finite dimensional projection. Let $\Pi_{1,N}:H \to H$ be the linear projection operator onto the span of $\{e_{1,1},...,e_{1,N}\}$. Then for any $N\geq 1$,
\begin{align*}
  &\left|\int\limits_{U \times \mathcal{Y} \times[0,t]} \Pi_{1,N} S_1(t-s) \Sigma_1(\bar{X}(s),Y)Q_1u \mathrm{P}^{\varepsilon,\Delta}(dudYds)\right.\\
   &\qquad \qquad\left.- \int\limits_{U \times \mathcal{Y}\times [0,t]}\Pi_{1,N} S_1(t-s) \Sigma_1(\bar{X}(s),Y)Q_1u \mathrm{P}(dudYds)  \right|_H^2\\
  &= \sum_{k=1}^N  \left(\int\limits_{U \times \mathcal{Y} \times[0,t]} \left<\Pi_{1,N} S_1(t-s) \Sigma_1(\bar{X}(s),Y)Q_1u,e_{1,k}\right>_H \mathrm{P}^{\varepsilon,\Delta}(dudYds)\right.\\
  &\qquad\qquad\left. - \int\limits_{U \times \mathcal{Y}\times [0,t]} \left<\Pi_{1,N}S_1(t-s) \Sigma_1(\bar{X}(s),Y)Q_1u,e_{1,k}\right>_H \mathrm{P}(dudYds) \right)^2.
\end{align*}
Each term of this sum converges to zero as
$\left<S_1(t-s)\Sigma_1(\bar{X}(s),Y)Q_1u, e_{1,k}\right>_H$ is a one dimensional function for any $k$. It is continuous in $s$, in the norm topology on $Y$, and in the weak topology on $u$ by Lemma \ref{Lm:Smu-compact} (recall that a compact linear operator is continuous from the weak topology to the norm topology \cite[Proposition VI.3.3(a)]{Conway}). Since it is a finite sum, and by the uniform integrability of the measures (Lemma \ref{Lm:UniformIntegrabilityOccupationMeasure}), the above finite sum converges to zero as $\varepsilon \to 0$.

The remainders are uniformly bounded as a consequence of (\ref{Eq:SmQf-tail}) and (\ref{Eq:Smu-bound})
\begin{align*}
  &\left|\int\limits_{U \times \mathcal{Y} \times [0,t]}(I- \Pi_{1,N}) S_1(t-s) \Sigma_1(\bar{X}(s),Y)Q_1u \mathrm{P}^{\varepsilon,\Delta}(dudYds) \right|_H^2\\
  & \leq C \left(\sum_{k=N+1}^\infty \alpha_{1,k}^{-\beta_1} |e_{1,k}|_0^2\right)^{\frac{1}{2p}} \times\nonumber\\
  &\qquad\times\int\limits_{U\times\mathcal{Y}\times[0,t]} (t-s)^{-\frac{\beta_1(\rho_1-2)}{2\rho_1}}e^{-\frac{\lambda}{2}(t-s)} \left(1 + |\bar{X}(s)|_H + |Y|_H \right)|u|_U \mathrm{P}^{\varepsilon,\Delta}(dudYds).
\end{align*}
The above expression is uniformly bounded and small by Lemma \ref{Lm:UniformIntegrabilityOccupationMeasure} and (\ref{Eq:Hypothesis1Equation2}). The above expression also holds with $\mathrm{P}$ replacing $\mathrm{P}^{\varepsilon,\Delta}$

Since the tails are uniformly bounded and the finite dimensional projections converge, the result holds. The analysis for the $B_1$ terms are similar but less technically difficult.
\end{proof}
\begin{proof}[Proof of Lemma \ref{L:LemmaForODElimit2}]
We focus on addressing the second statement of the lemma, since the first statement of the lemma follows along the same lines, but it is simpler technically. For notational convenience let us also write $g(X,Y,u)=\Sigma_{1}(X,Y)Q_{1}u$ for the purposes of this proof. We notice that
\begin{align*}
&\int_{U\times\mathcal{Y}\times [0, t]}
S_{1}(t-s)g(X^{\varepsilon,\delta,u}(s), Y, u)\mathrm{P}^{\varepsilon,\Delta}(dudYds)\nonumber\\
&\hspace{3cm}=\int_0^{t}\frac{1}{\Delta}\int_{s}^{s+\Delta}S_{1}(t-s)g(X^{\varepsilon,\delta,u}(s), Y^{\varepsilon,\delta,u}(r), u(r))drds
\end{align*}

By the uniform continuity of $g$ in $X$ and the uniform continuity of $X^{\varepsilon,\delta,u}$ from Lemma \ref{Lm:OscillationEstimateSlowProcess}, it follows that
\begin{align*}
  &\int_0^{t}\frac{1}{\Delta}\int_{s}^{s+\Delta}S_{1}(t-s)g(X^{\varepsilon,\delta,u}(s), Y^{\varepsilon,\delta,u}(r), u(r))drds\\
  &- \int_0^{t}\frac{1}{\Delta}\int_{s}^{s+\Delta}S_{1}(t-s)g(X^{\varepsilon,\delta,u}(r), Y^{\varepsilon,\delta,u}(r), u(r))drds
\end{align*}
converges to zero as $\varepsilon \downarrow 0$. Therefore, it is enough to study the limit of
\[\int_0^{t}\frac{1}{\Delta}\int_{s}^{s+\Delta}S_{1}(t-s)g(X^{\varepsilon,\delta,u}(r), Y^{\varepsilon,\delta,u}(r), u(r))drds.\]
By changing the order of integration, the above expression equals
\begin{align*}
  &\int_0^\Delta \frac{1}{\Delta} \int_0^r S_1(t-s) g(X^{\varepsilon,\delta,u}(r), Y^{\varepsilon,\delta,u}(r), u(r)) dsdr\\
  &+ \int_\Delta^t \frac{1}{\Delta} \int_{r-\Delta}^r S_1(t-s) g(X^{\varepsilon,\delta,u}(r), Y^{\varepsilon,\delta,u}(r), u(r)) dsdr\\
  &+ \int_{t}^{t+\Delta} \frac{1}{\Delta} \int_{r-\Delta}^t S_1(t-s) g(X^{\varepsilon,\delta,u}(r), Y^{\varepsilon,\delta,u}(r), u(r)) dsdr.
\end{align*}
The first and third terms in this expression converge to zero as $\Delta \to 0$, so we only need to focus on the second term. To motivate why we need to be careful about averaging the semigroup, we make the following observations. For any fixed $X \in H$,
\[\lim_{\Delta \downarrow 0} \frac{1}{\Delta}\int_0^\Delta S_1(t)Xdt = X.\]
This is due to the continuity of the semigroup. The convergence is, unfortunately, not uniform over $X$ in bounded subsets of $H$.
The convergence is uniform over bounded subsets of $H_1^\theta$ for any $\theta>0$, because of the compact embedding of $H_1^\theta$ into $H$, the set of trajectories $\{t \mapsto S_1(t)X: |X|_{\theta,1}\leq 1\}$ is equicontinuous. This means that
{\begin{equation} \label{Eq:SemigroupAvg}
 \lim_{\Delta\downarrow 0}\left\|\frac{1}{\Delta} \int_0^{\Delta} S_1(s) ds - I \right\|_{\mathcal{L}(H_1^\theta,H)}=0,
\end{equation}}
in operator norm. Consequently,
\begin{align*}
  &\Bigg| \int_\Delta^t \frac{1}{\Delta} \int_{r-\Delta}^r S_1(t-s) g(X^{\varepsilon,\delta,u}(r), Y^{\varepsilon,\delta,u}(r), u(r)) dsdr\\
  &\qquad\qquad- \int_{\Delta}^t S(t-r) g(X^{\varepsilon,\delta,u}(r),Y^{\varepsilon,\delta,u}(r),u(r))dr  \Bigg|_H\\
  & \leq \int_{\Delta}^t \left\|\frac{1}{\Delta} \int_0^{\Delta} S_1(s)ds - I \right\|_{\mathcal{L}(H_1^\theta,H)} \left|S(t-r)g(X^{\varepsilon,\delta,u}(r),Y^{\varepsilon,\delta,u}(r),u(r)) \right|_{\theta,1}dr.
\end{align*}
We bound this expression using (\ref{Eq:Smu-bound}) along with Lemma \ref{Lm:LpEstimateSlowAndFastProcess}, implying that the above display converges to zero.
\end{proof}

\subsection{Step 3: Proof of (\ref{Eq:ViablePairInvariantMeasure}) and (\ref{Eq:ViablePairNormalization})}\label{S:LLNInvariantMeasure}
We show in this section that any limit $\mathrm{P}(dudYdt)$ of $\mathrm{P}^{\varepsilon\Delta}(dudYdt)$ satisfies
 (\ref{Eq:ViablePairInvariantMeasure}) and (\ref{Eq:ViablePairNormalization}) under Definition \ref{Def:ViablePair}.

 We start by showing that (\ref{Eq:ViablePairInvariantMeasure}) holds. This is shown in Lemma \ref{Lm:ViablePairInvMeas}, but before doing that we need some preliminary estimates that we present in Lemmas \ref{Lm:ClosenessL2Integrated} and \ref{Lm:ClosenessL2IntegratedFrozenSlow} below. Recall the controlled fast process $Y^{\varepsilon,\delta,u}$ satisfying the equation
\begin{align}\label{Eq:FastProcessWithControl}
d Y^{\varepsilon,\delta,u}(t)&= \dfrac{1}{\delta^2}
\left[A_2Y^{\varepsilon,\delta,u}(t)+B_2(X^{\varepsilon,\delta,u}(t),Y^{\varepsilon,\delta,u}(t))\right]dt\nonumber\\
&\qquad
+\dfrac{1}{\delta^2}\dfrac{\delta}{\sqrt{\varepsilon}}\Sigma_2(X^{\varepsilon,\delta,u}(t), Y^{\varepsilon,\delta,u}(t))Q_2u(t)dt\nonumber\\
& \qquad \qquad +\dfrac{1}{\delta}\Sigma_2(X^{\varepsilon,\delta,u}(t),Y^{\varepsilon,\delta,u}(t))d W^{Q_2} \ , \ Y^{\varepsilon,\delta,u}(0)=Y_0\in H \ .
\end{align}

With some abuse of notation, let us also consider the uncontrolled fast process $\mathrm{Y}^{\varepsilon,\delta}$ \textit{driven by} the controlled slow process $X^{\varepsilon,\delta,u}$
from (\ref{Eq:MildFastSlowStochasticRDEWithControl}):
\begin{align}
d \mathrm{Y}^{\varepsilon,\delta}(t)&=  \dfrac{1}{\delta^2}
\left[A_2\mathrm{Y}^{\varepsilon,\delta}(t)+B_2(X^{\varepsilon,\delta,u}(t),\mathrm{Y}^{\varepsilon,\delta}(t))\right]dt\label{Eq:FastProcessDrivenControlledSlow}\\
& \qquad \qquad \qquad \qquad +\dfrac{1}{\delta}\Sigma_2(X^{\varepsilon,\delta,u}(t),\mathrm{Y}^{\varepsilon,\delta}(t))d W^{Q_2} \ ,   \mathrm{Y}^{\varepsilon,\delta}(0)=Y_0\in H.\nonumber
\end{align}

Note that the fast process $\mathrm{Y}^{\varepsilon,\delta}$ still depends on the control $u$,
but only through the controlled slow process $X^{\varepsilon,\delta,u}$. The driving slow process
$X^{\varepsilon,\delta,u}$ is the process that comes from (\ref{Eq:MildFastSlowStochasticRDEWithControl}), and we
remind the reader that this slow process in (\ref{Eq:MildFastSlowStochasticRDEWithControl}) depends on the controlled
fast process $Y^{\varepsilon,\delta,u}$ in (\ref{Eq:FastProcessWithControl}): the two driving slow processes
in (\ref{Eq:FastProcessWithControl}) and (\ref{Eq:FastProcessDrivenControlledSlow}) are actually the \textit{same} process.

In Lemma \ref{Lm:ClosenessL2Integrated} we show that the
processes $Y^{\varepsilon,\delta,u}$ and $\mathrm{Y}^{\varepsilon,\delta}$ are close in a time--averaged $L^2$
sense.
\begin{lemma}\label{Lm:ClosenessL2Integrated}
Let $u\in \mathcal{P}_2^N(U)$ and let $\varepsilon,\delta,\Delta>0$ be as in Hypothesis 5.
For any $T \geq 0$,
there exists $\varepsilon_0=\varepsilon_0(T,N)>0$ such that for any $0<\varepsilon\leq \varepsilon_0$, we have
\begin{equation}\label{Eq:ClosenessL2Integrated}
\mathbf{E}\dfrac{1}{\Delta}\int_{0}^{T}|Y^{\varepsilon,\delta,u}(t)-\mathrm{Y}^{\varepsilon,\delta}(t)|_H^2dt \leq C(T, N, \varepsilon) \ ,
\end{equation}
where for each fixed $(T,N)$, we have the upper bound $C(T, N, \varepsilon)\rightarrow 0$ as $\varepsilon \downarrow 0$.
\end{lemma}
\begin{proof}[Proof of Lemma \ref{Lm:ClosenessL2Integrated}]  Without loss of generality
we can assume that $Y_0=0$. 
%
Set
\begin{equation*}
\Gamma(t):=
\dfrac{1}{\delta}\int_0^t S_2\left(\dfrac{t-s}{\delta^2}\right)[\Sigma_2(X^{\varepsilon,\delta,u}(s), Y^{\varepsilon,\delta,u}(s))-\Sigma_2(X^{\varepsilon,\delta,u}(s),\mathrm{Y}^{\varepsilon,\delta}(s))]dW^{Q_2}(s) \ .
\end{equation*}

Let $\rho(t):=Y^{\varepsilon,\delta,u}(t)-\mathrm{Y}^{\varepsilon,\delta}(t)$ and set $\Lambda(t):=\rho(t)-\Gamma(t)-Z_{2}^{\varepsilon,\delta,u}(t)$, where we recall that
\[
Z_{2}^{\varepsilon,\delta,u}(t)=\dfrac{1}{\delta^2}\dfrac{\delta}{\sqrt{\varepsilon}}\int_0^t S_2\left(\dfrac{t-s}{\delta^2}\right)\Sigma_2(X^{\varepsilon,\delta,u}(s), Y^{\varepsilon,\delta,u}(s))Q_2u(s)ds.
\]

Notice that $\Lambda(t)$ satisfies the equation
$$\begin{array}{ll}
d\Lambda(t)=&\dfrac{1}{\delta^2}\left[A_2\Lambda(t)+(B_2(X^{\varepsilon,\delta,u}(t), Y^{\varepsilon,\delta,u}(t))
-B_2(X^{\varepsilon,\delta,u}(t), \mathrm{Y}^{\varepsilon,\delta}(t)))\right]dt, \Lambda(0)=0.
\end{array}$$

Therefore by Hypothesis 2 and Young's inequality  we know that, 
$$\begin{array}{ll}
& \dfrac{1}{2}\dfrac{d}{dt}|\Lambda(t)|_H^2=\left\langle\dfrac{d}{dt}\Lambda(t), \Lambda(t)\right\rangle_H
\\
= &\dfrac{1}{\delta^2}[\langle A_2\Lambda(t), \Lambda(t)\rangle_H+\langle (B_2(X^{\varepsilon,\delta,u}(t), Y^{\varepsilon,\delta,u}(t))
-B_2(X^{\varepsilon,\delta,u}(t), \mathrm{Y}^{\varepsilon,\delta}(t))), \Lambda(t)\rangle_H]
\\
\leq & -\dfrac{1}{\delta^2}\lambda|\Lambda(t)|_H^2+\dfrac{1}{\delta^2}L_{b_2}^Y|\rho(t)|_H|\Lambda(t)|_H
\\
\leq & -\dfrac{1}{\delta^2}\left(\lambda-\dfrac{\lambda}{2}\right)|\Lambda(t)|_H^2+\dfrac{1}{\delta^2}\dfrac{(L_{b_2}^Y)^2}{2\lambda}|\rho(t)|_H^2.
\end{array}$$

By comparison principle,  we know that for $0\leq t \leq T$
$$\begin{array}{ll}
|\Lambda(t)|_H^2\leq & \displaystyle{\dfrac{1}{\delta^2}\dfrac{(L_{b_2}^Y)^2}{\lambda}\int_0^t e^{-\lambda(t-s)/\delta^2}|\rho(s)|_H^2ds}.
\end{array}$$


By applying Young's inequality of convolutions,
we know that
$$\begin{array}{ll}
\displaystyle{\int_{0}^{T}|\Lambda(t)|_H^2dt}&\leq
\displaystyle{\dfrac{(L_{b_2}^Y)^2}{\lambda}\left(\dfrac{1}{\delta^2}\int_{0}^{T}e^{-\lambda t/\delta^2}dt\right) \int_{0}^{T} |\rho(t)|_H^2dt\leq
\dfrac{(L_{b_2}^Y)^2}{\lambda^{2}}
 \int_{0}^{T} |\rho(t)|_H^2dt }.
\end{array}$$

By applying Young's inequality we then obtain with $\eta_{1},\eta_{2},\eta_{3}>0$
\begin{align}
\mathbf{E} \int_0^T &|\rho(t)|_H^2 dt \leq   (1+\eta^{-1}_{1}+\eta_{3})\mathbf{E}\int_0^T |\Gamma(t)|_H^2 dt+ (1+\eta_{1}+\eta_{2})\mathbf{E} \int_0^T |\Lambda(t)|_H^2dt\nonumber\\
&\qquad\qquad+(1+\eta^{-1}_{2}+\eta^{-1}_{3})\mathbf{E} \int_0^T |Z_{2}^{\varepsilon,\delta,u}(t)|_H^2dt\nonumber\\
& \leq  (1+\eta^{-1}_{1}+\eta_{3}) \mathbf{E}\int_0^T |\Gamma(t)|_H^2 dt + (1+\eta_{1}+\eta_{2}) \dfrac{(L_{b_2}^Y)^2}{\lambda^{2}} \int_{0}^{T} \mathbf{E}|\rho(t)|_H^2dt \nonumber\\
& \qquad \qquad +c_{T,N}(1+\eta^{-1}_{2}+\eta^{-1}_{3})\dfrac{\delta^{2}}{\varepsilon}\left(1+\mathbf{E}\sup\limits_{0\leq t\leq T}|X^{\varepsilon,\delta,u}(t)|_H^2\right).\label{Eq:rhoBoundedByGamma}
\end{align}

Now let us bound the term $\mathbf{E}|\Gamma(t)|_H^2$ as in the proof of Lemma 3.1 and (4.10) of \cite{CerraiRDEAveraging1}.
We shall make use of the bound (3.5) in \cite{CerraiRDEAveraging1}, so that
for any $J\in \mathcal{L}(L^\infty(D), H)\cap \mathcal{L}(H, L^1(D))$ with $J=J^*$, and for any $s\geq 0$, we have
\begin{equation}\label{Eq:AuxiliaryBoundStochasticConvolution3-5}
\|S_2(s)JQ_2\|_2^2\leq K_2s^{-\beta_2\frac{\rho_2-2}{\rho_2}}e^{-\lambda \frac{\rho_2+2}{\rho_2}s}\|J\|^2_{\mathcal{L}(L^\infty(D),H)} \ ,
\end{equation}
where $$K_2=\left(\dfrac{\beta_2}{e}\right)^{\beta_2\frac{\rho_2-2}{\rho_2}}\zeta_2^{\frac{\rho_2-2}{\rho_2}}\lambda_2^{\frac{2}{\rho_2}} \ ,$$
and the constants $\beta_2, \rho_2, \zeta_2, \lambda_2, \lambda$ all come from Hypothesis 1.

By using (\ref{Eq:AuxiliaryBoundStochasticConvolution3-5}) and setting
$J=\Sigma_2(X^{\varepsilon,\delta,u}(s),Y^{\varepsilon,\delta,u}(s))-\Sigma_2(X^{\varepsilon,\delta,u}(s),\mathrm{Y}^{\varepsilon,\delta}(s))$ we can estimate
\begin{align}
\mathbf{E}|\Gamma(t)|_H^2&=\dfrac{1}{\delta^2}\int_0^t \mathbf{E}\left\|S_2\left(\dfrac{t-s}{\delta^2}\right)
[(\Sigma_2(X^{\varepsilon,\delta,u}(s),Y^{\varepsilon,\delta,u}(s))-\Sigma_2(X^{\varepsilon,\delta,u}(s), \mathrm{Y}^{\varepsilon,\delta}(s)))Q_2]\right\|_2^2ds\nonumber\\
& \leq \dfrac{1}{\delta^2}K_2\int_0^t \left(\dfrac{t-s}{\delta^2}\right)^{-\beta_2\frac{\rho_2-2}{\rho_2}}
e^{-\lambda\frac{\rho_2+2}{\rho_2}\frac{t-s}{\delta^2}}(L_{\sigma_2}^Y)^2\mathbf{E}|\rho(s)|_H^2ds
\label{Eq:EstimateDifferenceStochasticConvolutionControlledUncontrolledFast}\\
& = K_2(L_{\sigma_2}^Y)^2\dfrac{1}{\delta^2}\int_0^t \left(\dfrac{t-s}{\delta^2}\right)^{-\beta_2\frac{\rho_2-2}{\rho_2}}
e^{-\lambda\frac{\rho_2+2}{\rho_2}\frac{t-s}{\delta^2}}\mathbf{E}|\rho(s)|_H^2ds \ .\nonumber
\end{align}

Thus by applying Young's inequality of convolutions to (\ref{Eq:EstimateDifferenceStochasticConvolutionControlledUncontrolledFast}),
 (\ref{Eq:rhoBoundedByGamma}) and Lemma \ref{Lm:LpEstimateSlowAndFastProcess} give us
\begin{align}
&\int_{0}^{T}\mathbf{E}|\rho(t)|_H^2dt\leq
(1+\eta_{1}+\eta_{2})\dfrac{(L_{b_2}^Y)^2}{\lambda^{2}} \int_{0}^{T} \mathbf{E}|\rho(t)|_H^2dt \nonumber\\
& \qquad +(1+\eta^{-1}_{1}+\eta_{3}) K_2(L_{\sigma_2}^Y)^2\left(\dfrac{1}{\delta^2}\int_0^T\left(\dfrac{t}{\delta^2}\right)^{-\beta_2\frac{\rho_2-2}{\rho_2}}
e^{-\lambda\frac{\rho_2+2}{\rho_2}\frac{t}{\delta^2}}dt\right)\int_0^T \mathbf{E}|\rho(t)|_H^2dt\nonumber\\
&\qquad  +c_{T,N}(1+\eta^{-1}_{2}+\eta^{-1}_{3}) \dfrac{\delta^{2}}{\varepsilon}
\left(1+\mathbf{E}\sup\limits_{0\leq t\leq T}|X^{\varepsilon,\delta,u}(t)|_H^2\right)\nonumber\\
& \quad=(1+\eta_{1}+\eta_{2})\dfrac{(L_{b_2}^Y)^2}{\lambda^{2}} \int_{0}^{T} \mathbf{E}|\rho(t)|_H^2dt \nonumber\\
&  \qquad+(1+\eta^{-1}_{1}+\eta_{3})K_2(L_{\sigma_2}^Y)^2\left(\int_0^{\frac{T}{\delta^2}}s^{-\beta_2\frac{\rho_2-2}{\rho_2}}
e^{-\lambda\frac{\rho_2+2}{\rho_2}s}ds\right)\int_0^T \mathbf{E}|\rho(t)|_H^2dt\nonumber\\
& \qquad+c_{T,N}(1+\eta^{-1}_{2}+\eta^{-1}_{3})\dfrac{\delta^{2}}{\varepsilon}
(1+|X_0|_H^2+|Y_0|_H^2).\nonumber
\end{align}

Let us also choose $\eta_{2}=\eta_{3}=\frac{\delta}{\sqrt{\varepsilon}}\downarrow 0$. Then, we  obtain
\begin{align}
&\int_{0}^{T}\mathbf{E}|\rho(t)|_H^2dt\leq
\left(1+\eta_{1}+\frac{\delta}{\sqrt{\varepsilon}}\right)\dfrac{(L_{b_2}^Y)^2}{\lambda^{2}} \int_{0}^{T} \mathbf{E}|\rho(t)|_H^2dt \nonumber\\
&  \qquad+\left(\frac{1+\eta_{1}}{\eta_{1}}+\frac{\delta}{\sqrt{\varepsilon}}\right)K_2(L_{\sigma_2}^Y)^2\left(\int_0^{\frac{T}{\delta^2}}s^{-\beta_2\frac{\rho_2-2}{\rho_2}}
e^{-\lambda\frac{\rho_2+2}{\rho_2}s}ds\right)\int_0^T \mathbf{E}|\rho(t)|_H^2dt\nonumber\\
& \qquad+c_{T,N}\dfrac{\delta}{\sqrt{\varepsilon}}
(1+|X_0|_H^2+|Y_0|_H^2).\nonumber
\end{align}

Let us consider now for $\eta>0$ the function
\[
f(\eta)=(1+\eta)\left[\dfrac{(L_{b_2}^Y)^2}{\lambda^2}+\frac{1}{\eta}K_2(L_{\sigma_2}^Y)^2\int_0^{\infty}s^{-\beta_2\frac{\rho_2-2}{\rho_2}}
e^{-\lambda\frac{\rho_2+2}{\rho_2}s}ds\right].
\]

It is easy to see that $f(\eta)$ is convex with a minimum at
\[
\eta^{*}=\sqrt{K_2(L_{\sigma_2}^Y)^2\int_0^{\infty}s^{-\beta_2\frac{\rho_2-2}{\rho_2}}
e^{-\lambda\frac{\rho_2+2}{\rho_2}s}ds}/\sqrt{\dfrac{(L_{b_2}^Y)^2}{\lambda^2}}.
\]

Then, we compute that
\[
f(\eta^{*})= \left(\dfrac{L_{b_2}^Y}{\lambda}+\sqrt{K_2(L_{\sigma_2}^Y)^2\int_0^{\infty}s^{-\beta_2\frac{\rho_2-2}{\rho_2}}
e^{-\lambda\frac{\rho_2+2}{\rho_2}s}ds}\right)^{2}.
\]

By Hypothesis 2 we know that
the Lipschitz constants $L_{b_2}^Y$ and $L_{\sigma_2}^Y$ are chosen so that
\[
\left(\mathfrak{L}_{b_2,\sigma_2}^Y\right)^{2}=f(\eta^{*})<1,
\]
and therefore we obtain with $\eta_{1}=\eta^{*}$ and for $\delta/\sqrt{\varepsilon}$
sufficiently small
$$\int_0^T \mathbf{E}|\rho(t)|_H^2dt\leq \dfrac{c_{T,N}}{1-\left(\mathfrak{L}_{b_2,\sigma_2}^Y\right)^{2}-O(\delta/\sqrt{\varepsilon})}\dfrac{\delta}{\sqrt{\varepsilon}}
(1+|X_0|_H^2+|Y_0|_H^2)N \ ,$$
where $O(\delta/\sqrt{\varepsilon})\downarrow 0$ as $\delta/\sqrt{\varepsilon}\downarrow 0$ so that
$$\dfrac{1}{\Delta}\int_0^T \mathbf{E}|\rho(t)|_H^2dt\leq \dfrac{c_{T,N}}{1-\left(\mathfrak{L}_{b_2,\sigma_2}^Y\right)^{2}-O(\delta/\sqrt{\varepsilon})}\dfrac{\delta}{\Delta \sqrt{\varepsilon}}
(1+|X_0|_H^2+|Y_0|_H^2)=:C(T,N,\varepsilon) \ ,$$
and $C(T,N,\varepsilon)\rightarrow 0$ as $\varepsilon\downarrow 0$ by our Hypothesis 5.
\end{proof}

Lemma \ref{Lm:ClosenessL2Integrated} shows that $Y^{\varepsilon,\delta,u}$ is close in the appropriate ergodic sense to the process $\mathrm{Y}^{\varepsilon,\delta}$. Notice now that $\mathrm{Y}^{\varepsilon,\delta}$ depends on the controlled slow component $X^{\varepsilon,\delta,u}$. As in the finite dimensional case, one expects that in small time intervals one can regard the effect of $X^{\varepsilon,\delta,u}$ as frozen. To formalize this argument, for $t\leq s$, we introduce the two parameter process $\mathrm{Y}^{\delta, X^{\varepsilon,\delta,u}(t)}(s; t)$.
\begin{align*}
&d\mathrm{Y}^{\delta,X^{\varepsilon,\delta,u}(t)}(s;t)  =\dfrac{1}{\delta^2}\left[A_2 \mathrm{Y}^{\delta,X^{\varepsilon,\delta,u}(t)}(s;t)+B_2(X^{\varepsilon,\delta,u}(t), \mathrm{Y}^{\delta,X^{\varepsilon,\delta,u}(t)}(s;t))\right]ds
\\
&\quad+\dfrac{1}{\delta}\Sigma_2(X^{\varepsilon,\delta,u}(t), \mathrm{Y}^{\delta,X^{\varepsilon,\delta,u}(t)}(s;t))dW^{Q_2}_{s} \ , \ \mathrm{Y}^{\delta, X^{\varepsilon,\delta,u}(t)}(t;t)=\mathrm{Y}^{\varepsilon,\delta}(t) \in H,\nonumber
\end{align*}
where the initial condition $\mathrm{Y}^{\delta, X^{\varepsilon,\delta,u}(t)}(t;t)$ is taken to be $\mathrm{Y}^{\varepsilon,\delta}(t)$ as in (\ref{Eq:FastProcessDrivenControlledSlow}).
Similarly as in the previous lemma, we are going to demonstrate in the next lemma that the processes
$\mathrm{Y}^{\varepsilon,\delta}(s)$ and $\mathrm{Y}^{\delta, X(t)}(s;t)$ are close in a time--averaged $L^2$--sense on the interval $t\leq s \leq t+\Delta$. We have the following.

\begin{lemma}\label{Lm:ClosenessL2IntegratedFrozenSlow}
Let $u\in \mathcal{P}_2^N(U)$ and let $\varepsilon,\delta,\Delta>0$ be as in Hypothesis 5.
For any $t \geq 0$,
there exists $\varepsilon_0=\varepsilon_0(t,N)>0$ such that for any $0<\varepsilon\leq \varepsilon_0$, we have
\begin{equation}\label{Eq:ClosenessL2IntegratedFrozenSlow}
\mathbf{E}\dfrac{1}{\Delta}\int_{t}^{t+\Delta}|\mathrm{Y}^{\varepsilon,\delta}(s)-\mathrm{Y}^{\delta, X^{\varepsilon,\delta,u}(t)}(s;t)|_H^2ds \leq C(t,N,\varepsilon) \ ,
\end{equation}
where for each fixed $N$, we have the upper bound $C(t, N, \varepsilon)\rightarrow 0$ as $\varepsilon \downarrow 0$.
\end{lemma}
\begin{proof} The proof of the estimate (\ref{Eq:ClosenessL2IntegratedFrozenSlow}) follows very much the same line as Lemma
\ref{Lm:ClosenessL2Integrated}. Hence, we only describe what is different here. Notice that, for $t\leq s\leq t+\Delta$ and fixed $Y\in H$, we have
$$|B_2(X^{\varepsilon,\delta,u}(s), Y)-B_2(X^{\varepsilon,\delta,u}(t), Y)|_H\leq L_{b_2}^X |X^{\varepsilon,\delta,u}(s)-X^{\varepsilon,\delta,u}(t)|_H \ ,$$
and then by Lemma \ref{Lm:OscillationEstimateSlowProcess}, we get for $p=2/\zeta>2$ that
\begin{equation*}
\lim_{|t-s|\rightarrow 0} \sup\limits_{\varepsilon\in (0,1]}\mathbf{E}|X^{\varepsilon,\delta,u}(t)-X^{\varepsilon,\delta,u}(s)|_H^p=0.
\end{equation*}

With this estimate at hand, we can then proceed using the same estimates as we did in Lemma \ref{Lm:ClosenessL2Integrated}
to obtain (\ref{Eq:ClosenessL2IntegratedFrozenSlow}).
 \end{proof}

Now, we have all the necessary tools to show that (\ref{Eq:ViablePairInvariantMeasure}) holds. In particular we have the following lemma.
\begin{lemma}\label{Lm:ViablePairInvMeas}
Under Hypothesis 1,2 and 3, if $(X^{\varepsilon,\delta,u}, \mathrm{P}^{\varepsilon,\Delta})$ converges in distribution to $(\bar{X}, \mathrm{P})$ in $C([0,T];H) \times \mathscr{P}(E)$, then we have that $P\in\mathbb{P}$, i.e. that for any {$f\in C_{b}(\mathcal{Y})$},
\begin{equation*}\label{Eq:DecouplingOccupatiobMeasureLimit}
\int_{U\times \mathcal{Y}\times[0,T]}f(Y)\mathrm{P}(dudYdt)=\int_{0}^{T}\int_{\mathcal{Y}} f(Y)\mu^{\bar{X}_t}(dY)dt,
\end{equation*}
where $\mu^{\bar{X}_t}(dY)$ is the invariant measure associated to the operator $\mathcal{L}^{X}$ introduced in (\ref{Eq:FastGeneratorReg1})
 with $X=\bar{X}_t$.
\end{lemma}
\begin{proof}
Without loss of generality we can also assume that $f$ is Lipschitz continuous with Lipschitz constant $L_{f}$. We begin with the following decomposition
\begin{align}
&\int_{U\times \mathcal{Y}\times[0,T]}f(Y)\mathrm{P}(dudYdt)-\int_{0}^{T}\int_{\mathcal{Y}} f(Y)\mu^{\bar{X}(t)}(dY)dt=\nonumber\\
&\quad= \left(\int_{U\times \mathcal{Y}\times[0,T]}f(Y)\mathrm{P}(dudYdt)-\int_{U\times \mathcal{Y}\times[0,T]}f(Y)\mathrm{P}^{\varepsilon,\Delta}(dudYdt)\right)\nonumber\\
&\qquad+\left(\int_{U\times \mathcal{Y}\times[0,T]}f(Y)\mathrm{P}^{\varepsilon,\Delta}(dudYdt)- \int_{0}^{T}\int_{\mathcal{Y}} f(Y)\mu^{\bar{X}_t}(dY)dt\right)\nonumber
\end{align}
\begin{align}
&\quad= \left(\int_{U\times \mathcal{Y}\times[0,T]}f(Y)\mathrm{P}(dudYdt)-\int_{U\times \mathcal{Y}\times[0,T]}f(Y)\mathrm{P}^{\varepsilon,\Delta}(dudYdt)\right)\nonumber\\
&\qquad+\left(\int_{0}^{T} \frac{1}{\Delta}\int_{t}^{t+\Delta}f(Y^{\varepsilon,\delta,u}(s))ds dt - \int_{0}^{T} \frac{1}{\Delta}\int_{t}^{t+\Delta}f(\mathrm{Y}^{\varepsilon,\delta}(s))ds dt\right)\nonumber\\
&\qquad+\left(\int_{0}^{T} \frac{1}{\Delta}\int_{t}^{t+\Delta}f(\mathrm{Y}^{\varepsilon,\delta}(s))ds dt-\int_{0}^{T} \frac{1}{\Delta}\int_{t}^{t+\Delta}f(\mathrm{Y}^{\delta, X^{\varepsilon,\delta,u}(t)}(s;t))ds dt\right)\nonumber\\
&\qquad+\left(\int_{0}^{T} \frac{1}{\Delta}\int_{t}^{t+\Delta}f(\mathrm{Y}^{\delta,\bar{X}(t)}(s;t))ds dt-\int_{0}^{T}\int_{\mathcal{Y}} f(Y)\mu^{\bar{X}(t)}(dY)dt\right)\nonumber\\
&\qquad+\left(\int_{0}^{T} \frac{1}{\Delta}\int_{t}^{t+\Delta}f(\mathrm{Y}^{\delta, X^{\varepsilon,\delta,u}(t)}(s;t))ds dt-\int_{0}^{T} \frac{1}{\Delta}\int_{t}^{t+\Delta}f(\mathrm{Y}^{\delta,\bar{X}(t)}(s;t))ds dt\right)\nonumber\\
&\quad=\sum_{i=1}^{5}J^{\varepsilon,\delta,\Delta}_{i}(T).\nonumber
\end{align}

The next goal is to show that each of the $J^{\varepsilon,\delta,\Delta}_{i}(T)$ terms goes to zero in probability as $\varepsilon\downarrow 0$.  
We assumed that  $(X^{\varepsilon,\delta,u}(\cdot), \mathrm{P}^{\varepsilon,\Delta})$ converge in distribution to  $(\bar{X}(\cdot), \mathrm{P})$. At this point, we will use again the Skorokhod representation theorem (Theorem 1.8 in \cite{EithierKurtz1986}), which, for the purposes of identifying the limit, allows us to assume
that the aforementioned convergence holds with probability one. The Skorokhod representation theorem involves the introduction of another probability space, but this distinction is ignored in the
notation.

We immediately get that $J^{\varepsilon,\delta,\Delta}_{1}(T)$ goes to zero in probability as $\varepsilon\downarrow 0$. Lemma \ref{Lm:ClosenessL2Integrated} and dominated convergence theorem shows that  $J^{\varepsilon,\delta,\Delta}_{2}(T)$ goes to zero in $L^{1}$ as $\varepsilon,\delta\downarrow 0$.  Indeed, we notice that
\begin{align}
&\mathbf{E}\left|\frac{1}{\Delta}\int_{t}^{t+\Delta}f(Y^{\varepsilon,\delta,u}(s))ds-\frac{1}{\Delta}\int_{t}^{t+\Delta}f(Y^{\varepsilon,\delta}(s))ds\right|\leq\nonumber\\
&\qquad\leq L_{f} \mathbf{E} \frac{1}{\Delta}\int_{t}^{t+\Delta}\left|Y^{\varepsilon,\delta,u}(s)-Y^{\varepsilon,\delta}(s)\right|_{H}ds\nonumber\\
&\qquad \leq L_{f}\frac{1}{\Delta}\left(\int_{t}^{t+\Delta}1ds\right)^{1/2} \left(\mathbf{E} \int_{t}^{t+\Delta}\left|Y^{\varepsilon,\delta,u}(s)-Y^{\varepsilon,\delta}(s)\right|^{2}_{H}ds\right)^{1/2}\nonumber\\
& \qquad\leq L_{f} \left(\mathbf{E} \frac{1}{\Delta}\int_{t}^{t+\Delta}\left|Y^{\varepsilon,\delta,u}(s)-Y^{\varepsilon,\delta}(s)\right|^{2}_{H}ds\right)^{1/2}\rightarrow 0.\nonumber
\end{align}

Similarly, Lemma  \ref{Lm:ClosenessL2IntegratedFrozenSlow} and dominated convergence theorem shows that  $J^{\varepsilon,\delta,\Delta}_{3}(T)$ goes to zero in  $L^{1}$ as $\varepsilon,\delta\downarrow 0$. As far as  $J^{\varepsilon,\delta,\Delta}_{4}(T)$ is concerned, we define the time--rescaled process $\mathrm{Y}^{\bar{X}(t)}(s)=\mathrm{Y}^{\delta, \bar{X}(t)}(t+\delta^2 s; t)$
\begin{align*}
d \mathrm{Y}^{\bar{X}(t)}(s) & =
\left[A_2 \mathrm{Y}^{\bar{X}(t)}(s)+B_2(\bar{X}(t),\mathrm{Y}^{\bar{X}(t)}(s))\right]ds
+\Sigma_2(\bar{X}(t), \mathrm{Y}^{\bar{X}(t)}(s))d W^{Q_2} \nonumber\\
\mathrm{Y}^{\bar{X}(t)}(0) & = \mathrm{Y}^{\varepsilon,\delta}(t) \in H \ , \ 0\leq s\leq \dfrac{\Delta}{\delta^2} \ ,
\end{align*}
and we notice that
$$\dfrac{1}{\Delta}\int_{t}^{t+\Delta} f(\mathrm{Y}^{\delta,\bar{X}(t)}(s;t))ds=
\dfrac{1}{\frac{\Delta}{\delta^2}}\int_{0}^{\frac{\Delta}{\delta^2}}f(\mathrm{Y}^{\bar{X}(t)}(s))ds \ .
$$

Hence, by making use of Lemma \ref{Lm:ErgodicTheoremFastProcessReg1} and of Hypothesis 5, to obtain that in $L^{1}$
$$\lim\limits_{\varepsilon\downarrow 0}
\dfrac{1}{\frac{\Delta}{\delta^2}}\int_{0}^{\frac{\Delta}{\delta^2}}f(\mathrm{Y}^{\bar{X}(t)}(s))ds=\int_{\mathcal{Y}}f(Y)\mu^{\bar{X}(t)}(dY),
$$
which together with dominated convergence indeed implies that $J^{\varepsilon,\delta,\Delta}_{4}(T)$ goes to zero in probability as $\varepsilon\downarrow 0$.
It remains to study the term
\begin{align}
J^{\varepsilon,\delta,\Delta}_{5}(T)&=\int_{0}^{T} \frac{1}{\Delta}\int_{t}^{t+\Delta}f(\mathrm{Y}^{\delta, X^{\varepsilon,\delta,u}(t)}(s;t))ds dt-\int_{0}^{T} \frac{1}{\Delta}\int_{t}^{t+\Delta}f(\mathrm{Y}^{\delta,\bar{X}(t)}(s;t))ds dt\nonumber\\
&=\int_{0}^{T}\left[ \dfrac{1}{\frac{\Delta}{\delta^2}}\int_{0}^{\frac{\Delta}{\delta^2}}f(\mathrm{Y}^{X^{\varepsilon,\delta,u}(t)}(s))ds- \dfrac{1}{\frac{\Delta}{\delta^2}}\int_{0}^{\frac{\Delta}{\delta^2}}f(\mathrm{Y}^{\bar{X}(t)}(s))ds\right] dt.\nonumber
\end{align}

Due to dominated convergence and Lemma 3.1 of \cite{CerraiRDEAveraging1}, this term goes to zero.
\end{proof}

We end this section with the validation of (\ref{Eq:ViablePairNormalization}). As in the finite dimensional case, see \cite{LDPWeakConvergence}, this follows by the fact that  the analogous property holds at the prelimit level together with the fact that  $\mathrm{P}(U\times\mathcal{Y}\times\{t\})=0$ and the continuity of $t\rightarrow \mathrm{P}(U\times \mathcal{Y}\times [0,t])$ to deal with null sets.

\section{Derivation of the large deviation principle -- Proof of Theorem \ref{Theorem:LargeDeviationPrinciple}}\label{S:LDPproof}

In this section we prove the upper and lower bounds for the Laplace principle and compactness of level sets of the action functional. These results then directly imply Theorem \ref{Theorem:LargeDeviationPrinciple}. The upper bound is proven in Subsection \ref{SS:LDPLowerBound}, the lower bound in Subsection \ref{SS:LDPUpperBound} and compactness of level sets of the action functional in Subsection \ref{SS:CompactLevelSets}.

It turns out that based on the representation (\ref{Eq:LDPRepresentationFastSlowSRDE}), Theorem \ref{T:MainTheorem1}, and Fatou's lemma, the Laplace principle upper bound follows immediately. Things, however, are considerably more complicated for the Laplace principle lower bound. For the lower bound, we need to construct a nearly optimal control that achieves the lower bound. Due to the presence of the multiple scales, it turns out that any nearly optimal control, has in principle to depend on $Y$. Hence, averaging principle would then work if regularity properties of such a control were known. In the finite dimensional case \cite{LDPWeakConvergence}, this was done via an explicit construction of the control and possible connections to related Hamilton-Jacobi-Equations. The situation is considerably more complicated here.

If the spatial dimension is higher than one, i.e.  when $d>1$, and if $\sigma_{1}$ depends on both $X$ and $Y$ components it turns out that the available explicit constructions are problematic because of the colored noise, leading to potentially unbounded controls. We will see this in detail in Subsection \ref{SS:UpperBound_OneDim} and Section \ref{S:Generalizations}. As we will see in Subsection \ref{SS:UpperBound_OneDim}, even in the one-dimensional case the proof is quite involved. If, on the other hand $\sigma_{1}(x,X,Y)=\sigma_{1}(x,X)$ does not depend on $Y$, then one can effectively consider a nearly optimal control that depends only on time $t$ and not on $Y$ or $\varepsilon$, in which case the proof is rather straightforward as we shall see in Subsection \ref{SS:UpperBound_MultiDim}.

\subsection{Laplace principle upper bound}\label{SS:LDPLowerBound}
Our goal is to show that for any bounded, continuous functions $h$ mapping
$C([0,T]; H)$ into $\mathbb{R}$ we have
\begin{align*}
  &\limsup\limits_{\varepsilon\downarrow 0}\varepsilon \ln \mathbf{E}\left[\exp\left(-\dfrac{h(X^{\varepsilon,\delta})}{\varepsilon}\right)\right]
  \leq - \inf_{\phi \in C([0,T];H)} \{S(\phi) + h(\phi)\}\\
&= -\inf\limits_{(\phi, \mathrm{P})\in \mathcal{V}_{(\xi,\mathcal{L})}}
\left[\dfrac{1}{2}\int_{U\times \mathcal{Y} \times [0,T]}|u|_U^2\mathrm{P}(dudYdt)+h(\phi)\right] \ ,
\end{align*}
where $S$ is the rate function defined in Lemma \ref{Theorem:LargeDeviationPrinciple}.

It is sufficient to prove the above upper limit along any subsequence such that
$$\varepsilon\ln \mathbf{E}\left[\exp\left(-\dfrac{h(X^{\varepsilon,\delta})}{\varepsilon}\right)\right]$$
converges. From the moment that
$\left|\varepsilon\ln\mathbf{E}\left[\exp\left(-\dfrac{h(X^{\varepsilon,\delta})}{\varepsilon}\right)\right]\right|\leq
\sup\limits_{\phi\in C([0,T]; H)}|h(\phi)|$ such a subsequence will exist.

Recalling that the controlled process $X^{\varepsilon,\delta,u}$  defined via (\ref{Eq:FastSlowStochasticRDEWithControl}),  (\ref{Eq:LDPRepresentationFastSlowSRDE}) implies  that
there exists a family of controls $\{u^\varepsilon, \varepsilon>0\}$ in $\mathcal{P}_2(U)$ such that for every $\varepsilon>0$
{$$
\varepsilon\ln\mathbf{E} \left[\exp\left(-\dfrac{h(X^{\varepsilon,\delta})}{\varepsilon}\right)\right]\leq - \left( \mathbf{E}
\left[\dfrac{1}{2}\int_0^T |u^\varepsilon(t)|_U^2dt+h(X^{\varepsilon,\delta,u^\varepsilon})\right]-\varepsilon\right)
$$}

Without loss of generality, we can assume that $u^\varepsilon \in \mathcal{P}_2^N$ for $N$ large enough using the arguments of Theorem 4.4 of \cite{VariationalInfniteBM}.
Hence, using this family of controls and the associated controlled process $X^{\varepsilon,\delta, u^\varepsilon}$
to construct occupation measures $\mathrm{P}^{\varepsilon,\Delta}$ in (\ref{Eq:OccupationMeasureBeforeAveraging}), the results of Section \ref{SSS:TightnessViablePair}
guarantee that the family $\{(X^{\varepsilon,\delta,u}, \mathrm{P}^{\varepsilon,\Delta}), \varepsilon>0, 0\leq t \leq T\}$ will be tight. As a consequence of Theorem \ref{T:MainTheorem1}, given any subsequence of
$\varepsilon\downarrow0$ there is a further sub--subsequence for which $(X^{\varepsilon,\delta,u^\varepsilon}, P^{\varepsilon,\Delta}) \rightharpoonup (\bar{X}, \mathrm{P})$
in distribution, where $(\bar{X}, \mathrm{P})$ is a viable pair. By Fatou's Lemma we have
$$\begin{array}{ll}
&\limsup\limits_{\varepsilon\downarrow 0} \varepsilon\ln \mathbf{E}\left[\exp\left(-\dfrac{h(X^{\varepsilon,\delta})}{\varepsilon}\right)\right]
\\
\leq & \displaystyle{\limsup\limits_{\varepsilon\downarrow 0} \left(-\mathbf{E}\left[\dfrac{1}{2}\int_0^T |u^\varepsilon(t)|_U^2dt+h(X^{\varepsilon,\delta,u^\varepsilon})\right]+\varepsilon\right)}
\\
\leq & \displaystyle{-\liminf\limits_{\varepsilon\downarrow 0} \left(\mathbf{E}\left[\dfrac{1}{2}\int_0^T \dfrac{1}{\Delta}
\int_t^{t+\Delta}|u^\varepsilon(s)|_U^2dsdt+h(X^{\varepsilon,\delta,u^\varepsilon})\right]\right)}
\\
=    & \displaystyle{-\liminf\limits_{\varepsilon\downarrow 0} \left(\mathbf{E}\left[\dfrac{1}{2}\int_{U\times \mathcal{Y}\times [0,T]}
|u|_U^2\mathrm{P}^{\varepsilon,\Delta}(dudYdt)+h(X^{\varepsilon,\delta,u^\varepsilon})\right]\right)}
\\
\leq &\displaystyle{ -\left[\frac{1}{2} \int_{U\times\mathcal{Y}\times[0,T]} |u|_U^2 \mathrm{P}(dudYdt) + h(\bar{X}) \right]}\\
\leq  &\displaystyle{-\inf\limits_{(\phi, \mathrm{P})\in \mathcal{V}_{(\xi,\mathcal{L})}}\left[\dfrac{1}{2}\int_{U\times \mathcal{Y}\times [0,T]}|u|_U^2
\mathrm{P}(dudYdt)+h(\phi)\right]},
\end{array}$$
which concludes the proof of the Laplace principle upper bound.

\subsection{Laplace principle lower bound}\label{SS:LDPUpperBound}
To prove the Laplace principle lower bound we need to show that for all $h: C([0,T]; H)\rightarrow \mathbb{R}$ bounded and continuous
$$\limsup\limits_{\varepsilon\rightarrow 0}\varepsilon\ln \mathbf{E}\left[\exp\left(-\dfrac{h(X^{\varepsilon,\delta})}{\varepsilon}\right)\right]
\geq - \inf\limits_{\phi\in C([0,T]; H)}[S(\phi)+h(\phi)] \ .$$

For a given constant $\eta>0$,  consider $\psi\in C([0,T]; H)$ with $\psi_0=X_0$ such that
$$
\displaystyle{S(\psi)+h(\psi)}
\displaystyle{ \leq \inf\limits_{\phi\in C([0,T]; H)} [S(\phi)+h(\phi)]+\eta < \infty} \ .
$$

Before we continue with the proof of the lower bound, let us first rewrite the action functional in a more useful form. Using the definition of $S(\psi)$ we can write
\begin{align}
S(\psi)&=\inf\limits_{(\psi, \mathrm{P})\in \mathcal{V}_{(\xi, \mathcal{L})}}
\left[\dfrac{1}{2}\int_{U\times \mathcal{Y}\times [0,T]}|u|_U^2\mathrm{P}(dudYdt)\right]=
{L^r_T(\psi)} \nonumber \ ,
\end{align}
where 
\begin{equation*}
{L^r_{T}(\psi)=\inf\limits_{\mathrm{P}\in \mathcal{A}_{\psi,t}^{r}}\dfrac{1}{2}\int_0^T\int_{U\times \mathcal{Y}}|u|_U^2\mathrm{P}_{s}(dudY)ds} \ ,
\end{equation*}
with
\begin{align*}
\mathcal{A}^{r}_{\psi,T}&=\left\{\mathrm{P}:[0,T] \to \mathcal{P}(U\times \mathcal{Y}): \mathrm{P}_{s}(dudY)=\eta(du|Y,s)\mu^{\psi(s)}(dY) \ ,\right.\nonumber\\
 &\qquad\left.\int_0^T\int_{U\times \mathcal{Y}}\left(|u|_U^2+|Y|^{2}_{\theta,2}\right)\mathrm{P}_s(dudY)ds<\infty \ , \right.
\\
&\qquad\left.  \psi(t)=S_1(t)X_0+\int_{U \times\mathcal{Y}\times[0,t]}S_1(t-s)\xi(\psi(s),Y,u)\mathrm{P}_{s}(dudY)ds, \ \ {t \in [0,T]}\right\} \ ,
\end{align*}
where $\mu^X$ is the invariant measure from (\ref{Eq:mu-X}) {and $\xi$ is defined in \eqref{Eq:xiDef}}. Now, for each $t\in[0,T]$ let us define
\begin{equation*}
{L^o_{T}(\psi)=\inf\limits_{v\in \mathcal{A}_{\psi,T}^{o}}\dfrac{1}{2}\int_0^T\int_{\mathcal{Y}}|v(s,Y)|_U^2\mu^{\psi(s)}(dY)ds} \ ,
\end{equation*}
where
\begin{equation*}
\begin{array}{ll}
\mathcal{A}_{\psi,T}^{o} & \displaystyle{=\left\{v:[0,T]\times \mathcal{Y} \to U: \int_0^T\int_{\mathcal{Y}}\left(|v(s,Y)|_U^2+|Y|^{2}_{\theta,2}\right)\mu^{\psi(t)}(dY)ds<\infty  \ , \right.}
\\
&   \displaystyle{ \ \left.\psi(t)=S_1(t)X_0+\int_{\mathcal{Y}\times[0,t]}S_1(t-s)\xi(\psi(s),Y,v(s,Y))\mu^{\psi(s)}(dY)ds, \ \ {t \in [0,T]}\right\}} \ .
\end{array}
\end{equation*}

Our claim is that 
one actually has that $L^r_{T}(\psi)=L^o_{T}(\psi)$. This follows by the quadratic dependence of the cost on the control and by the affine dependence of $\xi$ on the control. Indeed if we let $v(t,Y)=\displaystyle{\int_{U}u\eta(du|Y,t)}$
where $\eta(du|Y,t)$ is the conditional distribution, so that $v\in \mathcal{A}_{\psi,T}^{o}$, then by Jensen's inequality
we get { for any fixed $t \in [0,T]$,}
$$\int_{U\times \mathcal{Y}}\dfrac{1}{2}|u|_U^2\mathrm{P}_{t}(dudY)dt\geq \int_{\mathcal{Y}}\dfrac{1}{2}
\left|\int_{U}u\eta(du|Y,t)\right|_U^2\mu^{\psi(t)}(dY)=\dfrac{1}{2}\int_{\mathcal{Y}}|v(t,Y)|_U^2\mu^{\psi(t)}(dY) \ ,$$
and so $L_T^r(\psi)\geq L_T^o(\psi)$.

For the reverse direction, for given $v \in \mathcal{A}_{\psi,T}^{o}$, we can define  $\mathrm{P}\in \mathcal{A}_{\psi,T}^{r}$
via $\mathrm{P}_{t}(dudY)=\delta_{v(t,Y)}(du)\mu^{\psi(t)}(dY)$. Hence, we have $L_T^r(\psi)\leq L_T^o(\psi)$.

Therefore, we have indeed obtained that 
\begin{align}
S(\psi)&= \inf\limits_{v\in \mathcal{A}_{\psi,T}^{o}} \dfrac{1}{2}\int_{0}^{T}\int_{\mathcal{Y}}|v(t,Y)|_U^2\mu^{\psi(t)}(dY)dt. \label{Eq:ActionFunctOrdinaryForm}
\end{align}

Having derived the representation of the last display, let us continue with the proof of the lower bound.
Let us consider $\tilde{v}(t,Y)\in \mathcal{A}_{\psi,t}^{o}$ such that
\[
\int_{0}^{T}\dfrac{1}{2}\int_{\mathcal{Y}}|\tilde{v}(t,Y)|_U^2\mu^{\psi(t)}(dY)dt\leq S(\psi)+\eta \ .
\]

In general, it is very difficult to find an explicit representation for $\tilde{v}(t,Y)$. At this point we strengthen Hypothesis 3 to Hypothesis 4. In Subsection \ref{SS:UpperBound_OneDim} we consider the one dimensional case with $\sigma_{1}^{2}$ bounded from below and above. In Subsection \ref{SS:UpperBound_MultiDim} we consider the multidimensional case, $d\geq1$, with $\sigma_{1}(x,X,Y)=\sigma_{1}(x,X)$ independent of $Y$.

\subsubsection{Lower bound for the $d=1$ case with $\sigma_{1}^{2}$ bounded from below and above}\label{SS:UpperBound_OneDim}
We will study the problem in the special case that $U=H = \mathcal{X} = \mathcal{Y}=L^2([0,1])$ and $Q_1 = I$. We also assume that $\sigma_1$ is bounded above and below $0<c_0 \leq \sigma^{2}_1(x,X,Y) \leq c_1$. In this case, we can use the methods from \cite{LDPWeakConvergence} to find an explicit formulation for $\tilde{v}(t,Y)$.

Define $a(X): L^2(\mathcal{Y},\mu^X; U) \to H$ by
\begin{equation}
  a(X)u = \int_\mathcal{Y} \Sigma_1(X,Y)u(Y) \mu^X(dY).
\end{equation}
For any $X \in \mathcal{X}$, $a(X)$ is a bounded operator and for any $u \in L^2(\mathcal{Y},\mu^{X};U)$,
\[|a(X)u|_H \leq \sqrt{c_1} |u|_{L^2(\mathcal{Y},\mu^{X};U)}.\]
Then the adjoint of $a(X)$ is $a^\star(X): H \to L^2(\mathcal{Y},\mu^X;U)$
\begin{equation}
  [a^\star(X)h](Y) = \Sigma_1^\star(X,Y)h.
\end{equation}
Define $q(X): H \to H$ by
\begin{equation}
  q(X)h = a(X)a^\star(X)h = \int_\mathcal{Y} \Sigma_1(X,Y)\Sigma_1^\star(X,Y) h \mu^X(dY).
\end{equation}

For presentation purposes, we first present a few technical lemmas that are essential for the proof. We defer their proof to the end of this subsection.
\begin{lemma}\label{L:Invertible_q}
  The operator $q(X): H \to H$ is invertible and $|q^{-1}(X)h|_H \leq \frac{1}{c_0} |h|_H$ for all $h \in H$, $X \in \mathcal{X}$.
\end{lemma}

By the assumption that $S(\psi) <+\infty$ and by the representation $\mathcal{A}^o_{\psi,T}$, there exists $u(t,Y)$ such that
  \[\frac{1}{2} \int_0^T \int_\mathcal{Y} |u(s,Y)|_U^2 \mu^{\psi(s)}(dY)ds \leq  S(\psi) + \eta\]
  and
  \begin{align*}
    \psi(t) = &S_1(t)\psi(0) + \int_0^t \int_\mathcal{Y} S_1(t-s) B_1(\psi(s),Y) \mu^{\psi(s)}(dY) ds \\
    &+ \int_0^t \int_\mathcal{Y} S_1(t-s) \Sigma_1(\psi(s), Y)u(s,Y) \mu^{\psi(s)}(dY)ds.
  \end{align*}
  Of course, there is no guarantee that $u(t,Y)$ is bounded or Lipschitz continuous in $Y$. For this reason, we have the following lemma.

\begin{lemma} \label{Lm:Lipschitz-control}
  If $\psi \in C([0,T];H)$, $S(\psi)<+\infty$, and $\eta>0$, then there exists a control ${v}(t,Y)$ that for each $t \in [0,T]$,  is bounded uniformly and Lipschitz continuous in $Y$ in the sense that there exists $\gamma \in L^2([0,T])$ such that for any $t>0$,
  \[\sup_{Y \in \mathcal{Y}} |{v}(t,Y)|_U^2  \leq \gamma(t),\]
  for any $Y_1, Y_2 \in \mathcal{Y}$,
  \[|{v}(t,Y_1) - {v}(t,Y_2)|_U \leq \gamma(t)|Y_1-Y_2|_H, \]
   and ${v}$ takes the form
  \begin{align} \label{Eq:v-form}
    {v}(t,Y)& = \Sigma_1^\star(\psi(t), Y)q^{-1}(\psi(t))a(\psi(t))u(t,\cdot)\nonumber\\
    &=\Sigma_1^\star(\psi(t), Y)q^{-1}(\psi(t)) \int_\mathcal{Y} \Sigma_1(\psi(t),Y)u(t,Y) \mu^{\psi(t)}(dY).
  \end{align}
  Furthermore,
  \[\frac{1}{2} \int_0^T \int_\mathcal{Y} |{v}(s,Y)|_U^2 \mu^X(dY)ds \leq S(\psi) + \eta\]
  and
  \begin{align} \label{Eq:ControlEq}
    \psi(t) = &S_1(t)\psi(0) + \int_0^t \int_\mathcal{Y} S_1(t-s) B_1(\psi(s),Y) \mu^{\psi(s)}(dY) ds \nonumber\\
    &+ \int_0^t \int_\mathcal{Y} S_1(t-s) \Sigma_1(\psi(s), Y){v}(s,Y) \mu^{\psi(s)}(dY)ds.
  \end{align}
\end{lemma}
\begin{lemma} \label{Lm:HilbertSpaceNorms}
  Let $H_1$ and $H_2$ be Hilbert spaces and $a: H_1 \to H_2$ be a bounded linear operator. Let $q = aa^\star$. Let $q^{-1}$ be the pseudo-inverse of $q$. Then for any $u \in H_1$,
  $|a^\star q^{-1} a u|_{H_1} \leq |u|_{H_1}$.
\end{lemma}

Now, if we have any $v(t,Y)$ that is Lipschitz continuous in $Y$, then we have unique solvability of the control problem.
\begin{lemma} \label{Lm:UniqueSolvability}
  Assume that $v(t,Y)$ is bounded and Lipschitz continuous in $Y$ in the sense that there exists $\gamma  \in L^2([0,T])$ such that for any $t \in [0,T]$,
  \[\sup_{Y \in \mathcal{Y}} |v(t,Y)|_U \leq \gamma(t),\]
  and for any $Y_1, Y_2 \in \mathcal{Y}$,
  \[|v(t,Y_1) - v(t,Y_2)|_U \leq \gamma(t)|Y_1 - Y_2|_H.\]
  Then there exists a unique $\psi$ solving (\ref{Eq:ControlEq}).
\end{lemma}

Now we show that if a sequence $v_n(t,Y)$ approaches $v(t,Y)$ in an appropriate way, then the control processes $\psi_n$ associated with $v_n$ in (\ref{Eq:ControlEq}), converge to $\psi$ associated with $v$ in (\ref{Eq:ControlEq}).
\begin{lemma} \label{Lm:ApproachControl}
  Assume that $v_n(t,Y)$ is a sequence of processes satisfying
  \[\sup_n |v_n(s,Y)|_U \leq \gamma_n(s), \text{ for all }Y \in \mathcal{Y}\]
  and
  \[\sup_n |v_n(s,Y_1) - v_n(s,Y_2)|_U \leq \gamma_n(s) |Y_1 - Y_2|_U, \text{ for all } Y_1,Y_2 \in \mathcal{Y}\]
  for some $\gamma_n \in L^2([0,T])$ with $\sup_n \int_0^T \gamma_n(s)^2 ds <+\infty$. Let $\psi_n \in C([0,T];H)$ be the solution to the control problem (\ref{Eq:ControlEq}) associated with $v_n$.  Assume $v(s,Y)$ satisfies the same boundedness and Lipschitz properties with respect to $\gamma \in L^2([0,T])$ and let $\psi\in C([0,T];H)$ be the solution to the control problem (\ref{Eq:ControlEq}) associated with $v$. Assume that
  \[\lim_{n \to +\infty} \int_0^T \int_\mathcal{Y} |v_n(t,Y) - v(t,Y)|_U^2 \mu^{\psi(t)}(dY)dt = 0.\]
  Then $\psi_n \to \psi$ in $C([0,T];H)$ and
  \[\lim_{n \to +\infty}\frac{1}{2}\int_0^T \int_{\mathcal{Y}} |v_n(t,Y)|_U^2 \mu^{\psi_n(t)}(dY)dt = \frac{1}{2}\int_0^T \int_{\mathcal{Y}} |v(t,Y)|_U^2 \mu^{\psi(t)}(dY)dt.\]
\end{lemma}
\begin{theorem} \label{Theorem:RegularMinimizer}
  Let $S: C([0,T];H) \to [0,+\infty]$ be the large deviations rate function and let $h: C([0,T];H) \to \mathbb{R}$ be a bounded continuous function. For any $\eta>0$ there exists a control $\tilde{v}(t,Y)$ that is bounded,  continuous in $t$, and Lipschitz continuous in $Y$ such that the unique solution $\tilde{\psi} \in C([0,T];H)$ to the control problem (\ref{Eq:ControlEq}) for $\tilde{v}$ is an approximate minimizer to $S +h$ in the sense that
  \[\frac{1}{2}\int_0^T |\tilde v(s,Y)|_U^2 \mu^{\tilde\psi(s)}(dY)ds + h(\tilde\psi) \leq \inf_{\phi \in C([0,T];H)} (S(\phi) + h(\phi)) + \eta.\]
\end{theorem}
\begin{proof}
  If $\inf_{\phi} (S(\phi) + h(\phi)) = +\infty$, then the result is trivial so we assume that the infimum is finite.  There must exist some $\psi \in C([0,T];H)$ such that
  \[S(\psi) + h(\psi) \leq \inf_{\phi} (S(\phi) + h(\phi)) + \frac{\eta}{3}.\]
   By Lemma \ref{Lm:Lipschitz-control}, there is a control $v$ and a function $\gamma \in L^2([0,T])$  such that
   \[\sup_{Y \in \mathcal{Y}} |v(t,Y)|_U \leq \gamma(t)\]
   and for any $Y_1, Y_2 \in \mathcal{Y}$,
   \[|v(t,Y_1) - v(t,Y_2)|_U \leq \gamma(t) |Y_1 - Y_2|_H, \]
   \begin{equation} \label{Eq:L2EstControl}
     \frac{1}{2} \int_0^T \int_{\mathcal{Y}} |v(s,Y)|_U^2 \mu^{\psi(s)}(dY) ds \leq \inf_{\phi \in C([0,T])} S(\phi) + \frac{\eta}{3},
   \end{equation}
   and $\psi$ and $v$ satisfy (\ref{Eq:v-form}) and (\ref{Eq:ControlEq}).
   Let
   \[h(t) = a(\psi(t))v(t,\cdot) = \int_\mathcal{Y} \Sigma_1(\psi(t),Y)v(t,Y)\mu^{\psi(t)}(dY).\]
   By H\"older inequality and the fact that $\sigma_1$ is bounded above, we know that for any $t \in [0,T]$,
   \[|h(t)|_H^2 \leq C \int_{\mathcal{Y}}|v(t,Y)|_U^2 \mu^{\psi(t)}(dY).\]
   Therefore, $h \in L^2([0,T];H)$ and
   \[\int_0^T |h(t)|_H^2 dt \leq C \int_0^T \int_{\mathcal{Y}} |v(t,Y)|_U^2 \mu^{\psi(t)}(dY)dt.\]
   It is standard that $C([0,T];H)$ is dense in $L^2([0,T];H)$. Therefore, we can find a sequence $\{h_n\}_{n \in \mathbb{N}} \subset C([0,T];H)$ such that
   \begin{equation*}
     \lim_{n \to +\infty} \int_0^T |h_n(t)  -h(t)|_H^2 dt =0.
   \end{equation*}
   Let $v_n(t,Y) = \Sigma_1^\star(\psi(t),Y)q^{-1}(\psi(t))h_n(t)$.
   We claim that $v_n(t,Y)$ is continuous in $t$ and Lipschitz continuous in $Y$. Furthermore,
   \begin{equation} \label{Eq:vn-conv}
      \lim_{ n \to +\infty} \int_0^T \sup_{Y \in \mathcal{Y}} |v_n(t,Y) - v(t,Y)|_H^2 dt  =0.
   \end{equation}
   The continuity in $Y$ of $v_n$ is a consequence of Lipschitz continuity of $\Sigma_1$ and the boundedness of $q^{-1}(\psi(t))$. Specifically, for $Y_1, Y_2 \in \mathcal{Y}$,
   \[\sup_{ t \in [0,T]} |v_n(t,Y_1) - v_n(t,Y_2)|_U \leq C \sup_{t \in [0,T]}|h_n(t)|_H|Y_1 - Y_2|_H.\]
   For continuity in $t$, we write for any fixed $n \in \mathbb{N}$ and $Y \in \mathcal{Y}$,
   \begin{align*}
     v_n(t,Y) - v_n(s,Y) = &(\Sigma_1^\star(\psi(t),Y) - \Sigma_1^\star(\psi(s),Y))q^{-1}(\psi(t))h_n(t)\\
     &+ \Sigma_1^\star(\psi(s),Y)(q^{-1}(\psi(t)) - q^{-1}(\psi(s)))h_n(t)\\
     &+ \Sigma_1^\star(\psi(s),Y)q^{-1}(\psi(s))(h_n(t) - h_n(s)).
   \end{align*}
   By Lipschitz continuity and boundedness of $\sigma_1$,
   \[\sup_{Y \in \mathcal{Y}} \|\Sigma_1^\star(\psi(t),Y) - \Sigma_1^\star(\psi(s),Y)\|_{\mathcal{L}(H,U)} \leq C \min\{1,|\psi(t) - \psi(s)|_H\}.  \]
   To understand the continuity of $q^{-1}$, it is helpful to recall that $H = L^2([0,1])$. For any $h \in H$, {$x \in [0,1]$,}
   \[[q(\psi(t))h](x) =\int_{\mathcal{Y}} \sigma_1^2(\psi(t)(x),Y(x))h(x)d\mu^{\psi(t)}(dY). \]
   Then $q^{-1}(\psi)$, must be given by
   \[[q^{-1}(\psi(t))h](x) = \dfrac{h(x)}{\int_{\mathcal{Y}} \sigma_1^2(\psi(t)(x),Y(x))d\mu^{\psi(t)}(dY)}.\]
   In operator norm,
   \[\lim_{s \to t}\|q^{-1}(\psi(t)) - q^{-1}(\psi(s))\|_{\mathcal{L}(H)}=0\]
   due to the continuity of $\psi$ and  $\sigma_1$ and the fact that $\sigma_1$ is bounded from below.
   Because $h_n$ is built to be continuous in time, we see that for any $t \in [0,T]$,
   \[\lim_{s \to t}\sup_{Y \in \mathcal{Y}} |v_n(t,Y) - v_n(s,Y)|_U=0.\]
   We also need to show (\ref{Eq:vn-conv}). This is a simple consequence of the fact that
   \[v_n(t,Y) - v(t,Y) = \Sigma_1(\psi(t),Y)q^{-1}(\psi(t))(h_n(t) - h(t)),\]
   the boundedness of the operators, and the fact that $h_n \to h$ in $L^2([0,T];H)$.

   By Lemma \ref{Lm:UniqueSolvability}, for each $v_n$, there exists $\psi_n$ solving (\ref{Eq:ControlEq}). By Lemma \ref{Lm:ApproachControl},
   \[\frac{1}{2} \int_0^T |v_n(s,Y)|_U^2 \mu^{\psi_n(s)}(dY)ds \to \frac{1}{2} \int_0^t |v(s,Y)|_U^2 \mu^{\psi(s)}(dY)ds\]
   and $\psi_n \to \psi$ in $C([0,T];H)$. Because $h$ is continuous, $h(\psi_n) \to h(\psi)$. We can find $n$ large enough so that
   \[\frac{1}{2} \int_0^T |v_n(s,Y)|_U^2 \mu^{\psi_n(s)}(dY)ds + h(\psi_n)< \inf_{\phi}(S(\phi) + h(\phi)) + \eta.\]
   Set $\tilde{v}=v_n$ and $\tilde{\psi} = \psi_n$.
\end{proof}

Now we use this $\tilde{\psi}$ and $\tilde{v}(t,Y)$ function from Theorem \ref{Theorem:RegularMinimizer} to build approximating stochastic control problems that approximate $\tilde{\psi}$.
Let us introduce the auxiliary process, where $\Delta(\delta)/\delta^2 \to +\infty$, $\Delta(\delta) \to 0$,
\begin{align} \label{Eq:Y-dt-psi}
  dY^{\delta,\tilde\psi}(t) = &\frac{1}{\delta^2} \left(A_2 Y^{\delta,\tilde\psi}(t) + B_2(\tilde\psi([t/\Delta]\Delta), Y^{\delta,\tilde\psi}(t)) \right)dt\\
  &+ \frac{1}{\delta}\Sigma_2(\tilde\psi([t/\Delta]\Delta), Y^{\delta,\tilde\psi}(t))Q_2dW(t) \ .
\end{align}
In the above equation $[\bullet]$ is the floor function.
We will study the pair $(X^{\varepsilon,\delta,u^\delta}, Y^{\varepsilon,\delta,u^\delta})$ where
\[u^\delta(t) = \tilde{v}([t/\Delta]\Delta,Y^{\delta,\tilde\psi}(t)).\]
We will be able to prove the Laplace principle lower bound by showing both that
\[\lim_{\delta \to 0} \mathbf{E} \frac{1}{2} \int_0^T |\tilde{v}([t/\Delta]\Delta,Y^{\delta,\tilde\psi}(t))|_U^2 dt
= \frac{1}{2} \int_0^T \int_{\mathcal{Y}} |\tilde{v}(t,Y)|_U^2 \mu^{\tilde\psi(t)}(dY)dt\]
and
\[X^{\varepsilon,\delta,u^\delta} \to \tilde\psi \text{ in } C([0,T];H) \text{ as } \varepsilon,\delta \rightarrow 0 \ .\]

\begin{lemma}\label{L:ConvergenceOfCost}
  If $Y^{\delta,\tilde\psi}$ solves \eqref{Eq:Y-dt-psi}, then
  \[\lim_{\delta \to 0} \mathbf{E} \frac{1}{2} \int_0^T |\tilde{v}([t/\Delta]\Delta,Y^{\delta,\tilde\psi}(t))|_U^2 dt
   = \frac{1}{2} \int_0^T \int_{\mathcal{Y}} |\tilde{v}(t,Y)|_U^2 \mu^{\tilde\psi(t)}(dY)dt \ . \]
\end{lemma}
\begin{proof}
  For any $\Delta>0$,
  \[\int_{0}^{[T/\Delta]\Delta}  |\tilde{v}([t/\Delta]\Delta,Y^{\delta,\tilde\psi}(t))|_U^2 dt
    = \sum_{k=0}^{[T/\Delta]-1} \int_{k\Delta}^{(k+1)\Delta} |\tilde{v}(k\Delta, Y^{\delta,\tilde\psi}(t)|_U^2 dt \ .\]
  Recall that for $t \in [k\Delta,(k+1)\Delta)$, the process $Y^{\delta,\tilde\psi}$ solves
  \[dY^{\delta,\tilde\psi}(t) = \frac{1}{\delta^2}[A_2 Y^{\delta,\tilde\psi}(t) + B_2(\tilde\psi(k\Delta), Y^{\delta,\tilde\psi}(t))]dt + \frac{1}{\delta} \Sigma_2(\tilde\psi(k\Delta), Y^{\delta,\tilde\psi}(t))dw(t). \]
  Perform the time change $\tilde{Y}^{\delta,\tilde\psi}(t;s) = Y^{\delta,\tilde\psi}(s + \delta^2 t)$. Then $\tilde{Y}^{\delta,\tilde\psi}(t;s)$ solves
  \[d\tilde{Y}^{\delta,\tilde\psi}(t;s) = [A_2 \tilde Y^{\delta,\tilde\psi}(t) + B_2(\tilde\psi(s), \tilde{Y}^{\delta,\tilde\psi}(t;s))]dt + \Sigma_2(\tilde\psi(s), \tilde{Y}^{\delta,\tilde\psi}(t;s))d\tilde{w}(t).\]
  Here $w$ is a cylindrical Wiener process, and $\tilde{w}$ is a cylindrical Wiener process that depends on $\delta$, but we suppress this technicality.
  Then
  \[\int_{k\Delta}^{(k+1)\Delta} |\tilde{v}(k\Delta,Y^{\delta,\tilde\psi}(t))|_U^2dt = \Delta \frac{\delta^2}{\Delta} \int_0^{\frac{\Delta}{\delta^2}} |\tilde{v}(k\Delta, \tilde{Y}^{\delta,\tilde\psi}(t;k\Delta))|_U^2dt.\]

  By Appendix \ref{S:ErgodicProperties}, we have, for any fixed $s>0$, $\tilde{Y}^{\delta,\tilde\psi}(t;s)$ is ergodic with respect to the invariant measure $\mu^{\tilde\psi(s)}$.
  This means that for any $s \in [0,T]$,
  \[\lim_{\delta \to 0}  \frac{\delta^2}{\Delta} \int_0^{\frac{\Delta}{\delta^2}} |\tilde{v}(s, \tilde{Y}^{\delta,\tilde\psi}(t;s))|_U^2dt = \int_\mathcal{Y} |\tilde{v}(s,Y)|_U^2 \mu^{\tilde\psi(s)}(dY).\]
  Therefore, due to the continuity of $\tilde{v}$ in $t$ and Lipschitz continuity in $Y$, and the fact that $[ t/\Delta]\Delta \to t$,
    \begin{align*}
      &\left|\int_{0}^{T} |\tilde{v}([ t/\Delta]\Delta,Y^{\delta,\tilde\psi}(t))|_U^2dt - \int_0^T \int_\mathcal{Y} |\tilde{v}(t,Y)|_U^2 \mu^{\tilde\psi(t)}(dY) \right|\\
      &\leq \Delta \sum_{k=0}^{[T/\Delta]} \left|\frac{\delta^2}{\Delta}\int_{0}^{\frac{\Delta}{\delta^2}} |\tilde{v}(k\Delta,\tilde{Y}^{\delta,\tilde\psi}(t;k\Delta))|_U^2dt -  \int_{\mathcal{Y}}|\tilde{v}(k\Delta,Y)|_U^2 \mu^{\tilde\psi(k\Delta)}(dY) \right|\\
      & \qquad + \int_0^T \left|\int_{\mathcal{Y}} |\tilde{v}(t,Y)|_U^2\mu^{\tilde\psi(t)}(dY)dt - \int_{\mathcal{Y}}|\tilde{v}([t/\Delta]\Delta,Y)|_U^2\mu^{\tilde\psi([t/\Delta]\Delta)}(dY)\right|dt\\
      &\leq \int_0^T \left|\frac{\delta^2}{\Delta}\int_0^{\frac{\Delta}{\delta^2}} |\tilde{v}([s/\Delta]\Delta,\tilde{Y}^{\delta,\tilde\psi}(t;[s/\Delta]\Delta))|_U^2dt - \int_\mathcal{Y} |\tilde{v}([s/\Delta]\Delta,Y)|_U^2 \mu^{\tilde\psi([s/\Delta]\Delta)}(dY) \right|ds\\
      & \qquad + \int_0^T \left|\int_{\mathcal{Y}} |\tilde{v}(t,Y)|_U^2\mu^{\tilde\psi(t)}(dY)dt - \int_{\mathcal{Y}}|\tilde{v}([t/\Delta]\Delta,Y)|_U^2\mu^{\tilde\psi([t/\Delta]\Delta)}(dY)\right|dt \ .
    \end{align*}
    This converges to $0$ by the ergodic theorem (in $Y$), the continuity of $\mu^X$ from Lemma \ref{Lm:mu-continuous}, and the dominated convergence theorem (in $t$) ($\tilde{v}$ was built to be bounded) and the continuity of $\tilde{v}$ in $t$.
\end{proof}

\begin{lemma}
  If $(X^{\varepsilon,\delta,u^\delta},Y^{\varepsilon,\delta,u^\delta})$ solves the stochastic control problem with control
  $u^\delta(t) = \tilde{v}(t,Y^{\delta,\tilde\psi}(t))$, then $X^{\varepsilon,\delta,u^\delta} \to \tilde\psi$ in $C([0,T];H)$ as $\varepsilon\rightarrow 0$ and $\delta\rightarrow 0$.
\end{lemma}
\begin{proof}
First by similar techniques in Lemma \ref{Lm:ClosenessL2Integrated}, Lemma \ref{Lm:ClosenessL2IntegratedFrozenSlow}, and by using
(\ref{Eq:ClosenessL2Integrated}) of Lemma \ref{Lm:ClosenessL2Integrated} we see that we have
to show that there exists $C>0$ and $\varepsilon_0>0$ such that for $0< \varepsilon< \varepsilon_0$ and $T>0$ we have
  \begin{equation} \label{Eq:L2-Y-diff}
    \int_0^T |Y^{\varepsilon,\delta,u^\delta}(t)-Y^{\delta, \tilde\psi}(t)|_H^2 dt \leq C \int_0^T |X^{\varepsilon,\delta,u^\delta}(t) -\tilde\psi(t)|_H^2 dt+ I_0\ ,
  \end{equation}
where the term $I_0$ depends on $\varepsilon$ and $I_0(\varepsilon)\rightarrow 0$ as $\varepsilon\rightarrow 0$. Here and below, we use the constant $C>0$ to represent
a general positive constant that is independent of $\varepsilon$ and $\delta$, but may depend on $t$.

Then, we observe that
\begin{align*}
  X^{\varepsilon,\delta,u^\delta}(t) - \tilde\psi(t) & = \int_0^t S_1(t-s)(B_1(X^{\varepsilon,\delta,u^\delta}(s),Y^{\varepsilon,\delta,u^\delta}(s)) - B_1(\tilde\psi(s),Y^{\delta,\tilde\psi}(s)))ds\\
  & \qquad + \sqrt{\varepsilon} \int_0^t S_1(t-s)\Sigma_1(X^{\varepsilon,\delta,u^\delta}(s),Y^{\varepsilon,\delta,u^\delta}(s))dw(s)\\
  & \qquad + \Bigg(\int_0^t S_1(t-s) \Sigma_1(X^{\varepsilon,\delta,u^\delta}(s),Y^{\varepsilon,\delta,u^\delta}) \tilde{v}([s/\Delta]s,Y^{\delta,\tilde\psi}(s))ds\\
  & \qquad  \ \ \ \ \ -  \int_0^t \int_{\mathcal{Y}} \Sigma_1(\tilde\psi(s),Y) \tilde{v}(s,Y) \mu^{\tilde\psi(s)}(dY)ds \Bigg)\\
  &=:I_1 + I_2 + I_3 \ ,
\end{align*}
where the terms $I_1$, $I_2$ and $I_3$ depend on $\varepsilon$, and $w$ is a cylindrical Wiener process.

By the Lipschitz continuity of $B_1$ and (\ref{Eq:L2-Y-diff}),
\begin{align*}
  |I_1|_H &\leq C\int_0^t  \left(|X^{\varepsilon,\delta,u^\delta}(s)-\tilde\psi(s)|_H  + |Y^{\varepsilon,\delta,u^\delta}(s) - Y^{\delta,\psi}(s)|_H \right)ds\\
  &\leq C \int_0^t  |X^{\varepsilon,\delta,u^\delta}(s)-\tilde\psi(s)|_H ds \ ,
\end{align*}
so that, by H\"{o}lder inequality we have
\begin{align*}
  |I_1|^p_H\leq C \int_0^t  |X^{\varepsilon,\delta,u^\delta}(s)-\tilde\psi(s)|^p_H ds \ .
\end{align*}

Due to the fact that $\Sigma_1$ is bounded, as $\varepsilon\rightarrow 0$, we have $I_2 \rightarrow 0$ since it is multiplied by $\sqrt{\varepsilon}$.

As for the $I_3$ term, it is helpful to rewrite it as
\begin{align*}
  |I_3|_H \leq &\left|\int_0^t S_1(t-s) \bigg(\Sigma_1(X^{\varepsilon,\delta,u^\delta}(s), Y^{\varepsilon,\delta,u^\delta}(s)) - \Sigma_1(\tilde\psi(s), Y^{\delta,\tilde\psi}(s))\bigg) \tilde{v}([s/\Delta]\Delta,Y^{\delta,\tilde\psi}(s))ds \right|_H\\
   &+\Bigg| \int_0^t S_1(t-s) \Sigma_1(\tilde\psi(s), Y^{\delta,\tilde\psi}(s))\tilde{v}([s/\Delta]\Delta, Y^{\delta,\tilde\psi}(s))ds\\
   &\qquad\qquad -  \int_0^t \int_{\mathcal{Y}} S_1(t-s)\Sigma_1(\tilde\psi(s), Y)\tilde{v}(s,Y)\mu^{\tilde\psi(s)}(dY)ds
   \Bigg|_H\\
   =:& |I_{3,1}|_H + |I_{3,2}|_H \ ,
\end{align*}
where the terms $I_{3,1}$ and $I_{3,2}$ both depend on $\varepsilon$.

By arguments similar to Lemma \ref{Lm:Control-terms-estimates} and by (\ref{Eq:L2-Y-diff}), we can show that
\begin{align*}
  |I_{3,1}|_H^p \leq & \ C \left(\int_0^t |\tilde{v}([s/\Delta]\Delta,Y^{\delta,\tilde\psi}(s))|_H^2 ds \right)^{\frac{p}{2}} \times\nonumber\\
  &\times\int_0^t \left(  |X^{\varepsilon,\delta,u^\delta}(s) - \tilde\psi(s)|_H^p ds + |Y^{\varepsilon,\delta,u^\delta}(s) - Y^{\delta,\tilde\psi}(s)|^2 \right) ds\\
  \leq & \ C \int_0^t |X^{\varepsilon,\delta,u^\delta}(s) - \tilde\psi(s)|_H^p ds \ .
\end{align*}

Combining all of these estimates and applying some H\"older inequalities we see that
\[|X^{\varepsilon,\delta,u^\delta}(t) - \tilde\psi(t)|_H^p \leq C \int_0^t |X^{\varepsilon,\delta,u^\delta}(s) -\tilde\psi(s)|_H^pds + I_2 + I_{3,2} \ ,\]
 and using the Gr\"onwall inequality, we conclude that
\begin{align*}
  |X^{\varepsilon,\delta,u^\delta} - \tilde\psi|_{C([0,T];H)} \leq C e^{CT} \left( I_2 + I_{3,2}\right).
\end{align*}
We conclude by showing that $I_{3,2}$ goes to zero by the Ergodic Theorem 3.3.1 of \cite{DaPrato-ZabczykErgodicityBook}, the fact that $\Sigma_1$ is bounded, as well as the Dominated Convergence Theorem.
%
%
\end{proof}

Now we can finish the proof of the lower bound.
By the variational representation and using the specific control $u^\delta (t) = \tilde{v}([t/\Delta]\Delta,Y^{\delta,\tilde\psi})$ and Theorem \ref{Theorem:RegularMinimizer},
\begin{align*}
  \liminf_{\varepsilon \to 0} \varepsilon \ln \mathbf{E} \left[\exp\left(-\frac{h(X^{\varepsilon,\delta})}{\varepsilon} \right) \right]
  &= \liminf_{\varepsilon \to 0}\left(-\inf_{ u \in \mathcal{P}_2^N} \mathbf{E} \left[\frac{1}{2}\int_0^T|u(s)|_U^2 ds + h(X^{\varepsilon,\delta,u}) \right]\right)\\
  &\geq \liminf_{\varepsilon \to 0}\left(-  \mathbf{E} \left[\frac{1}{2}\int_0^T|u^\delta(s)|_U^2 ds + h(X^{\varepsilon,\delta,u^\delta}) \right]\right)\\
  &\geq -\left(\frac{1}{2} \int_0^T \int_{\mathcal{Y}} |\tilde{v}(s,Y)|_U^2 \mu^{\tilde\psi(s)}(dY)ds + h(\tilde\psi)\right)\\
  &\geq -\left(S(\tilde\psi) + h(\tilde\psi) + \frac{2\eta}{3}\right)\\
  &\geq -\left(\inf_{\phi \in C([0,T];H)} \{S(\phi) + h(\phi)\} +\eta\right).
\end{align*}
Because $\eta>0$ was arbitrary, the result is proven.

We conclude this subsection with the proofs of Lemmas \ref{L:Invertible_q}-\ref{Lm:ApproachControl}.
\begin{proof}[Proof of Lemma \ref{L:Invertible_q}]
  First, observe $\Sigma_1^\star(X,Y) = \Sigma_1(X,Y)$ because for any $h,u \in H=U$,
  \[\left<\Sigma_1(X,Y)u, h\right>_H = \int\limits_{[0,1]} \sigma_1(X(x),Y(x))u(x)h(x) dx
  =\left<u, \Sigma_1^\star(X,Y) h \right>_U.\]
  In particular, the fact that $\sigma^{2}_1(x,X,Y)\geq c_0$ for all $x,y \in \mathbb{R}$, implies that for any $h \in H$, $X \in \mathcal{X}$, and $Y \in \mathcal{Y}$,
  \[|[\Sigma_1(X,Y)\Sigma_1^\star(X,Y)h](x)| = \sigma_1^2(X(x),Y(x))|h(x)| \geq c_0 |h(x)|.\]
  Then because $\mu^X(\mathcal{Y})=1$,
  \[\left|\left[\int_\mathcal{Y} \Sigma_1(X,Y)\Sigma_1^\star(X,Y)h\mu^X(dY)\right](x) \right| \geq c_0 |h(x)|.\]
  We can conclude that
  \[|q(X)h|_H \geq c_0 |h|_H.\]
\end{proof}
\begin{proof}[Proof of Lemma \ref{Lm:Lipschitz-control}]
  By the assumption that $S(\psi) <+\infty$ and by the representation \eqref{Eq:ActionFunctOrdinaryForm} and $\mathcal{A}^o_{\psi,T}$, there exists $u: [0,T]\times \mathcal{Y} \to U$ such that
  \[\frac{1}{2} \int_0^T \int_\mathcal{Y} |u(s,Y)|_U^2 \mu^{\psi(s)}(dY)ds \leq  S(\psi) + \eta\]
  and
  \begin{align*}
    \psi(t) = &S_1(t)\psi(0) + \int_0^t \int_\mathcal{Y} S_1(t-s) B_1(\psi(s),Y) \mu^{\psi(s)}(dY) ds \\
    &+ \int_0^t \int_\mathcal{Y} S_1(t-s) \Sigma_1(\psi(s), Y)u(s,Y) \mu^{\psi(s)}(dY)ds.
  \end{align*}
  Of course, there is no guarantee that $u(t,Y)$ is bounded or Lipschitz continuous in $Y$. We only know that it is in $L^2([0,T]\times\mathcal{Y}, \mu^{\psi(s)}(dY)ds;U)$.
  We build ${v}(t,Y) = \Sigma_1^\star(\psi(t), Y)q^{-1}(\psi(t)) a(\psi(t)) u(t,\cdot)$ and claim that this has the desired properties.

  For any $t>0$, $Y\in \mathcal{Y}$,
  \begin{align*}
  |{v}(t,Y)|_U  &\leq \sqrt{c_1} \|q^{-1}(t)\|_{\mathcal{L}(H)}\|a(\psi(t))\|_{\mathcal{L}(L^2(\mathcal{Y},\mu^{\psi(t)};U),H)} |u(t,\cdot)|_{L^2(\mathcal{Y},\mu^{\psi(t)};U)}\nonumber\\
   &\leq \frac{c_1}{c_0}|u(t,\cdot)|_{L^2(\mathcal{Y},\mu^{\psi(t)};U)}.
  \end{align*}
  Therefore, there exists $C>0$ such that for all $t\in[0,T]$
  \[\sup_{Y} |{v}(t,Y)|_U^2dt \leq C  \int_{\mathcal{Y}}|u(t,Z)|_U^2 \mu^{\psi(t)}(dZ).\]
  Similarly, by the Lipschitz continuity of $\Sigma_1^\star$,
  \[|{v}(t,Y_1)-{v}(t,Y_2)|_U \leq \frac{\sqrt{c_1}L^Y_{\sigma_1}|Y_1-Y_2|_H}{c_0}\left(\int_{\mathcal{Y}}|u(t,Z)|_U^2 \mu^{\psi(t)}(dZ)\right)^{\frac{1}{2}}.\]
  We can set
  \[\gamma(t) = C \left(\int_{\mathcal{Y}}|u(t,Z)|_U^2 \mu^{\psi(t)}(dZ)\right)^{\frac{1}{2}}\]
  for an appropriately big constant. This is square integrable  in $t$.

  Additionally, ${v}$ solves the same equations as $u$ because for any $t>0$,
  \begin{align*}
    &a(\psi(t)){v}(t,\cdot) = a(\psi(t))a^\star(\psi(t))q^{-1}(\psi(t))a(\psi(t))u(t,\cdot)\\
    &=q(\psi(t))q^{-1}(\psi(t))a(\psi(t))u(t,\cdot) = a(\psi(t))u(t,\cdot).
  \end{align*}
  Then because $a(\psi(t)){v}(t,\cdot) = \int_\mathcal{Y} \Sigma_1(\psi(t),Y){v}(t,Y)\mu^{\psi(t)}(dY)$,
 (\ref{Eq:v-form}) and (\ref{Eq:ControlEq}) are satisfied.

  Notice that
  \begin{align*}
    &\int_0^T \int_{\mathcal{Y}} |{v}(t,Y)|_U^2 \mu^{\psi(t)}(dY)dt = \int_0^T |a^\star(\psi(t))q^{-1}(\psi(t)) a(\psi(t))u(t,\cdot)|_{L^2(\mathcal{Y},\mu^{\psi(t)};U)}^2 dt\\
    &\leq \int_0^T |u(t,\cdot)|_{L^2(\mathcal{Y},\mu^{\psi(t)};U)}^2 dt = \int_0^T \int_\mathcal{Y} |u(t,Y)|_U^2 \mu^{\psi(t)}(dY)dt \leq S(\psi) + \eta.
  \end{align*}
  The last line is a consequence Lemma \ref{Lm:HilbertSpaceNorms}.
\end{proof}
\begin{proof}[Proof of Lemma \ref{Lm:HilbertSpaceNorms}]
  For any $u \in H_1$,
  \begin{align*}
    & \left| a^\star q^{-1} a u\right|_{H_1}^2  = \left< a^\star q^{-1} a u, a^\star q^{-1} a u\right>_{H_1} = \left<a a^\star q^{-1} a u, q^{-1} a u \right>_{H_2}\\
    & \leq \left< a u, q^{-1} a u \right>_{H_2} = \left< u, a^\star q^{-1} a u \right>_{H_1} \leq |u|_{H_1} |a^\star q^{-1} a u|_{H_1}.
  \end{align*}
  The second line follows from the fact that $a a^\star q^{-1}$ is equal to the identity operator.
  The result follows by dividing both sides by $|a^\star q^{-1} a u|_{H_1}$.
\end{proof}
\begin{proof}[Proof of Lemma \ref{Lm:UniqueSolvability}]
  We prove this using the contraction mapping principle.
  Let $\mathscr{K}: C([0,T];H) \to C([0,T];H)$ be defined by
  \begin{align*}
    \mathscr{K}(\varphi)(t) = &S(t)X_0 + \int_0^t \int_\mathcal{Y} S_1(t-s) B_1(\varphi(s),Y) \mu^{\varphi(s)}(dY)ds\\
     &+ \int_0^t \int_{\mathcal{Y}} S_1(t-s) \Sigma_1(\varphi(s),Y) v(s,Y) \mu^{\varphi(s)}(dY)ds.
  \end{align*}
  For any $\varphi_1, \varphi_2 \in C([0,T];H)$,
  \begin{align*}
    \mathscr{K}(\varphi_1)(t) &- \mathscr{K}(\varphi_2)(t)  =
    \int_0^t \int_\mathcal{Y} S_1(t-s)(B_1(\varphi_1(s),Y) - B(\varphi_2(s),Y))\mu^{\varphi_1(s)}(dY)ds\\
    & + \int_0^t \int_\mathcal{Y} S_1(t-s)B_1(\varphi_2(s),Y)(\mu^{\varphi_1(s)}(dY) - \mu^{\varphi_2(s)}(dY))ds\\
    &+ \int_0^t \int_\mathcal{Y} S_1(t-s)(\Sigma_1(\varphi_1(s),Y) - \Sigma_1(\varphi_2(s),Y))v(s,Y)\mu^{\varphi_1(s)}(dY)ds\\
    &+ \int_0^t \int_\mathcal{Y} S_1(t-s) \Sigma_1(\varphi_2(s),Y)v(s,Y)(\mu^{\varphi_1(s)}(dY) - \mu^{\varphi_2(s)}(dY))ds.
  \end{align*}
  Notice that for any $X \in \mathcal{X}$ and $s \in [0,T]$, $Y \mapsto \Sigma_1(X,Y)v(s,Y)$ is Lipschitz continuous. In particular, for $Y_1,Y_2 \in \mathcal{Y}$, and $s \in [0,T]$, by the Lipschitz continuity and boundedness of $\Sigma_1$,
  \begin{align} \label{Eq:Smv-Lipschitz}
    &\sup_{X \in \mathcal{X}}|\Sigma_1(X,Y_1)v(s,Y_1) - \Sigma_1(X,Y_2)v(s,Y_2)|\nonumber\\
    &\leq \sup_{X \in \mathcal{X}}\left( |(\Sigma_1(X,Y_1) - \Sigma_1(X,Y_2))v(s,Y_1)|_H +  |\Sigma_1(X,Y_2)(v(s,Y_1) - v(s,Y_2))|_H\right)\nonumber\\
    &\leq C\gamma(s)|Y_1-Y_2|_H.
  \end{align}
  By the Lipschitz continuity of $B_1$, $\Sigma_1$, $v$, and the Lipschitz properties of the measures $\mu^X$ (see (\ref{Eq:WeaklyContinuousInvariantMeasure})), and (\ref{Eq:Smv-Lipschitz}), it follows that
  \begin{align*}
    &|\mathscr{K}(\varphi_1)(t) - \mathscr{K}(\varphi_2)(t)|_H
    \leq C\int_0^t \left|\varphi_1(s) - \varphi_2(s)\right|_H ds\\
    &+ C\int_0^t \gamma(s)\left|\varphi_1(s) - \varphi_2(s) \right|_H ds\\
    &\leq C\left(T^{\frac{1}{2}} + \left(\int_0^T \gamma^2(s)ds \right)^{\frac{1}{2}} \right) \left( \int_0^T |\varphi_1(s) - \varphi_2(s)|_H^2 ds\right)^{\frac{1}{2}}.
  \end{align*}
  The last line follows by H\"older inequality. By taking $T_0$ small enough, we can guarantee that $\mathscr{K}$ is a contraction mapping on $C([0,T_0])$. Using standard arguments we can string together solutions until we get a unique fixed point in $C([0,T];H)$.
\end{proof}

\begin{proof}[Proof of Lemma \ref{Lm:ApproachControl}]
  For any $t>0$,
  \begin{align*}
    \psi_n(t) - \psi(t) =
    & \int_0^t \int_{\mathcal{Y}} S_1(t-s) (B_1(\psi_n(s),Y) - B_1(\psi(s),Y))\mu^{\psi_n(s)}(dY)ds\\
    &+ \int_0^t \int_{\mathcal{Y}} S_1(t-s)B_1(\psi(s),Y)(\mu^{\psi_n(s)}(dY) - \mu^{\psi(s)}(dY))ds \\
    &+ \int_0^t \int_{\mathcal{Y}} S_1(t-s) (\Sigma_1(\psi_n(s),Y) - \Sigma_1(\psi(s),Y))v_n(s,Y) \mu^{\psi_n(s)}(dY)ds\\
    &+ \int_0^t \int_{\mathcal{Y}} S_1(t-s)\Sigma_1(\psi(s),Y) v_n(s,Y) (\mu^{\psi_n(s)}(dY) - \mu^{\psi(s)}(dY))ds\\
    &+ \int_0^t \int_{\mathcal{Y}} S_1(t-s) \Sigma_1(\psi(s),Y) (v_n(s,Y) - v(s,Y)) \mu^{\psi(s)}(dY).
  \end{align*}
  By the Lipschitz properties of $B_1$, $\Sigma_1$, $v_n$, $v$, and $\mu^X$, and a few applications of H\"older's inequality, we see that for $t \in [0,T]$,
  \begin{align*}
    |\psi_n(t) - \psi(t)|_H \leq
    & C \int_0^t |\psi_n(s) - \psi(s)|_Hds +  C \left(\int_0^t \gamma_n^2(s)ds \right)^{\frac{1}{2}} \left( \int_0^t |\psi_n(s) - \psi(s)|_H^2 ds\right)^{\frac{1}{2}}\\
    &+ C\sqrt{T} \left(\int_0^t \int_{\mathcal{Y}} |v_n(s,Y) - v(s,Y)|_U^2 \mu^{\psi(s)}(dY)ds \right)^{\frac{1}{2}}.
  \end{align*}
  the fact that $\psi_n \to \psi \in C([0,T];H)$ follows by squaring both sides and applying a Gr\"onwall inequality.

  Finally we show that the energies converge. We claim that $Y \mapsto |v_n(t,Y)|_U^2$ is Lipschitz continuous with Lipschitz constant $2\gamma^2_n(t)$.
  Notice that for $Y_1,Y_2 \in \mathcal{Y}$,
  \begin{align*}
    &|v_n(t,Y_1)|_U^2 - |v(t,Y_2)|_U^2 = \left(|v_n(t,Y_1)|_U + |v_n(t,Y_2)|_U \right) \left(|v_n(t,Y_1)|_U - |v_n(t,Y_2)|_U \right)\\
    &\leq 2 \gamma_n^2(t) |Y_1 - Y_2|_\mathcal{Y}.
  \end{align*}

  Therefore, using the Lipschitz property from Lemma \ref{Lm:mu-continuous},
  \begin{align*}
    &\frac{1}{2}\int_0^T \int_{\mathcal{Y}} |v_n(t,Y)|_U^2 \mu^{\psi_n(t)}(dY)dt - \frac{1}{2}\int_0^T \int_{\mathcal{Y}} |v(t,Y)|_U^2 \mu^{\psi(t)}(dY)dt\\
    &= \frac{1}{2}\int_0^T \int_{\mathcal{Y}} |v_n(t,Y)|_U^2 (\mu^{\psi_n(t)}(dY) - \mu^{\psi(t)}(dY))dt\\
     &\qquad+ \frac{1}{2}\int_0^T \int_{\mathcal{Y}} (|v_n(t,Y)|_U^2 - |v(t,Y)|_U^2) \mu^{\psi(t)}(dY)dt
    \end{align*}
    \begin{align*}
    &\leq  \int_0^T \gamma_n^2(t)|\psi_n(t) - \psi(t)|_Hdt + \frac{1}{2}\int_0^T \int_{\mathcal{Y}} (|v_n(t,Y)|_U^2 - |v(t,Y)|_U^2) \mu^{\psi(t)}(dY)dt\\
    & \leq  \left(\int_0^T \gamma_n^2(t)dt \right)|\psi_n - \psi|_{C([0,T];H)}  + \frac{1}{2}\int_0^T \int_{\mathcal{Y}} (|v_n(t,Y)|_U^2 - |v(t,Y)|_U^2) \mu^{\psi(t)}(dY)dt.
  \end{align*}
  This converges to zero by the assumptions of the lemma and the previously established fact that $\psi_n \to \psi$.
\end{proof}

\subsubsection{Lower bound for the $d\geq 1$ case with $\sigma_{1}(x,X,Y)=\sigma_{1}(x,X)$}\label{SS:UpperBound_MultiDim}
If $\Sigma_1(X,Y) = \Sigma_1(X)$ is independent of $Y$, then the proof of the Laplace principle lower bound is very similar to the standard cases \cite{LDPInfiniteSDE,VariationalInfniteBM}. Let $h: C([0,T];H) \to \mathbb{R}$ be bounded and continuous. Fix $\eta>0$, and let $\psi \in C([0,T];H)$ with $\psi_0 = X_0$ such that
\[S(\psi) + h(\psi) \leq \inf_{\phi \in C([0,T];H)}[S(\phi) + h(\phi)] + \frac{\eta}{2}.\]

There exists a function $v: [0,T]\times \mathcal{Y} \to U$ such that
{\[ \frac{1}{2} \int_0^T \int_\mathcal{Y} |v(s,Y)|_U^2 \mu^{\psi(s)}(dY) ds  \leq S(\psi) + \frac{\eta}{2}\]}
and
{\begin{align*}
  \psi(t) = S_1(t)X_0 &+ \int_0^t \int_\mathcal{Y} S_1(t-s)B_1(\psi(s), Y) \mu^{\psi(s)}(dY) ds \\
  &+ \int_0^t \int_\mathcal{Y} S_1(t-s)\Sigma_1(\psi(s))Q_1v(s,Y)\mu^{\psi(s)}(dY)ds.
\end{align*}}
Define the time dependent  control $u \in L^2([0,T];U)$
\[u(s) = \int_\mathcal{Y} v(s,Y)\mu^{\psi(s)}(dY).\]
Notice that because $\Sigma_1$ is independent of $Y$, $\psi$ solves
\begin{align*}
  \psi(t) = S_1(t)X_0 &+ \int_0^t \int_\mathcal{Y} S_1(t-s)B_1(\psi(s), Y) \mu^{\psi(s)}(dY) ds \\
  &+ \int_0^t S_1(t-s)\Sigma_1(\psi(s))Q_1u(s)ds.
\end{align*}

Consider the sequence of controlled processes $X^{\varepsilon,\delta,u}$ with this control. By Lemma \ref{L:LemmaForODElimit1}, we can show that $X^{\varepsilon,\delta,u}$ converges to $\psi$. {Indeed, by Lemma \ref{Lm:TightnessSlowProcess} we get tightness of the family of processes $\{X^{\varepsilon,\delta,u}: \varepsilon \in (0,1)\}$  in $C([0,T]; H)$. Then, Lemma \ref{L:LemmaForODElimit1} shows that $X^{\varepsilon,\delta,u}\rightarrow\bar{X}$ in distribution, where
\begin{align*}
  \bar{X}(t) = S_1(t)X_0 &+ \int_0^t \int_\mathcal{Y} S_1(t-s)B_1(\bar{X}(s), Y) \mu^{\bar{X}(s)}(dY) ds \\
  &+ \int_0^t S_1(t-s)\Sigma_1(\bar{X}(s))Q_1u(s)ds.
\end{align*}
Uniqueness of this equation shows that $\bar{X}(t)=\psi(t)$ for every $t\in[0,T]$ with probability one. Then
\begin{align*}
  \liminf_{\varepsilon \to 0} \varepsilon \ln \mathbf{E} \left[\exp\left(-\frac{h(X^{\varepsilon,\delta})}{\varepsilon} \right) \right]
  &= \liminf_{\varepsilon \to 0}\left(-\inf_{ u \in L^{2}([0,T];U)} \mathbf{E} \left[\frac{1}{2}\int_0^T|u(s)|_U^2 ds + h(X^{\varepsilon,\delta,u}) \right]\right)\\
  &\geq \liminf_{\varepsilon \to 0} \left(- \mathbf{E} \left[\frac{1}{2}\int_0^T|u(s)|_U^2 ds + h(X^{\varepsilon,\delta,u}) \right]\right)\\
  &= -\left(\frac{1}{2} \int_0^T \int_{\mathcal{Y}} |v(s,Y)|_U^2 \mu^{\psi(s)}(dY)ds + h(\psi)\right)\\
  &\geq -\left(S(\psi) + h(\psi) + \frac{\eta}{2}\right)\\
  &\geq -\left(\inf_{\phi \in C([0,T];H)} \{S(\phi) + h(\phi)\} +\eta\right).
\end{align*}}
Because $\eta>0$ was arbitrary, the result is proven.
Notice that this method cannot work if $\Sigma_1$ depends on $Y$.

\subsection{Compactness of level sets}\label{SS:CompactLevelSets}
We want to prove that for each $s<\infty$ and for any $X_0 \in H$, the set
$$\Phi_{s,X_0}=\{\phi\in C([0,T]; H): S_{X_{0}}(\phi)\leq s\}$$
is a compact subset of $C([0,T]; H)$. 
\begin{lemma}\label{Lm:PreCompactnessLevelSet}
Fix $K<\infty$, $\theta>0$ and consider any sequence $\{(\phi^n, \mathrm{P}^n) \ , \ n>0\}$ such that
 for every $n>0$ $(\phi^n, \mathrm{P}^n)\in \mathcal{V}_{(\xi, \mathcal{L})} $ is viable and
$$\int_{U\times \mathcal{Y}\times [0,T]}\left(|u|_U^2+|Y|_{\theta,2}^{2}\right)\mathrm{P}^{n}(dudYdt)<K \ .$$

Then $\{(\phi^n, \mathrm{P}^n) \ , \ n>0 \}$ is pre-compact.
\end{lemma}

\begin{proof}
Noticing that the third marginal of $\mathrm{P}^n(dudYdt)$ is the Lebesgue measure,  a similar, but easier technically, argument as in Lemma \ref{Lm:LpEstimateSlowAndFastProcess} establishes that for $p=2/\zeta$
\[
\sup_{t\in[0,T],n\in\mathbb{N}}|\phi^{n}(t)|_{H}^{p}\leq c_{T,K} (1+|X_{0}|_{H}^{p})
\]

The last display and the fact that for any $t,h\geq 0$ such that $t,t+h\in[0,T]$ we have
\begin{align*}
\phi(t+h)-\phi(t)&=(S_1(h)-I)\phi(t)+\int_{U\times \mathcal{Y}\times [t,t+h]}S_1(t+h-s)\xi(\phi(s), Y, u)\mathrm{P}(dudYds)
\end{align*}
imply via an argument similar to the proof of Lemma \ref{Lm:OscillationEstimateSlowProcess} that there exists $0<\bar{\theta}<\frac{1-\zeta}{2}$ and $p=\frac{2}{\zeta}$,
such that for any  $T>0$, $X_0\in H$,  it holds
\begin{equation*}
\sup\limits_{n\in\mathbb{N}}|\phi^{n}(t+h)-\phi^{n}(t)|_H^p\leq c_{\theta,p,T,K}\left(h^{\beta(\theta)p}(|X_0|_{H}^p+1)+\left|(S_{1}(h)-I)X_{0}\right|_{H}^{p}\right)
\end{equation*}
for some positive constants $c_{\theta,p,T,K}$ and $\beta(\theta)>0$.

Pre-compactness of $\{\phi^{n},n>0\}$ then follows. Pre-compactness of  $\{\mathrm{P}^n \ , \ n>0 \}$ follows exactly as in the proof of Lemma \ref{Lm:TightnessOccupationMeasure}, concluding the proof of the lemma.
\end{proof}

Next, we prove that limit of a viable pair is also viable.
\begin{lemma}\label{Lm:LimitOfViablePairStillViable}
For $K<\infty$ and $\theta>0$, consider any convergent sequence $\{(\phi^n, \mathrm{P}^n), n>0\}$,
such that for  every $n>0$, $(\phi^n, \mathrm{P}^n)\in\mathcal{V}_{(\xi, \mathcal{L})}$ is viable and{
$$\int_{U\times \mathcal{Y}\times [0,T]}\left(|u|_U^2+|Y|_{\theta,2}^{2}\right)\mathrm{P}^n(dudYdt)<K \ .$$}
Then the limit $(\phi, \mathrm{P})\in\mathcal{V}_{(\xi, \mathcal{L})}$, i.e. it is a viable pair.
\end{lemma}
\begin{proof} Since $(\phi^n, \mathrm{P}^n)\in\mathcal{V}_{(\xi, \mathcal{L})}$  we have
\begin{equation}\label{Eq:SequenceViablePairAveragingEquation}
\phi^n(t)=S_1(t)X_0+\int_{U\times \mathcal{Y}\times [0,t]} S_1(t-s)\xi(\phi^n(s), Y, u)\mathrm{P}^n(dudYds) \ ,
\end{equation}
and
\begin{equation}\label{Eq:SequenceViablePairInvariantMeasure}
\mathrm{P}^n(dudYds)=\eta^{n}(du|Y,s)\mu^{\psi^{n}(s)}(dY)ds.
\end{equation}

By Fatou's lemma we can show that $\mathrm{P}$ satisfies
\[
\int_{U\times \mathcal{Y}\times [0,T]}\left(|u|_U^2+|Y|_{\theta,2}^{2}\right)\mathrm{P}(dudYdt)<\infty
\]

Now, observe that the function $\xi(X,Y,u)$ is continuous in $X$ and $Y$, grows at most sublinearly in $Y$
and is affine in $u$. In addition, one can prove a uniform integrability lemma for $\mathrm{P}^n$ analogously to Lemma \ref{Lm:UniformIntegrabilityOccupationMeasure}. Hence, since by the assumption of this lemma we know that
$(\phi^n,\mathrm{P}^n)\rightarrow (\phi, \mathrm{P})$ and thus that $(\phi, \mathrm{P})$
also satisfy equation (\ref{Eq:SequenceViablePairAveragingEquation}) with $(\phi^n, \mathrm{P}^n)$
replaced by $(\phi, \mathrm{P})$.

Next we show that (\ref{Eq:SequenceViablePairInvariantMeasure}) holds with $(\phi^n, \mathrm{P}^n)$ replaced by $(\phi, \mathrm{P})$. Essentially it is enough to show that the second marginal of $\mathrm{P}(dudYds)$ will be
$\mu^{\phi(s)}(dY)$. This follows, by the fact that for any $X$, $\mu^{X}(dY)$ is Lipschitz weakly continuous with respect to X, which due to Hypothesis 2 follows as in the proof of Lemma \ref{Lm:ViablePairInvMeas}.

Finally, it follows from $\mathrm{P}^n(U\times \mathcal{Y}\times [0,t])=t$ and $\mathrm{P}(U\times \mathcal{Y}\times \{t\})=0$ that
$\mathrm{P}(U\times \mathcal{Y}\times [0,t])=t$ for all $t\in [0,T]$.
\end{proof}

We finally have  that for each $X_{0}\in H$, the action functional $S_{X_{0}}(\phi)$ is lower semicontinuous. The proof of this lemma is omitted as it follows from Lemmas \ref{Lm:PreCompactnessLevelSet} and \ref{Lm:LimitOfViablePairStillViable} in a standard way.
\begin{lemma}\label{Lm:LowerSemicontinuityActionFunctional}
For every $X_{0}\in H$, the map $\phi\mapsto S_{X_{0}}(\phi)$ is a lower semicontinuous map from $C\left([0,T];H\right)$ to $[0,\infty)$.
\end{lemma}

\section{Remarks and Generalizations}\label{S:Generalizations}

In this section we comment on the obstacles that one faces when trying to extend the proof of the Laplace principle lower bound to $d>1$ under the general Hypothesis 3. In addition, we comment on the possibility of considering scaling regimes different from the one considered in this paper, i.e., different than $\delta/\sqrt{\varepsilon}\rightarrow 0$.
\subsection{Difficulties for proving Laplace principle lower bound for $d>1$}
In order to prove the Laplace principle lower bound, we need to construct a nearly optimal control that achieves the lower bound.
Under the general Hypothesis 3 (under which we can prove averaging Theorem \ref{T:MainTheorem1}), in dimension $d>1$, $\sigma_{1}(x,X,Y)$ depends on both $X$ and $Y$ and the nearly optimal control $v(t,Y)$ will be a true feedback form control with respect to $Y$ (see the discussion in Subsection \ref{SS:LDPUpperBound}). The generalization of (\ref{Eq:v-form}) to $d>1$ now takes the form,
\[v(t,Y) = Q_1^\star \Sigma_1^\star(\psi(t),Y)  q^{-1}(\psi(t)) a(\psi(t))v(t,\cdot) \]
with
\[
a(X)u=\int_{\mathcal{Y}}\Sigma_{1}(X,Y)Q_{1}u(Y)\mu^{X}(dY)
\]
and
\[
q(X)h =a(X)a^\star(X)h = \int_\mathcal{Y}  \Sigma_1(X,Y)Q_{1}Q_1^\star\Sigma_1^\star(X,Y) h \mu^X(dY).
\]

Notice now that the covariance matrix $Q_{1}$ enters the calculations and recall that for $d>1$ it needs to have decaying eigenvalues. But in the formula for $v(t,Y)$, the inverse operator $q^{-1}(X)$ appears which now is an unbounded operator. The issue of unboundedness of $q^{-1}(X)$ complicates the subsequent mathematical analysis of Subsection \ref{SS:UpperBound_OneDim} significantly. For instance, the statement of Lemma \ref{L:ConvergenceOfCost} would not be necessarily true anymore, or at least a different non-obvious argument is needed.

We believe that this is a technical issue that one should be able to overcome. However, despite our best efforts, we had not been able to do so.
\subsection{Generalization to other regimes}
In this paper, we analyzed the regime $\delta/\sqrt{\varepsilon}\downarrow 0$. One can of course ask what is the behavior in all possible interaction regimes
\begin{equation*}
  \lim_{\varepsilon \downarrow 0} \frac{\delta}{\sqrt{\varepsilon}} =
   \begin{cases}
     0, & \text{Regime 1}, \\
     \gamma \in (0, \infty), & \text{Regime 2}, \\
     \infty, & \text{Regime 3}.
   \end{cases}
\end{equation*}

Regime 1, that we studied in this paper, allows to decouple the invariant measure and the control from the limiting occupation measures $\mathrm{P}$. Namely, it allows us to write
\[
\mathrm{P}(dudYdt)=\eta_t(du|Y)\mu^{\psi_t}(dY)dt
\]
and what is important is that the measure $\mu$ does not depend on the control variable $u$. However, it is easy to see that in the cases of Regimes 2 and 3, one would have
\[
\mathrm{P}(dudYdt)=\eta_t(du|Y)\mu^{\psi_t}(dY|u)dt
\]
which means that in these cases $\mu$ depends on the control variable $u$. This dependence on $u$ makes the analysis considerably more complicated and in particular there is no guarantee that $\mu$ is invariant measure to some process, as this process is a controlled process in which case one needs to know regularity properties of the optimal controls.

This program was carried out in the finite dimensional case with periodic coefficients, in \cite{LDPWeakConvergence}, using the characterization of optimal controls through solutions to Hamilton--Jacobi--Bellman equations. Such a characterization is not rigorously known in infinite dimensions and even if that becomes the case, one would need to establish sufficient regularity properties of such equations that would then imply that the resulting controlled process has a well defined invariant measure that is regular enough.

\appendix

\section{Ergodic and mixing properties of the fast process $Y^{X, Y_0}$}\label{S:ErgodicProperties}

Let us start with reviewing some basic ergodic and mixing properties
of the fast process $Y^{X,Y_0}$. We show the exponential ergodicity of the fast transition semigroup defined by (\ref{Eq:FastProcessSRDEReg1}). For more details we refer the interested reader to \cite{CerraiRDE2006,CerraiRDEAveraging1}.

Under Hypotheses 1 and 2, for any $T>0$ and $p\geq 1$, and any fixed frozen slow variable $X\in H$ and initial condition $Y_0\in H$,
such a problem admits a unique mild solution $Y^{X, Y_0}\in \mathcal{C}_{T, p}$ (\cite[Theorem 5.3.1]{DaPrato-ZabczykErgodicityBook}). As it is proven in Theorem 7.3 of \cite{CerraiRDE2006},  there
exists some $\delta_1>0$ such that for any $p\geq 1$ we have
\begin{equation*}
\mathbf{E}|Y^{X, Y_0}(t)|_H^p\leq c_p (1+|X|_H^p+e^{-\delta_1 p t}|Y_0|_H^p) \ , \ t\geq 0 \ .
\end{equation*}

In addition, the latter statement  implies that there exists some $\theta>0$ such that for any $a>0$ we have
\begin{equation}\label{Eq:NegativeSobolevEstimateFastProcessReg1}
\sup\limits_{t\geq a}\mathbf{E}|Y^{X, Y_0}(t)|_{H^\theta_{2}}\leq c_a (1+|X|_H+|Y_0|_H) \ .
\end{equation}

Due to (\ref{Eq:NegativeSobolevEstimateFastProcessReg1}), the family $\{\mathcal{L}(Y^{X,Y_0}(t))\}_{t\geq 0}$
is tight in the space $\mathcal{P}(H, \mathcal{B}(H))$ and thus by Krylov--Bogoliubov theorem there exists
an invariant measure $\mu^X$ for the semigroup $P_t^X$ generated by the process
$Y^{X, Y_0}(t)$. Moreover, by Lemma 3.4 of \cite{Cerrai-FreidlinRDEAveraging1} we  have
\begin{equation}\label{Eq:LpEstimateInvariantMeasureFastProcessReg1}
\int_{H}|Y|_H^p \mu^X(dY)\leq c_p(1+|X|_H^p) \ .
\end{equation}

As in \cite[Theorem 7.4]{CerraiRDE2006}, it is possible to show that if $\lambda$
is sufficiently large and/or $L_{b_2}^Y$, $L_{\sigma_2}^Y$, $\zeta_2$ and $\kappa_2$ are sufficiently
small, then there exist some $c, \delta_2>0$ such that
\begin{equation}\label{Eq:CouplingEstimateFastProcessReg1}
\sup\limits_{X\in H}\mathbf{E}|Y^{X, Y_1}(t)-Y^{X, Y_2}(t)|_H\leq ce^{-\delta_2 t}|Y_1-Y_2|_H \ , \ t \geq 0 \ ,
\end{equation}
for any $Y_1, Y_2\in H$. In particular, this implies that $\mu^X$ is the unique invariant measure
for $P_t^X$ and is strongly mixing. By arguing as in \cite[Theorem 3.5 and Remark 3.6]{Cerrai-FreidlinRDEAveraging1},
from (\ref{Eq:LpEstimateInvariantMeasureFastProcessReg1}) and (\ref{Eq:CouplingEstimateFastProcessReg1}), we have, for some $\delta>0$,
\footnote{In fact $\delta=\frac{\lambda-L_{b_2}^Y}{2}>0$ by Hypothesis 2 part 2.}

\begin{equation*}
\left|P_t^X \varphi(Y_0)-\int_H \varphi(Y) \mu^X(dY)\right|\leq c(1+|X|_H+|Y_0|_H)e^{-\delta t}[\varphi]_{\text{Lip}(H)}
\end{equation*}
for any $X,Y_0\in H$ and $\varphi\in \text{Lip}(H)$, and
\begin{equation*}
\left|P_t^X \varphi(Y_0)-\int_H \varphi(Y) \mu^X(dY)\right|\leq c(1+|X|_H+|Y_0|_H)e^{-\delta t}(t\wedge 1)^{-\frac{1}{2}}|\varphi|_0
\end{equation*}
for any $X,Y_0\in H$ and $\varphi\in B_b(H)$.

As in \cite[Lemma 2.3]{CerraiRDEAveraging1}, we have the following lemma.
\begin{lemma}\label{Lm:ErgodicTheoremFastProcessReg1}
Under the above conditions, for any $\varphi\in \text{Lip}(H)$, $T>0$, $X, Y_0\in H$
and $t\geq 0$ we have
\begin{equation*}
\mathbf{E}\left|\dfrac{1}{T}\int_t^{t+T}\varphi(Y^{X, Y_0}(s))ds-\int_H \varphi(Y)\mu^X(dY)\right|\leq \dfrac{c}{\sqrt{T}}(H_{\varphi}(X,Y_0)+|\varphi(0)|) \ ,
\end{equation*}
for some $c>0$, where $$H_{\varphi}(X, Y_0):=[\varphi]_{\text{Lip}(H)}(1+|X|_H+|Y_0|_H) \ .$$
\end{lemma}

\end{document}